\documentclass[a4paper,english,12pt]{article}
\usepackage[margin=20mm]{geometry}

\usepackage{inputenc}
\usepackage{amsmath,amssymb,amsthm}
\usepackage{xcolor}
\usepackage{color}
\usepackage{babel}
\usepackage{graphicx} 

\usepackage{mathtools}
\mathtoolsset{showmanualtags=true}
\usepackage{enumerate}
\usepackage{enumitem}
\usepackage{mathrsfs,cancel,comment}
\usepackage{stmaryrd}

\usepackage{physics} 
\usepackage[style=alphabetic,backend=biber, giveninits, maxnames=6, maxalphanames=4,url=false,isbn=false]{biblatex}
\numberwithin{equation}{section}

\usepackage{thmtools}
\usepackage[colorlinks=true,linkcolor=blue,citecolor=blue!50!cyan,urlcolor=teal]{hyperref}

\usepackage{cleveref}

\newtheorem{theorem}{Theorem}[section]

\newtheorem{proposition}[theorem]{Proposition}
\newtheorem{lemma}[theorem]{Lemma}
\newtheorem{corollary}[theorem]{Corollary}
\newtheorem{definition}[theorem]{Definition}
\newtheorem{remark}[theorem]{Remark}

\declaretheorem[name={A}ssertion,numberwithin=section]{thm}

\allowdisplaybreaks[2]

\newcommand{\supp}{\operatorname{supp}} 
\newcommand{\divv}{\operatorname{div}} 
 
\newcommand{\diag}{\operatorname{diag}}

\newcommand{\av}{-\hspace{-2.4ex}\int}

\newcommand{\Om}{\Omega}
\newcommand{\ve}{\varepsilon}
\newcommand{\om}{\omega}
\newcommand{\vp}{\varphi}
\newcommand{\ol}{\overline}
\newcommand{\ul}{\underline}

\newcommand{\kap}{\kappa}

\newcommand{\ee}{\mathrm{e}}

\newcommand{\vlo}{\zeta}
\newcommand{\npo}{\nu_{\pa\Om}}
\newcommand{\iv}{I}
\newcommand{\param}{{[\alpha]}}

\newcommand{\sff}{\mathsf{f}}
\newcommand{\sfg}{\mathsf{g}}
\newcommand{\sfV}{\mathsf{V}}
\newcommand{\sfs}{\mathsf{s}}
\newcommand{\sfI}{\mathsf{I}}

\newcommand{\errfi}{\frac{1}{2i+1}}

\newcommand{\nablaG}{\nabla_{\gamma}}
\newcommand{\DeltaG}{\Delta_\gamma}
\newcommand{\divG}{\divv_\gamma}
\newcommand{\nablaGeps}{\nabla_{\gamma_{\ve\rho}}}
\newcommand{\DeltaGeps}{\Delta_{\gamma_{\ve\rho}}}
\newcommand{\divGeps}{\divv_{\gamma_{\ve\rho}}}
\newcommand{\dimx}{\mathsf{d}}
\newcommand{\bbU}{[U^0]_-^+}
\newcommand{\mrem}{\boldsymbol{r}}
\newcommand{\Hd}{H^1(\Gamma)}

\newcommand{\wh}{\widehat}

\newcommand{\la}{\langle}
\newcommand{\ra}{\rangle}
\newcommand{\pa}{\partial}

\newcommand{\boldu}{\boldsymbol{y}}
\newcommand{\hms}{\mathcal{H}^{d-1}}

\newcommand{\Mb}{p}
\newcommand{\eps}{{\epsilon}}
\newcommand{\proj}{\mathfrak{p}}
\newcommand{\RD}{\mathfrak{R}}
\newcommand{\SL}{\mathfrak{S}}

\newcommand{\sfa}{{\sf a}}
\newcommand{\sfm}{{\sf m}}
\newcommand{\sfh}{{\sf h}}

\DeclareMathSymbol{\mlq}{\mathord}{operators}{``}
\DeclareMathSymbol{\mrq}{\mathord}{operators}{`'}

\title{Interface dynamics in a degenerate Cahn--Hilliard model for viscoelastic phase separation}
\author{Katharina Hopf\thanks{Weierstrass Institute for Applied Analysis and Stochastics, Berlin, Germany} , John King\thanks{University of Nottingham, UK} , Andreas M\"unch\thanks{University of Oxford, UK} , Barbara Wagner$^*$}
\date{}

\providecommand{\keywords}[1]
{\small\text{\textit{Keywords and phrases: }} #1
}

\begin{document}
	
	\maketitle
	
	\begin{abstract}\noindent		
		The formal sharp-interface asymptotics in a degenerate Cahn--Hilliard model for viscoelastic phase separation with cross-diffusive coupling to a bulk stress variable are shown to lead to non-local lower-order counterparts of the classical surface diffusion flow.
		The diffuse-interface model is a variant of the Zhou--Zhang--E model and has an Onsager gradient-flow structure with a rank-deficient mobility matrix reflecting the ODE character of stress relaxation. 
		In the case of constant coupling, we find that the evolution of the zero level set of the order parameter approximates the	so-called intermediate surface diffusion flow.  For non-constant coupling functions monotonically connecting the two phases, our asymptotic analysis leads to a new family of third-order evolution laws with associated propagation operators behaving, at leading order, like the square root of the minus Laplace–Beltrami operator.
		 In this case, the normal velocity of the moving sharp interface arises as the Lagrange multiplier in a constrained elliptic equation, which is at the core of our derivation. 	The constrained elliptic problem can be solved rigorously by a variational argument, and is shown to encode the gradient structure of the effective geometric evolution law.
		
		The asymptotics are presented for \textit{deep quench}, an intermediate free boundary problem based on the double-obstacle potential. 
	\end{abstract}
	
		\noindent\keywords{sharp-interface asymptotics,  intermediate surface diffusion flow,  fractional surface diffusion, degenerate Cahn--Hilliard equation, cross-diffusion, Onsager structure, rank-deficient mobility matrix, Lagrange multiplier.}

\section{Introduction}\label{sec:intro}

Phase separation occurs widely in multi-component systems involving immiscible or partially miscible constituents including melted alloys quenched to low temperature, complex fluids like emulsions, and  biological materials. It applies to situations where parameters are altered in such a way that a material's composition close to one of the pure phases is energetically favourable.
 In the early stage of phase separation, a mixture is often seen to undergo spinodal decomposition, leading to the formation of small droplets corresponding to an energetically preferred volume fraction. This process is primarily driven by a reduction in bulk free energy.
 At a later stage, when the mixture has already decomposed into distinct phases, decrease of interfacial energy, surface diffusion effects, and coarsening are key characteristics of the evolution. If a material's constituents have different mechanical properties, the internal time scales dictating the unmixing process may differ between species, inducing a dynamic asymmetry in the system~\cite{tanaka2000viscoelastic}. 
 Dynamically asymmetric materials can display complex transient morphologies during phase separation, including the early stages of coarsening. 
The presence of multiple time scales is frequently observed in polymer solutions due to the longer relaxation time of the polymer component.
Given the inherent viscoelastic effects, phase separation in polymer solutions is modelled by viscoelastic phase separation (VPS). It is thought to play a significant role in cell biology~\cite{Tanaka_2022_vps} due to its ability to exhibit transient patterns like volume shrinking and phase inversion.
Early two-fluid models for VPS in a binary mixture able to reproduce these phenomena were proposed in~\cite{doi1992dynamic,TO_1996_network}, and developed further by Tanaka et al., see~\cite{tanaka2000viscoelastic} and references therein. A shortcoming of these models is their lack of thermodynamic consistency accompanied by 
deficiencies in the stability properties of numerical approximation schemes. 
The first two-fluid model for VPS consistent with the second law of thermodynamics was derived by Zhou, Zhang, and E~\cite{ZZE_2006} based on ideas from non-equilibrium thermodynamics.
The detailed fluid model involves both reversible and irreversible processes, and consists of a degenerate Cahn--Hilliard equation coupled to a viscoelastic version of the incompressible Navier--Stokes equations involving the momentum equation and a tensorial equation describing stress relaxation. The global existence of weak solutions for a regularised version of this model with stress diffusion was established by Brunk and  Luk\'a\v{c}ov\'a-Medvid'ov{\'a}~\cite{BL_2022_vps_degmob}.

 In this article, we focus on a purely dissipative variant, proposed in~\cite{ZZE_2006} as a simplification of the original fluid model in a regime where hydrodynamic transport can be neglected. In this simplified description, the barycentric velocity of the mixture vanishes and 
the effects of viscoelasticity are encoded in an ODE-like equation for the spherical part of the stress tensor.
 Specifically, the evolving state is modelled by two scalar quantities, the difference in volume fraction $u(t,x)\in[-1,1]$ of the two components and an extra variable $q(t,x)\in \mathbb{R}$ accounting for spherical stress. The equations are posed in a bounded smooth domain $\Om\subset\mathbb{R}^\dimx$, $\dimx\ge2$, and take the form
 	\begin{subequations}\label{eq:Zhou}
 		 \begin{align}
 				\label{eq:Zhou.u}
 				&\partial_tu=-\divv\big(m(u)\,\boldsymbol{j}\big),
 				\qquad \boldsymbol{j}=-\Big[\nabla\tfrac{\delta F}{\delta u}-\frac{1}{(1-u^2)}\nabla (A(u)q)\Big],
 				\quad&& t>0,x\in\Om,
 				\\&\partial_tq=-\frac{1}{\tau(u)}q
 				+A(u)\divv\Big(\frac{m(u)}{(1-u^2)}\boldsymbol{j}\Big),
 				&& t>0,x\in\Om,
 				\label{eq:Zhou.q}
 		\end{align}
with $m(u)=(1-u^2)^2\tilde m(u)$, where $\tilde m,A,\tau\in C^\infty(\mathbb{R},\mathbb{R}_+)$ and $\inf_{\mathbb{R}} \tilde m>0$. The function $A$ denotes the bulk modulus and $\tau$ the relaxation time. For polymer solutions, we typically have $A(-1)<A(+1)$ and $\tau(-1)<\tau(+1)$ whenever $\{u=-1\}$ describes the pure solvent and $\{u=1\}$ the polymer phase.
Equations~\eqref{eq:Zhou.u}--\eqref{eq:Zhou.q} are supplemented by no-flux type and homogeneous Neumann boundary conditions
\begin{align}\label{eq:Zhou.bc}
m(u)\boldsymbol{j}\cdot\npo=0,\qquad\nabla u\cdot\npo=0,\;\qquad t>0,x\in\pa\Om,
\end{align}
where $\npo$ denotes the outer unit normal field to $\Om$.

The driving free energy underlying system~\eqref{eq:Zhou} is given by 
$H(u,q)=F(u)+\int_\Om\frac{q^2}{2}\,\dd x$,
where, in~\cite{ZZE_2006}, $F$ is chosen to be the logarithmic Cahn--Hilliard free energy 
 \begin{align}\label{eq:log.pot}
 F(u)=\int_\Om\big(\frac{\ve^2}{2}|\nabla u|^2+f(u)\big)\,\dd x,\qquad
 f(u)=\tfrac{1}{2}\theta\big(\lambda(1{+}u)+\lambda(1{-}u)\big)+\tfrac{1}{2}(1-u^2),& \\\text{with }\lambda(s)=s\log s,\quad\theta>0.&\nonumber
 \end{align}\end{subequations}
Here, $0<\ve\ll1$ denotes the interface thickness parameter, while $\theta$ describes the fixed temperature of this isothermal model.
 The quadratic part $\frac12\|q\|_{L^2}^2$ of $H$ can be seen as a penalty term for polymeric displacements. 
 Notice that for $A\equiv0$, equation~\eqref{eq:Zhou.u} reduces to a variant of the Cahn--Hilliard equation with `doubly' degenerate mobility in the sense that the mobility function $m(u)$ 
 vanishes quadratically rather than linearly in each of the pure phases $\{u=\pm1\}$, while equation~\eqref{eq:Zhou.q} turns into an ordinary differential equation describing the relaxation of bulk stress to the equilibrium state $q\equiv0$ at an exponential rate with decay constant $\frac{1}{\tau(u)}$.
 If  $A\not\equiv0$, the second-order term arising on the right-hand side of~\eqref{eq:Zhou.q} is needed to ensure the thermodynamic structure of the PDE system. Notice that for $A\not\equiv0$ the system~\eqref{eq:Zhou} is strongly coupled of cross-diffusion type. Further note that the diffusive fluxes in~\eqref{eq:Zhou.u} and~\eqref{eq:Zhou.q} are linearly dependent, so that the system~\eqref{eq:Zhou} cannot be fully parabolic.  
 Let us also mention that, numerically, the simplified model~\eqref{eq:Zhou} is still able to capture the phenomena of volume shrinking and phase inversion, cf.~\cite{ZZE_2006,STDL_2019_energy-stable}.
 
In the present work, we wish to investigate the late-stage evolution of a class of degenerate Cahn--Hilliard models for VPS motivated by~\eqref{eq:Zhou} in the limit of vanishing interface thickness. 
The late-stage evolution represents the most stable regime of the dynamics, beyond the equilibrium problem. Therefore, it is a natural starting point when trying to understand the geometric properties underlying the dynamics of VPS. 
Our goal is to formally identify the geometric flow that governs the evolution, once distinct interfaces have formed.
 A particular interest lies in understanding the effect of the cross-diffusive coupling and the linear dependence of the diffusion fluxes on the asymptotic analysis and the resulting effective interface evolution law.
 
 \paragraph{Interface dynamics in degenerate Cahn--Hilliard equations.} 
 The first work relating a Cahn--Hilliard model to a sharp-interface evolution law is due to Pego~\cite{Pego_1989}. 
 He considered the Cahn--Hilliard equation with constant mobility and a smooth double-well potential, and studied the formal asymptotics for vanishing interface width $\ve\downarrow0$ along different time scalings. 
 Most notably, he showed that, on the slow time scale $t\mapsto \ve t$, the motion of the limiting interface agrees with the Mullins--Sekerka flow. This finding was made rigorous by Alikakos, Bates, and Chen~\cite{ABC_1994} for sufficiently smooth solutions. 
 Cahn, Elliott, and Novick--Cohen~\cite{CEN_1996} were the first to perform the sharp-interface asymptotics for the Cahn--Hilliard equation with degenerate mobility. They studied the physically well-grounded case with the free energy~\eqref{eq:log.pot} involving a logarithmic singular potential and the linearly degenerate mobility $m(u)=1-u^2$ on the time scale $t\mapsto \ve^2 t$ and with vanishing temperature $\theta=O(\ve^\alpha)$, $\alpha>0$, obtaining the surface diffusion flow
 \begin{align}\label{eq:sd}
 	\sfV_\Gamma=-\frac{\sigma}{\delta}\,\Delta_\Gamma \kappa_\Gamma,
 \end{align}
where $\frac{\sigma}{\delta}=\tfrac{\pi^2}{16}>0$ (cf.~\eqref{eq:sig-del.u}), 
as the geometric law governing, at leading order, the interface evolution.
 In~\eqref{eq:sd}, $\sfV_\Gamma$ denotes the scalar normal velocity and $\kappa_\Gamma$ the mean curvature of the moving interface $\Gamma=\cup_{t\in\iv}\{t\}{\times}\Gamma(t)$ (for details, see \Cref{sec:prelim}), while $\Delta_\Gamma$ denotes the Laplace--Beltrami operator. 
 The authors of~\cite{CEN_1996} further show that the law~\eqref{eq:sd} can equally be obtained for a simplified degenerate Cahn--Hilliard model, where the logarithmic potential with small temperature is replaced by its deep quench limit, the double-obstacle potential. 
The idea in these asymptotics is that, on the slow time scale $t\mapsto \ve^2t$, solutions to the Cahn--Hilliard equation should mimic, at leading order, the asymptotic behaviour as $\ve\downarrow0$ of the minimisers of the free energy:
 letting $w_\ve^*=-\ve^2\Delta u_\ve+f'(u_\ve)\in\mathbb{R}$ denote the chemical potential associated to a minimiser of the volume-constrained Cahn--Hilliard free energy $F_\ve$, which acts as a Lagrange multiplier,
 it is a classical result that, asymptotically as $\ve\downarrow0$,
 \begin{align}\label{eq:quasistat}
 	w_\ve^*=\ve \frac{\sigma}{\llbracket u\rrbracket}\kappa+o(\ve),\qquad \sigma=
 	\int_{-1}^{+1}\sqrt{2f(u)}\,\dd u,\qquad\llbracket u\rrbracket=2.
 \end{align}
 See~\cite{LM_1989} for smooth double-well potentials $f(u)$, and \cite{BE_1991} for the non-smooth case.
For the degenerate Cahn--Hilliard equation with logarithmic potential $f(u)$ as in~\eqref{eq:log.pot}, the formal asymptotics in~\cite{CEN_1996} even entail the quantitative asymptotic behaviour~\eqref{eq:quasistat} for the inner solution.
It should be noted that the problem of a rigorous derivation of the surface diffusion flow as the sharp-interface limit of a degenerate Cahn--Hilliard equation is still open.
 
 The choice of the mobility in degenerate Cahn--Hilliard equations is well-known to be able to impact the precise structure of the formal effective interface law, and a subtle interplay 
 between mobility and potential has been observed~\cite{GSK_2008}, not necessarily leading to pure surface diffusion in the sharp-interface asymptotics.
 See also~\cite{LMS_2016_sharp}, where an additional bulk-diffusion term (of lower differential order) was observed numerically and through asymptotic analysis in the interfacial dynamics. In the present paper, we investigate the effect of a rank-deficient matrix-valued degenerate mobility inducing a cross-diffusive coupling to the scalar variable $q$ on the sharp-interface evolution law,
  where the singular part of the bulk free energy part is chosen to be the double-obstacle potential.

\paragraph{Outline of this manuscript.}
In \Cref{ssec:Onsager} we identify the formal gradient-flow structure of the diffuse-interface problem~\eqref{eq:Zhou}, which we use as a basis for introducing generalisations of the model~\eqref{eq:Zhou}. Our findings on the geometric evolution laws governing the sharp-interface dynamics for two variants of the model~\eqref{eq:Zhou}, obtained for constant resp.\ for strictly monotonic coupling, 
are summarised in~\Cref{ssec:results}.
Section~\ref{sec:prelim} introduces suitable parametrisations and coordinate transformations needed in the formal asymptotic analysis.
Sections~\ref{sec:asymptotics},~\ref{sec:constr-elliptic}, and~\ref{sec:geom.law} comprise the main contributions of this work. In Section~\ref{sec:asymptotics}, the sharp-interface asymptotic expansions are performed.
For non-constant monotonic cross-diffusive coupling the asymptotic analysis at third order leads to a constrained second-order elliptic equation in tangential and normal variables (cf.\ Section~\ref{sssec:fsd}), which to the authors' knowledge is new and does not usually occur in sharp-interface asymptotic analyses. An (independent) rigorous well-posedness analysis of the constrained elliptic equation is developed in Section~\ref{sec:constr-elliptic}. As a consequence, we obtain an abstract characterisation of the propagation operator inducing the interface dynamics (cf.\  \Cref{sssec:varchar}).
\Cref{sec:geom.law} is devoted to a structural analysis of the geometric evolution law derived in \Cref{sssec:varchar}. First, based on the rigorous framework in \Cref{sec:constr-elliptic}, we establish the formal gradient-flow structure in the sense of proving symmetry and positivity of the propagation operator (cf.\ \Cref{ssec:nonlin.str}). Subsequently, in \Cref{ssec:symbol}, we focus on explicitly identifying the (leading-order contribution of the) propagation operator.
Relying on spectral and semi-explicit ODE methods, 
we here mostly focus on a specific class of coefficient functions, where bulk modulus and relaxation time are linked to the mobility function.
Generalisations and the investigation of more singular models closer to~\eqref{eq:Zhou} will be left to future research.
Finally, in \Cref{ssec:vanishing.slope}, we show how the two different interface evolution laws derived in Sections~\ref{sec:asymptotics} resp.\ in  Sections~\ref{sec:asymptotics}--\ref{sec:geom.law} for constant resp.\ for strictly monotonic coupling, are formally connected in a singular limit by considering coupling functions with small positive slope.
Some auxiliary geometric identities and transformation rules are recalled in Appendix~\ref{app:diffgeo}.

\subsection{Onsager gradient-flow structure}\label{ssec:Onsager}

The model~\eqref{eq:Zhou} belongs to a class of dissipative evolution equations  characterised by a formal Onsager structure~\cite{Onsager_1931_recipr_I,Mielke_2013_bulk-interface}
\begin{subequations}\label{eq:Ons}
\begin{align}\label{eq:OnsStr}
	\dot{\boldu}=-\mathcal{K}(\boldu)DH(\boldu),
\end{align}
where $H$ denotes the driving functional acting on the state variables $\boldu$, and $DH$ an appropriate differential. 
The linear map $\mathcal{K}=\mathcal{K}(\boldu)$ is the so-called Onsager operator,  a symmetric positive semi-definite operator at each point $\boldu$ in state space.
In the context of~\cite{ZZE_2006}, the state $\boldu=(u,q)$ 
consist of an order parameter $u$ and a quantity $q$ related to bulk stress, while
the driving functional $H(\boldu)=H(u,q)$ is of the form
\begin{align}\label{eq:free.energy}
	H(u,q)=F(u)+\int_\Om\frac{q^2}{2}\,\dd x,
\end{align}
where 
\begin{align}\label{eq:F.gen}
	F(u)=F_\ve(u)=\int_\Om\big(\frac{\ve^2}{2}|\nabla u|^2+f(u)\big)\,\dd x,
\qquad\quad\nabla u\cdot\npo=0, \;\;x\in\pa\Om,
\end{align}
for $\ve>0$. In the notation below, 
 we identify $DH$ with the $L^2$-gradient of $H$, so that 
$DH(u,q)\simeq\left(\begin{matrix}\frac{\delta F}{\delta u}\\q\end{matrix}\right)$.
Then, the Onsager operator takes the form
\begin{align}\label{eq:defOnsager}
	\mathcal{K}(u,q) \square= -N_1(u)^T\divv \big(\mathsf{M}(u)\nabla(N_1(u)\square)\big) + \mathsf{L}(u)\square
\end{align}
with no-flux boundary conditions $\mathsf{M}(u)\nabla(N_1(u)\square)\cdot\npo=0$ for $x\in\pa\Om$,
where $\mathsf{M}(u),\mathsf{L}(u)\in \mathbb{R}^{2{\times}2}_\text{sym}$ are positive semi-definite, and $N_1(u)\in\mathbb{R}^{2{\times}2}$. 
\end{subequations}
Notice that $\mathcal{K}(\boldu)$ is indeed formally symmetric and positive semi-definite with respect to $L^2(\Om)$.
The matrices $N_1(u), \mathsf{L}(u)$ are chosen such that the invariance property $\mathcal{K}(\boldu)\begin{pmatrix}
	1_\Om\\0
\end{pmatrix}\equiv0$ is fulfilled, where $1_\Om$ denotes the constant function, equal to $1$, on $\Om$.

Symmetry and positivity of $\mathcal{K}$ imply,  along suitably regular solution trajectories $\boldu=\boldu(t)$, the entropy--entropy-dissipation identity
\begin{align*}
	\frac{\dd}{\dd t}H(\boldu) = -\mathcal{D}(\boldu),
\end{align*}
where for $\boldu=(u,q)$ 
\begin{align*}
	\mathcal{D}(\boldu) = \int_\Om\nabla (N_1(u)DH(\boldu)):\mathsf{M}(u)\nabla(N_1(u)DH(\boldu))\,\dd x+
	\int_\Om DH(\boldu)\cdot\mathsf{L}(u)DH(\boldu)\,\dd x\ge0.
\end{align*}
The invariance property, in turn, combined with $\mathcal{K}=\mathcal{K}^*$ entails volume conservation $\frac{\dd}{\dd t}\int_\Om u\,\dd x = 0$.

Below, we provide some examples.

\paragraph{The Zhou--Zhang--E model.} 
Let $H$ be of the form~\eqref{eq:free.energy}, \eqref{eq:F.gen} with $f$ given by~\eqref{eq:log.pot}.
Then, system~\eqref{eq:Zhou} is obtained from~\eqref{eq:OnsStr}--\eqref{eq:defOnsager} by choosing
\begin{subequations}\label{eq:zze}
\begin{align}\label{eq:M.zze}
	\mathsf{M}(u) &= N_2(u)^Tm(u)(\mathbf{1}\otimes\mathbf{1})N_2(u),\quad\text{where }\mathbf{1}=(1,1)^T,
	\\\mathsf{L}(u)&=\left(\begin{matrix}
		0&0\\0&\frac{1}{\tau(u)}
	\end{matrix}\right),\qquad \label{eq:L.zze}
	\\N_1(u)&=\diag(1,-A(u)),\quad N_2(u)=\diag(1,\frac{1}{n(u)}),\label{eq:N1N2.zze}
\end{align}
\end{subequations}
where
$$m(u)=(1-u^2)^2\tilde m(u),\qquad n(u)=1-u^2.$$
Observe that the $2{\times}2$-mobility matrix $\mathsf{M}(u)$ in~\eqref{eq:M.zze} is singular of rank one for all $u\in(-1,1)$.

\paragraph{A special variant.}
Taking $n(u)\equiv 1$ in~\eqref{eq:Ons},~\eqref{eq:zze} yields the PDE system
\begin{align*}
	\partial_tu&=-\divv(m(u)\boldsymbol{j}),
	\\\partial_tq&=A(u)\divv(m(u)\boldsymbol{j})-\frac{1}{\tau(u)}q,
\end{align*}
where $\boldsymbol{j}=\nabla(\frac{\delta F}{\delta u}-A(u)q)$.
Equivalently, this PDE system can be written in the form
\begin{align*}
	\begin{split}
	\partial_tu&=-\divv(m(u)\boldsymbol{j}),
	\\\partial_tz&=-\frac{1}{\tau(u)}q,\qquad z:= q+R(u),\qquad \boldsymbol{j}=-\nabla(\frac{\delta F}{\delta u}-A(u)q),
	\end{split}
\end{align*}
where $R'=\frac{A}{n}$, which exposes the hyperbolic/ODE-like features of viscoelasticity.

\paragraph{Double-obstacle potential.}
A primary purpose of this article is to understand the effect of the dissipation mechanism~\eqref{eq:defOnsager}--\eqref{eq:zze} on the sharp-interface asymptotics in~\eqref{eq:Ons}.
To focus on the main ideas, we will directly work with the deep quench limit of the logarithmic entropy function in~\eqref{eq:log.pot}. 
Thus, we consider~\eqref{eq:free.energy} with a Cahn-Hilliard free energy $F=F_\ve^{\rm\textsc{(DO)}}$~\eqref{eq:F.gen}, where $f=	f^{\rm\textsc{(DO)}}$ is of double obstacle type
\begin{align*}
	f^{\rm\textsc{(DO)}}(u)=\iota_{[-1,1]}(u)+\frac12(1-u^2),\qquad\iota_{[-1,1]}(u)=\begin{cases}
		0&\text{ if }u\in[-1,1],
		\\+\infty&\text{ if }u\in\mathbb{R}\setminus[-1,1].
	\end{cases}
\end{align*}
As will be detailed in Section~\ref{sec:asymptotics}, the double-obstacle potential turns the diffuse-interface model into a free-boundary problem.
The global existence of weak solutions to the Cahn--Hilliard equation ($A\equiv0$) with double-obstacle potential in the case of a constant mobility has been established in~\cite{BE_1991}. 
For results concerning degenerate mobilities, we refer to~\cite{EG_1996,BL_2022_vps_degmob}.

\subsection{Main results}\label{ssec:results}
Our basic strategy is to adapt the approach of~\cite[Section~3]{CEN_1996} involving the double-obstacle potential to a class of Onsager-type VPS models~\eqref{eq:Ons},~\eqref{eq:zze}. After rescaling to the appropriate slow time scale, $t\mapsto \ve^2 t$, the equations for $(u,q,w)=(u_\ve,q_\ve,w_\ve)$ take the form
	\begin{subequations}\label{eq:vpsn}
	\begin{align}
		\label{eq:vpsn.u}
		&\ve^2\partial_tu=-\divv\big(m(u)\,\boldsymbol{j}\big),
		\qquad \boldsymbol{j}=-\Big[\nabla w 
		-\frac{1}{n(u)}\nabla (A(u)q)\Big],
		\quad&& t>0,x\in\Om,
		\\&\ve^2\partial_tq=-\frac{1}{\tau(u)}q
		+A(u)\divv\Big(\tfrac{m(u)}{n(u)}\boldsymbol{j}\Big),
		&& t>0,x\in\Om,
		\label{eq:vpsn.q}
		\\[2mm]	&m(u)\boldsymbol{j}\cdot\npo=0,\qquad\nabla u\cdot\npo=0,\;\qquad&& t>0,x\in\pa\Om,
	\end{align}
	where 
	\begin{align}
		w\in\pa_u F_\ve^{\rm\textsc{(DO)}}=-\ve^2\Delta u-u+\pa\iota_{[-1,1]}.
	\end{align}
	\end{subequations}
	
The mobility function $m$  is assumed to degenerate precisely in the two pure phases $\{u=\pm1\}$. In our asymptotic analysis, we restrict to linearly degenerate mobilities of the form
\begin{enumerate}[label=(m\arabic*)]
	\item\label{hp:m.full} $m(u)=(1-u^2)\tilde m(u)$, where  $\tilde m\in C^\infty([-1,1])$ with $\min_{[-1,1]} \tilde m>0$,
\end{enumerate}
which is the classical choice when combined with the logarithmic singular potential. 

The relaxation time $\tau$ and the bulk modulus $A$ are chosen in such a way that 
\begin{enumerate}[label=(\text{$\tau$}\arabic*)]
	\item\label{hp:tauA.reg}  $\tau, A\in C^\infty([-1,1])$ with 
		$A^2\tau>0$ on $(-1,1)$.  
\end{enumerate}\medskip 

\noindent
Our formal asymptotics are based on the assumption of a well-defined smooth interface motion in the limit $\ve\to0:$
\begin{enumerate}[label=(G\arabic*)]
	\item \label{hp:Gamma}
	There exists a non-trivial compact time interval $I\subset\mathbb{R}_{\ge0}$ such that
	the zero level set $\Gamma_\ve\Subset I\times\Om$ of $u_\ve:I\times\Om\to\mathbb{R}$ approaches an evolving 
	hypersurface $\Gamma=\cup_{t\in\iv}\{t\}{\times}\Gamma(t)\Subset I\times\Om$ with the property that, for all $t\in \iv$, $\Gamma(t)\Subset\Om$ is a smooth, closed, connected, and embedded hypersurface smoothly varying in $t\in I$.
	\\Furthermore, the order parameter $u_\ve$ converges, as $\ve\to0$, to a pointwise limit $u=u(t,x)\in\{\pm1\}$ in $\iv\times\Omega$, and with $\Om^\pm(t):=\{u(t,\cdot)=\pm1\}$ it holds that 
	\[\Om=\Om^-(t)\cup \Gamma(t)\cup \Om^+(t),\qquad t\in\iv.\]
\end{enumerate}
Throughout this article, by a \textit{closed} hypersurface we mean a $(\dimx-1)$-dimensional differentiable submanifold of $\mathbb{R}^\dimx$ that is topologically compact and without boundary.

The starting point of our asymptotic analysis is the assumption that the late-stage evolution captured along the slow time scale inherits property~\eqref{eq:quasistat} \textit{at leading order} in the sense that, to leading order, the chemical potential vanishes across the interface. While this hypothesis leads to a consistent asymptotic analysis, we leave it open whether or not it may be deduced from our set-up.
Let us point out that, in the pure Cahn--Hilliard case, a similar assumption was made in~\cite[Section~3]{CEN_1996}.
We emphasise that the quantitative property~\eqref{eq:quasistat} up to first order cannot be expected for the present VPS models unless $A\equiv0$.  
Finally, let us note that it is not necessary to impose an analogous stationarity assumption on the leading order contribution of the bulk stress variable, whose $O(\ve)$ behaviour is a consequence of our asymptotics procedure.  

Depending on the choice of the coupling function $n=n(u)$, our asymptotic analysis leads to two different non-local lower-order variants of the surface diffusion flow:

\paragraph{Intermediate surface diffusion.} 
The following result is a by-product of our asymptotics.
\begin{thm}[Intermediate surface diffusion]
	\label{ass:isd}
	Consider~\eqref{eq:vpsn} assuming~\ref{hp:m.full},~\ref{hp:tauA.reg}, and $n\equiv1$ 
	on the model coefficients and~\ref{hp:Gamma} on the limiting geometry.
	Then,  as $\ve\downarrow0$, the formal sharp-interface asymptotics lead to the 
	intermediate surface diffusion flow
	\begin{align}\label{eq:isd.main}
		\sfV_\Gamma=-\sigma\Delta_\Gamma\big(\delta\,\text{\normalfont Id}-\om\Delta_\Gamma\big)^{-1}\kappa_\Gamma,
	\end{align}
	where 
	\begin{subequations}\label{eq:sig-del-om.u}
		\begin{align}\label{eq:sig-del.u}
			&\sigma=\int_{-1}^{+1}\sqrt{(1-u^2)}\,\dd u=\int_{-1}^{+1}\sqrt{2f^{\rm\textsc{(DO)}}(u)}\,\dd u,
			\quad
			\delta=4\Big(\int_{-1}^{+1}\frac{m(u)}{\sqrt{1-u^2}}\,\dd u\Big)^{-1},
			\\\label{eq:om.u}
			&\omega=\int_{-1}^{+1}A^2(u)\tau(u)\sqrt{1-u^2}\,\dd u.
		\end{align}	
	\end{subequations}
\end{thm}
The intermediate surface diffusion flow~\eqref{eq:isd.main}
was introduced by Cahn and Taylor~\cite{CahnTaylor_1994_surfaceDiffusion,TaylorCahn_1994_anisotropic} as a volume-preserving and area-decreasing geometric evolution connecting the classical volume-preserving mean-curvature flow ($\delta\downarrow0$) to the surface diffusion flow ($\om\downarrow0$). It is the formal gradient flow of the surface area functional with respect to a metric structure induced by the weighted sum of the (volume-preserving) $\dot L^2(\Gamma)$ and $\dot H^{-1}(\Gamma)$.
Elliott and Garcke~\cite{EG_1997_surface_motion} proposed the viscous degenerate Cahn--Hilliard equation as its diffuse-interface counterpart, see also~\cite{CahnTaylor_1994_surfaceDiffusion,TaylorCahn_1994_anisotropic}. In view of \Cref{ass:isd}, the model~\eqref{eq:vpsn} with $n\equiv1$ provides an alternative phase-field approximation.
Heuristically, the  viscous degenerate Cahn--Hilliard equation can be obtained from the viscoelastic Cahn--Hilliard model~\eqref{eq:vpsn} with $n\equiv 1$ and $A\equiv A_0>0,\tau\equiv\tau_0>0$ in the regime $\tau_0\ll1$ with $A^2_0\tau_0\sim 1$.

The first existence result under a smallness condition (short time or close to a steady state) for the intermediate surface diffusion flow was  obtained in~\cite{EG_1997_surface_motion} for planar curves by means of energy estimates.
In the general multi-dimensional case,  well-posedness results under smallness were established by Escher and Simonett~\cite{ES_1999_abstractParabolic} for smooth, closed, embedded, connected hypersurfaces based on tools from maximal parabolic regularity and analytic semigroups.  
The rigorous singular limit, locally in time, towards the volume-preserving mean-curvature flow 
was performed in~\cite{EGI_2001_limiting_isd_curves,EGI_2002_limiting_isd}. We refer to~\cite{EscherIto_2004_intermediate} for a review and further qualitative properties of the intermediate surface diffusion flow.

\paragraph{Fractional surface diffusion.} The main result of this work pertains to strictly monotonic coupling functions $n=n(u)$ satisfying 
\begin{enumerate}[label=(n\arabic*)]
	\item\label{hp:nm}  $n\in C^\infty([-1,1])$ with $\min_{[-1,1]} |n|>0$ 
	\item\label{hp:n.ndeg} $\min_{[-1,1]}|n'|>0$
\end{enumerate}
under the condition that 
\begin{enumerate}[label=(\text{$\tau$}\arabic*)]\setcounter{enumi}{1}
	\item\label{hp:am} the function $a(u):=\Big(\frac{1}{A^2\tau}\big(\frac{n^2}{n'}\big)^2\Big)_{|u}\,\frac{1}{1-u^2}$ satisfies $a(u)=\frac{\tilde a(u)}{m(u)}$ 	for some $\tilde a\in C^\infty([-1,1])$ with $\min_{[-1,1]}\tilde a>0$.
\end{enumerate}
Hypothesis~\ref{hp:n.ndeg} is complementary to the case $n\equiv 1$, where this condition is clearly violated.
Under hypotheses~\ref{hp:m.full}, \ref{hp:tauA.reg}, \ref{hp:nm}, and~\ref{hp:n.ndeg}, assumption~\ref{hp:am} essentially means that $A^2\tau\sim1$. 

Below, for a smooth closed connected embedded hypersurface $\Sigma$ we let $\{\ee_k\}_{k\in \mathbb{N}}$ denote an orthonormal basis of $\dot L^2(\Sigma):=\{\sfh\in L^2(\Sigma):\av_{\Sigma} \sfh\,\dd\mathcal{H}^{\dimx-1}=0\}$ composed of eigenfunctions of the minus Laplace--Beltrami operator $-\Delta_{\Sigma}$ with associated eigenvalues $0<\lambda_1\le\lambda_2\le\dots$ satisfying $\lambda_k\to\infty$ (see also \Cref{ssec:regularity}). 
We set $\Lambda:=\{\lambda_k:k\in\mathbb{N}\}$.

\begin{thm}[Square-root minus Laplace--Beltrami]
	\label{ass:fsd}
	Consider~\eqref{eq:vpsn}, and assume hypotheses~\ref{hp:m.full}, \ref{hp:tauA.reg}, \ref{hp:nm}, \ref{hp:n.ndeg}, and~\ref{hp:am} on the model coefficients and~\ref{hp:Gamma} on the limiting geometry. Further, let $\sigma,\delta$ be as in~\eqref{eq:sig-del.u}.
	Then, the sharp-interface asymptotics lead to fractional versions of the surface diffusion flow
	\begin{align*}
		\sfV_\Gamma=\mathcal{G}_\Gamma\kappa_\Gamma.
	\end{align*}
	For any smooth closed connected embedded hypersurface $\Sigma$,  $\mathcal{G}_\Sigma$ is an unbounded linear operator with respect to $L^2(\Sigma)$ enjoying the following properties
	(cf. Sections~\ref{sec:asymptotics}--\ref{sec:geom.law}):
		\begin{itemize}
		\item 
		{\normalfont Curvature flow:} $\mathcal{G}_\Sigma$ is symmetric and positive 
		\item {\normalfont Volume preservation:}  $\mathcal{G}_{\Sigma}1_{\Sigma}=0$ 
		\item {\normalfont Dominance by surface diffusion:} $\mathcal{G}_\Sigma\le-\frac\sigma\delta\Delta_{\Sigma}$ 
		\item {\normalfont Representation via $-\Delta_\Sigma:$}  $\exists!$  $\zeta:\Lambda\to\mathbb{R}_{>0}$ such that
		$(\mathcal{G}_\Sigma \ee_k,\ee_l)_{L^2(\Sigma)}=\zeta(\lambda_k)\delta_{kl}$ for all $k,l\in \mathbb{N}$
		\item {\normalfont Fractional surface diffusion:} 
		Let $a(u)m(u)=1$. Then 	\begin{align}\label{eq:G=sigsqrt}
			\mathcal{G}_\Sigma=\sigma\eta\sqrt{-\Delta_\Sigma}+\sigma\mathcal{R}(\sqrt{-\Delta_\Sigma}),
		\end{align}
		with $\eta=\big((\tfrac{n(1)}{n'(1)})^2+(\tfrac{n(-1)}{n'(-1)})^2\big)^{-1}$, where $\mathcal{R}(\sqrt{-\Delta_\Sigma})$
		stands for a lower-order perturbation:
		\\the map
		$\varrho:\Lambda\to\mathbb{R}$ given by
		$\varrho(\lambda_k)=(\mathcal{R}(\sqrt{-\Delta_\Sigma})\ee_k,\ee_k)_{L^2(\Sigma)}$ 
		satisfies $\lim_{\lambda\uparrow\infty}\tfrac{|\varrho(\lambda)|}{\sqrt{\lambda}}=0$.
		\item {\normalfont Asymptotically close to $\sigma\eta\sqrt{-\Delta_\Sigma}:$}  Let $a(u)m(u)=1$ and $n(u)=\beta_0+\beta_1(u+1)$, $\beta_0,\beta_1>0$. Then
	\[(\mathcal{R}(\sqrt{-\Delta_\Sigma})\ee_k,\ee_k)_{L^2(\Sigma)}\to0\;\;\text{ exponentially
		 as }\sqrt{\lambda_k}\to\infty.\] 
   Consequently, the pseudo-differential operator $\mathcal{R}$ has order $-\infty$.
	\end{itemize}
\end{thm}
Rigorous statements concerning the construction and properties of the operator $\mathcal{G}_\Gamma$ are provided in \Cref{prop:solveP.vn} (abstract definition), \Cref{prop:gs} (curvature flow), and~\Cref{prop:spectral} (PDE structure).

It appears that the geometric flow induced by fractional versions of
 the surface Laplacian, 
  even in the case of the square root minus Laplace--Beltrami as the propagation operator, has so far not been investigated systematically in the literature.
We expect that the local existence and uniqueness of classical solutions may be obtained, for instance,
by an adaptation of the maximal parabolic regularity approach~\cite{EMS_1998_surface-diffusion,ES_1999_abstractParabolic}.

\paragraph{Connecting the two laws.}
Observe that, with regard to differential order, there is an apparent discontinuity between the interface evolution laws in \Cref{ass:fsd} (second order) and \Cref{ass:isd} (third order). 
In \Cref{ssec:vanishing.slope}, we will show that the special case considered in \Cref{ass:isd} with $n\equiv1$ can formally be recovered from the laws derived in \Cref{ass:fsd} by taking the singular limit $\eps\downarrow0$ in a family of problems involving coupling coefficients $n_\eps$ with small positive slope $\eps$. A summary of this result is provided in the following remark. 
\begin{remark}[Intermediate surface diffusion as a singular limit in the fractional third-order laws]
	\label{ass:fsd->isd}
Let the hypotheses of \Cref{ass:fsd} be in force and let $m(u)=(1-u^2)\tilde m(u)$ be even. Consider the family of coupling functions $n=n_\eps$ satisfying
\begin{align*}
n_\eps(u)=1+\eps u, \qquad u\in[-1,1],
\end{align*}
with $0<\eps\ll1$.
Further let $A^2_\ve\tau_\eps=n_\eps^4\tilde m$, so that $a_\eps=\frac{\eps^{-2}}{m}$.
Let $\mathcal{G}_{\eps,\Sigma}$ denote the propagation operator of the interface law derived in \Cref{ass:fsd} and set $\zeta_\eps(\lambda_k):=(\mathcal{G}_{\eps,\Sigma}\ee_k,\ee_k)_{L^2(\Sigma)}$.
 Then, 
	\[\zeta_\eps(\lambda)=\sigma\lambda\big(\delta+\om\lambda+O_R(\eps)\big)^{-1} \qquad \text{if }\lambda=\lambda_k\le R, \]
	as long as $0<\eps\ll_R1$.
	Here, the coefficients $\delta$ and $\omega$ are identical to those in \Cref{ass:isd} for the given, $\eps$-independent functions $m(u)=(1-u^2)\tilde m(u)$, $n\equiv 1$, and $A^2\tau=n^4\tilde m=\tilde m$.
\end{remark}
Thus, loosely speaking, at low frequencies we recover the propagation operator associated to the intermediate surface diffusion flow in the sense that, on compact subsets in frequency space and for $0<\eps\ll1$,
\[\mathcal{G}_{\eps,\Gamma}=\;{\text{\Large$^{\mlq\mlq}$}}\;-\sigma\Delta_\Gamma\big(\delta\,{\rm Id}-\om\Delta_\Gamma+O_{\Delta_\Gamma}(\eps)\big)^{-1}{\text{\Large$^{\mrq\mrq}$}}\;,\]
where $O_{\Delta_\Gamma}(\eps)$ stands for an (unbounded) linear operator that converges to zero as $\eps\downarrow0$, at least linearly, on finite linear  combinations of the basis functions $\{\ee_k\}_{k\in \mathbb{N}}$.

For the precise closed formula for $\zeta=\zeta_\eps(\lambda)$ in the setting of \Cref{ass:fsd->isd}, we refer to equation~\eqref{eq:zeta.to.isd.exact}.

\section{Preliminaries}\label{sec:prelim}

In this preparatory section, we introduce the coordinate transformations and geometric identities needed in the formal asymptotic analysis. 
The setting chosen below is motivated by the following. We expect that for $0<\ve\ll1$ the phase field component $u_\ve$ of the solution to~\eqref{eq:vpsn} changes from one phase to the other on a thin interfacial layer of width $\sim\ve$.
In the transition layer, which lies in the vicinity of the limiting interface $\Gamma$ (cf.~\ref{hp:Gamma}), we introduce new coordinates mostly following~\cite{AGG_2012}.

\subsection{Evolving interface}

For a non-trivial compact time interval $I\subset\mathbb{R}_{\ge0}$, consider a finite family of smooth local parametrisations $\gamma_\param:I\times \mathcal{O}_\param\to\mathbb{R}^\dimx$ with $\mathcal{O}_\param\subset\mathbb{R}^{d-1}$ open and 
 $\gamma_\param(t,\cdot):\mathcal{O}_\param\to\gamma_{[\alpha]}(t,\mathcal{O}_\param)\subset \Gamma(t)$  a diffeomorphism  for every
$1\le\alpha\le N$ such that 
 $\Gamma(t)=\cup_{1\le\alpha\le N}\gamma_\param(t,\mathcal{O}_\param)$  for all $t\in I$.
In the following, we let $\alpha\in\{1,\dots,N\}$ be fixed but arbitrary and abbreviate $\gamma:=\gamma_\param$,  $\mathcal{O}:=\mathcal{O}_\param$.
Unless stated otherwise, geometric quantities of the evolving interface $\cup_{t\in I}\{t\}\times\big(\Gamma(t)\cap\gamma(t,\mathcal{O})\big)$ will be considered as functions on $I\times\mathcal{O}$ by means of the parametrisation $\gamma$. The unit normal field to $\Gamma(t)$ pointing towards $\Om^+(t)$ will be denoted by $\nu(t,\cdot):\mathcal{O}\to \mathbb{R}^d$.
Then the (scalar) normal velocity $V:I\times\mathcal{O}\to\mathbb{R}$ of the evolving interface $\Gamma$ is defined via (see e.g.\ \cite[Chapter 2.2.5]{PruessSimonett_2016}) \[V=\partial_t\gamma\cdot\nu.\]
Let $d(t,\cdot):\Om\to\mathbb{R}$ denote the signed distance function to $\Gamma(t)$, with the convention that $d>0$ in the phase $\Om^+=\{u=+1\}$.
Then there exists $\ol d>0$ such that $d(t,\cdot)$ is smooth 
 in the $\ol d$-tubular neighbourhood $\mathcal{N}_{\ol d}(t):=\{|d(t,\cdot)|<\ol d\}\Subset\Om$ of $\Gamma(t)$ for all $t\in I$ and such that  on $\mathcal{N}_{\ol d}(t)$ the orthogonal projection $\proj_{\Gamma(t)}$ onto $\Gamma(t)$ is well-defined.
 We note the following basic identities for $x\in\mathcal{N}_{\ol d}(t)$ (see e.g.\ \cite{Ambrosio_2000,BMST_2022_M3AS_sd}):
 \begin{align}\label{eq:signed-dist.bulk}
 	\nabla_x d(t,x)=\nu_{\Gamma}(t,\proj_{\Gamma(t)}(x)),\qquad
 	\pa_td(t,x)=-\sfV_{\Gamma}(t,\proj_{\Gamma(t)}(x)),
 \end{align}
 where $\nu_{\Gamma}:\Gamma\to\mathbb{R}^d$ denotes the unit normal field to $\Gamma$ determined by $\nu_\Gamma(t,\gamma(t,s))=\nu(t,s)$, 
 and
 $\sfV_{\Gamma}:\Gamma\to\mathbb{R}$ the normal velocity of the moving interface related to $V$ by $\sfV_{\Gamma}(t,\gamma(t,s))=V(t,s)$.
Below, by $\kappa_\gamma$ we denote the (scalar) mean curvature of $\Gamma$, 
i.e.\ the sum of its principle curvatures, considered as a function on $I\times \mathcal{O}$, where we adopt the sign convention that $\kappa_\gamma(t,\cdot)\le0$ if $\Om^-(t)$ is convex. 
By $\kappa_\Gamma:\Gamma\to\mathbb{R}$ we denote the mean curvature of $\Gamma$, considered as a function on $\Gamma$, so that $\kappa_{\Gamma}(t,\gamma(t,s)):=\kappa_{\gamma}(t,s)$.

\subsection{Parametrisation for the bulk region}

Based on the mappings $\gamma(t,\cdot)$, we construct local parametrisations of the tubular neighbourhood $\mathcal{N}_{\ol d}(t)$ of  $\Gamma(t)$ via
\begin{align*}
	\gamma^\ve_t(s,\rho)=\gamma(t,s)+\ve\rho\nu(t,s),\quad (t,s)\in I\times\mathcal{O}, \quad \rho\in J_\ve:=(-\ve^{-1}\ol d,\ve^{-1}\ol d).
\end{align*}
Here, the rescaling $\rho=\frac{d}{\ve}$ serves to normalise, at leading order, the thickness of the interfacial transition region in the new coordinates.
We sometimes omit the dependence on the time parameter $t$, and simply write $\gamma^\ve(\cdot,\rho)$. Furthermore, we abbreviate $\gamma_{\ve\rho}=\gamma^\ve(\cdot,\rho)$, if no confusion arises with the time subscript.
Then,  the map
  \begin{align*}
G^\ve:I\times\mathcal{O}\times J_\ve\to G^\ve(I\times\mathcal{O}\times J_\ve)=:\mathcal{N},\qquad  
G^\ve(t,s,\rho)=(t,\gamma_t^\ve(s,\rho))
  \end{align*}  
  is a local parametrisation of the (spatial) $\ol d$-tubular neighbourhood $\mathcal{N}$ of $\Gamma$.
  We denote its inverse $(t,x)\mapsto (t,s,\rho)$ by
  $ (\text{id}_I,\mathfrak{S},\RD):=(G^\ve)^{-1}: \mathcal{N}\to I\times \mathcal{O}\times J_\ve.$
Thus, $\RD(t,x)=\frac{d(t,x)}{\ve}$ and, owing to~\eqref{eq:signed-dist.bulk}, we deduce
 \begin{align}\label{eq:dRdt}
 	\pa_t\RD\circ G^\ve=\ve^{-1}\pa_t d\circ G^\ve=-\ve^{-1} V.
 \end{align}
We now compute the differential operators in the new coordinates. 
For differentiable scalar functions $u=u(t,x), b=b(t,x)$, and a vectorial function $\boldsymbol{j}=\boldsymbol{j}(t,x)$, we write $U(t,s,\rho):=u(G^\ve(t,s,\rho))$, $B(t,s,\rho):=b(G^\ve(t,s,\rho))$, and $\boldsymbol{J}(t,s,\rho):=\boldsymbol{j}(G^\ve(t,s,\rho))$.  
 From~\eqref{eq:dRdt} we infer
\begin{subequations}\label{eq:diffop.traf}
\begin{align}\label{eq:diffop.traf.time}
\pa_tu\circ G^\ve &= -\ve^{-1}V\partial_\rho U+\partial_t\SL\circ G^\ve\cdot\nabla_sU+\partial_tU,
\end{align}
The following identities follow from basic geometric calculus (cf.\  Appendix~\ref{app:diffgeo}):
\begin{align}\label{eq:diffop.traf.space}
	\begin{aligned}
\nabla_xu\circ G^\ve&=\ve^{-1}\partial_\rho U\,\nu+\nablaGeps U, \qquad 
\\(\divv_x \boldsymbol{j})\circ G^\ve&=\ve^{-1}\pa_\rho \boldsymbol{J}\cdot \nu+\divGeps \boldsymbol{J},
\\\divv_x(b\nabla_x u)\circ G^\ve&=\ve^{-2}\pa_\rho(B\pa_\rho U)+\ve^{-1}B\pa_\rho U \Delta_x d\circ G^\ve+\divGeps(B\nablaGeps U),
\end{aligned}
\end{align}
\end{subequations}
and, in particular, $\Delta_xu\circ G^\ve=\ve^{-2}\partial_\rho^2 U+\ve^{-1}\Delta_x d\circ G^\ve\,\partial_\rho U+\DeltaGeps U$.
Here, $\nablaGeps U:=\nabla_{\Gamma_{\ve\rho}} u\circ G^\ve$ resp.\ $\DeltaGeps U:=\Delta_{\Gamma_{\ve\rho}} u\circ G^\ve$ denote the surface gradient resp.\ Laplace--Beltrami operator of $u$ with respect to the hypersurface ($t$-dependence omitted) 
\[\Gamma_{\ve\rho}=\{\gamma^\ve(s,\rho),\; s\in\mathcal{O}\},\]
expressed in terms of the parametrisation $(\mathcal{O},\gamma^\ve(\cdot,\rho))$.
Likewise, $\divGeps\boldsymbol{J}:=(\divv_{\Gamma_{\ve\rho}} \boldsymbol{j})\circ G^\ve$ denotes the surface divergence of $\boldsymbol{j}$ with respect to $\Gamma_{\ve\rho}$ in local coordinates.

We want to expand these operators in terms of their $\ve$-independent counterparts $\nablaG U:=\nabla_\Gamma u\circ \gamma$, $\DeltaG U:=\Delta_\Gamma u\circ\gamma$, where  $\nabla_\Gamma u$ and $\Delta_\Gamma u$ denote the surface gradient and Laplace--Beltrami operator applied to $u_{|\Gamma}:\Gamma\to\mathbb{R}$.
As shown in Appendix~\ref{ssec:diffgeo.transform}, for any smooth scalar $U=U(s,\rho)$ and vectorial $\boldsymbol{J}=\boldsymbol{J}(s,\rho)$ 
\begin{equation}\label{eq:diffop.expansion}
	\begin{split}
\nablaGeps U&=\nablaG U+
\ve\rho\sum_{i=1}^{\dimx-1} \mrem^i\pa_{s_i}U+O(|\ve\rho|^2),
\\\divGeps \boldsymbol{J}&=\divG \boldsymbol{J}
+\ve\rho\sum_{i=1}^{\dimx-1} \mrem^i\cdot\pa_{s_i}\boldsymbol{J}+O(|\ve\rho|^2),
\end{split}
\end{equation}
 for tangential fields $\mrem^i(s)$ (satisfying  $\nu\cdot\mrem^i\equiv0$), $ i=1,\dots,\dimx-1$, 
  that only depend on $\gamma$.
In particular, 
\begin{align}\label{eq:divgrad.expansion}
	\divGeps(B\nablaGeps U)=\divG(B\nablaG U)+O(|\ve\rho|).
	\end{align}
 We further note that (cf.\ Appendix~\ref{ssec:diffgeo.id})
\begin{align}\label{eq:d.Lapl}
	\Delta_xd\circ G^\ve&= -\kappa_\gamma-\ve\rho |\mathcal{W}_\gamma|^2-\ve^2\rho^2k_3^3+O(|\ve\rho|^3),
\end{align}
where $|\mathcal{W}_\gamma|=(\sum_{i=1}^{\dimx-1}\kappa_i^2)^{1/2}$ denotes the Frobenius norm of the Weingarten tensor of $\Gamma$, and $k_3^3:=\sum_{i=1}^{\dimx-1}\kappa_i^3$, where $\kappa_i$ are the principle curvatures of $\Gamma$, considered as functions on $I\times\mathcal{O}$.

In the next section, we will adapt the approach of~\cite{CEN_1996} to study the sharp-interface asymptotics of the cross-diffusion models~\eqref{eq:vpsn}. 
We caution that the authors in~\cite{CEN_1996} use a different parametrisation. 

\section{Sharp-interface asymptotics}
\label{sec:asymptotics}

In this section, we apply the method of formal asymptotic expansions to the Onsager VPS models~\eqref{eq:vpsn}. Throughout this section, we assume~\ref{hp:Gamma} and impose hypotheses~\ref{hp:m.full},~\ref{hp:tauA.reg}, and~\ref{hp:nm}. In addition, we will assume that either $n\equiv1$ (cf.\ \Cref{ass:isd}) or~\ref{hp:n.ndeg} holds (cf.\ \Cref{ass:fsd}).

 To begin with, we rewrite equation~\eqref{eq:vpsn.q} using~\eqref{eq:vpsn.u}, to obtain the formally equivalent problem
	\begin{subequations}\label{eq:phn}
	\begin{align}
		\label{eq:phn.u}
		&\ve^2\partial_tu=-\divv\big(m(u)\,\boldsymbol{j}\big),
		\qquad \boldsymbol{j}=-\Big[\nabla w 
		-\frac{1}{n(u)}\nabla (A(u)q)\Big],
		\quad&& t>0,x\in\Om,
		\\&\ve^2\partial_tz=-\frac{1}{\tau(u)}q
		+A(u)m(u)\nabla(\frac{1}{n(u)})\cdot \boldsymbol{j},\quad
		z=q+R(u), 
		&& t>0,x\in\Om,
		\label{eq:phn.q}
		\\[2mm]	&m(u)\boldsymbol{j}\cdot\npo=0,\qquad\nabla u\cdot\npo=0,\;\qquad&& t>0,x\in\pa\Om,	\label{eq:phn.bc}
	\end{align}
 where $R'=\frac An$, $R(0)=0$, and 
	\begin{align*}
		w\in\pa_u F_\ve^{\rm\textsc{(DO)}}=-\ve^2\Delta u-u+\pa\iota_{[-1,1]}.
	\end{align*}
\end{subequations}
Notice that for $n\equiv 1$, equation~\eqref{eq:phn.q} reduces to a $u$-dependent ordinary differential equation for $z$, and in fact, 
the sharp-interface analysis of~\eqref{eq:phn} turns out to be much less delicate if $n$ is constant.

\paragraph{Free boundary problem.}

We wish to study the asymptotic behaviour of solutions $(u_\ve,q_\ve,w_\ve)=(u,q,w)$ of~\eqref{eq:phn} as $\ve\downarrow0$.
Equation~\eqref{eq:phn} can formally be written as a free boundary problem, where at each point in time the domain $\Om$ is decomposed as
\begin{align*}
	\Om=\Om_\ve^-(t)\cup\Om_\ve^\sfI(t)\cup\Om_\ve^+(t),
\end{align*}
with $\Om_\ve^\pm(t):=\{u_\ve(t,\cdot)=\pm1\}$, and where for $t>0$,  $x\in\Om_\ve^\sfI(t):=\{|u_\ve(t,\cdot)|<1\}$ the unknowns $(u_\ve,q_\ve,w_\ve)$ are subject to the equations
	\begin{subequations}\label{eq:fb}
	\begin{align}
		\label{eq:fb.u}
		&\ve^2\partial_tu=-\divv\big(m(u)\,\boldsymbol{j}\big),
		\qquad \boldsymbol{j}=-\Big[\nabla w 
		-\frac{1}{n(u)}\nabla (A(u)q)\Big],
		\\&\ve^2\partial_tz=-\frac{1}{\tau(u)}q
		+A(u)m(u)\nabla(\frac{1}{n(u)})\cdot \boldsymbol{j},\quad
		z=q+R(u), 
		\label{eq:fb.q}
	\\&w=-\ve^2\Delta u-u,\label{eq:fb.w}
	\end{align}
	where $R'=\frac An$, $R(0)=0$.  We complement these equations by appropriate continuity conditions on the free boundary  $\pa\Om_\ve^\sfI(t)\cap\Omega_\ve^\pm(t)$, which take the form
\begin{align}\label{eq:match.toy.ori}
	\begin{aligned}
		&m(u)\boldsymbol{j}\cdot\nu_\ve^\sfI=0,
		\\&	u=\pm1,\quad
		\nabla u\cdot\nu_\ve^\sfI=0.
	\end{aligned}
\end{align}
Here, $\nu_\ve^\sfI$ denotes the outer unit normal field to $\Om_\ve^\sfI(t)$.
\end{subequations}

The last two conditions in~\eqref{eq:match.toy.ori} necessarily follow if $u_\ve(t,\cdot)\in C^1(\Om)$. Combined with the first line in~\eqref{eq:match.toy.ori}, these continuity conditions ensure the conservation of the quantitiy $\int_\Om u_\ve\,\dd x$ in time.
 In this paper, we study the sharp-interface asymptotics  postulating the strong formulation~\eqref{eq:fb.u}--\eqref{eq:match.toy.ori}. 
 A more detailed investigation based on the logarithmic potential will be left to future research.
\medskip

In our asymptotic analysis, we focus on a simple geometric setting  without boundary effects, assuming that $\Om_\ve^\sfI(t)\Subset\Om$ is connected and annular-like (of thickness at most $\sim\ve$), encloses the domain $\Om_\ve^-(t)$, which is supposed to be simply connected, and is separated from $\pa\Om$ by $\Om_\ve^+(t)$. Then, the conditions~\eqref{eq:phn.bc} on the outer boundary $\pa\Om$ are trivially satisfied. 
We henceforth let (cf.\ \Cref{fig:geom}) \[\Gamma_\ve^\pm(t):=\pa\Om_\ve^\sfI(t)\cap\Om_\ve^\pm(t)\qquad\text{and}\qquad\nu_\ve^{\sfI,\pm}:={\nu_\ve^{\sfI}}_{|{\Gamma_\ve^\pm}}.\]
\begin{figure}[h]
	\centering
\includegraphics{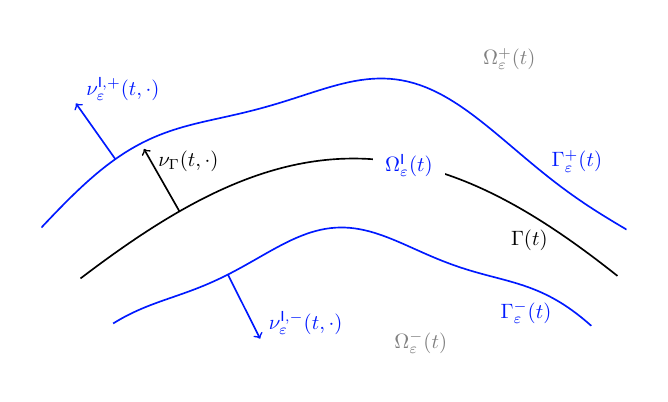}\vspace{-.1cm}
	\caption{Free transition layer $\Om_\ve^{\sfI}(t)$ and sharp interface $\Gamma(t)$.
	}\label{fig:geom}
\end{figure}

\subsection{Formulation in local reference coordinates}\label{sssec:new.coord.toy}

\paragraph{Free boundary.}
Our geometric set-up implies that the moving boundary $\pa\Om_\ve^\sfI(t)$
is composed of two connected components $\Gamma_\ve^\pm(t)$, which are part of the unknowns. 
For $\ve$ small, we can assume that $|d(t,x)|<\ol d$ for all $x\in \overline{\Om_\ve^\sfI(t)}$. Thus, in line with our setting, we may assume that, locally, each of the components $\Gamma_\ve^\pm$ can be written as a graph over $\Gamma$ in the sense that
\[\Gamma_\ve^\pm
\cap 
\mathcal{N}
=\{G^\ve(t,s,Y_\ve^\pm(t,s)): (t,s)\in I{\times}\mathcal{O}\},\qquad 
\mathcal{N}:=G^\ve(I\times\mathcal{O}\times J_\ve),\]
or equivalently
\begin{align}\label{eq:fb.explicit}
	\Gamma_\ve^\pm(t)\cap\mathcal{N}(t)=\{\gamma(t,s)+\ve Y_\ve^\pm(t,s)\nu(t,s) : s\in \mathcal{O}\}, \quad t\in I,
\end{align}
with $\mathcal{N}(t):=\gamma^\ve_t(\mathcal{O}\times J_\ve)$,
where the height functions
\[Y_\ve^\pm:I\times\mathcal{O}\to\mathbb{R}\]
are part of the unknowns.
In the reference coordinates $(s,\rho)$, the transition region $\Om_\ve^\sfI(t)$ then takes the form
\begin{align*}
	(G^\ve(t,\cdot))^{-1}(\Om_\ve^\sfI(t)\cap\mathcal{N}(t))= \{(s,\rho): s\in\mathcal{O}, \rho \in (Y_\ve^-(t,s),Y_\ve^+(t,s))\}, \quad t\in I.
\end{align*}

\paragraph{Equations in the transition layer.}

The transformation rules~\eqref{eq:diffop.traf} allow us to reformulate equations~\eqref{eq:fb.u}--\eqref{eq:fb.w} in terms of $(U,Q,W)=(u,q,w)\circ G^\ve$ as 
\begin{subequations}\label{eq:innerVPS.vn}
		\begin{align}\label{eq:ch.in.full}
			-\ve\pa_\rho U V
			+\ve^2\nabla_sU\cdot\,\pa_t\SL\circ G^\ve+\ve^2\partial_tU&=\ve^{-2}\pa_\rho(m(U)[\pa_\rho W-\frac{1}{n(U)}\pa_\rho(A(U)Q)])
			\\&\quad+\ve^{-1}m(U)[\pa_\rho W-\frac{1}{n(U)}\pa_\rho(A(U)Q)]\Delta_xd\circ G^\ve
			\nonumber\\&\quad +\divGeps(m(U)[\nablaGeps
			W-\frac{1}{n(U)}\nablaGeps(A(U)Q)]),
			\nonumber\\[2mm]\label{eq:q.in.full}
			-\ve\pa_\rho Z V
			+\ve^2\nabla_sZ\cdot\,\pa_t\SL\circ G^\ve+\ve^2\partial_tZ&= - \frac{1}{\tau(U)}Q 
			\\&\quad-\ve^{-2}A(U)m(U)\pa_\rho \big(\frac{1}{n(U)}\big)\,[\pa_\rho W-\frac{1}{n(U)}\pa_\rho(A(U)Q)]
			\nonumber\\&\quad -A(U)m(U)\nablaGeps \big(\frac{1}{n(U)}\big)\cdot[\nablaGeps W-\frac{1}{n(U)}\nablaGeps(A(U)Q)],\nonumber
		\end{align}
		where $Z=Q+R(U)$, $R'=\frac An$, and
		\begin{align}
			W&=-(\pa_\rho^2U+U)-\ve\pa_\rho U \Delta_xd\circ G^\ve-\ve^2\DeltaGeps U.\label{eq:w.in.full}
		\end{align}
\end{subequations}
These equations are to be imposed on $\{(t,s,\rho): \; Y_\ve^-(t,s)<\rho<Y_\ve^+(t,s), \; (t,s)\in I\times\mathcal{O}\}$.

\paragraph{Conditions at the free boundary.}  The continuity conditions~\eqref{eq:match.toy.ori} at the free boundary turn into conditions at $\{\rho=Y_\ve^\pm(t,s)\}$ in the reference coordinates, and take the form
\begin{subequations}\label{eq:match.full}
	\begin{align}\label{eq:ch.m.full}
		\ve^{-1}m(U) (\pa_\rho W-\frac{1}{n(U)}\pa_\rho(A(U)Q))\nu\cdot\nu_\ve^\pm+m(U)\big(\nablaGeps W-\frac{1}{n(U)}\nablaGeps(A(U)Q)\big)\cdot\nu_\ve^\pm=0,\\\label{eq:ell1.m.full}
		U=\pm1,	\\[1mm]
		\ve^{-1}\pa_\rho U\nu\cdot\nu_\ve^\pm+\nablaGeps U\cdot\nu_\ve^\pm=0,\label{eq:ell2.m.full}
	\end{align}
\end{subequations}
where $\nu_\ve^\pm(t,s):=\nu_\ve^{\sfI,\pm}(G_\ve(t,s,Y^\pm_\ve(t,s)))$ denotes the outer unit normal field $\nu_\ve^\sfI(t,\cdot)$ restricted to $\Gamma_\ve^\pm(t)$ in the local coordinates.
The equations~\eqref{eq:match.full} are to be understood in the trace sense.

In view of~\eqref{eq:fb.explicit}, $\nu_\ve^\pm$ is determined by the conditions
\begin{align}\label{eq:cond.nu.pm}
	\nu^\pm_\ve \perp \pa_{s_i}\gamma+\ve Y_\ve^\pm\pa_{s_i}\nu+\ve\pa_{s_i}Y^\pm\nu, \quad i=1,\dots,\dimx-1,\qquad |\nu^\pm_\ve|=1,\quad \pm\nu^\pm_\ve\cdot\nu\ge0.
\end{align}

\subsection{Asymptotic expansions}
We assume the following expansions of the unknowns written in the local reference coordinates $(U,Q,W)(t,s,\rho;\ve)=(u_\ve,q_\ve,w_\ve)\circ G^\ve(t,s,\rho)$ and the height functions $Y_\ve^\pm=Y_\ve^\pm(t,s)$ determining the moving boundary
\begin{align*}
	U(\cdot;\ve)&=\sum_{i\ge0}\ve^iU^i,
	\quad	Q(\cdot;\ve)=\sum_{i\ge0}\ve^iQ^i, 
	\quad	W(\cdot;\ve)=\sum_{i\ge0}\ve^iW^i,
	\\	  Y_\ve^\pm&=\sum_{i\ge0}\ve^iY_\pm^i.
\end{align*}
Thus, in view of~\eqref{eq:cond.nu.pm}, we also have expansions 
\[\nu_\ve^\pm=\nu^0_\pm
+\ve\nu^1_\pm+O(\ve^2),
\qquad \nu^0_\pm=
\pm\nu.
\] 
In particular, $\pm\nu\cdot\nu_\pm^0=1.$
The first-order corrections $\nu^1_\pm$ are determined by 
\[\nu^1_\pm\cdot\pa_{s_i}\gamma=\mp\pa_{s_i}Y^0_\pm\quad\text{for all }i=1,\dots,\dimx-1, \qquad\text{and}\quad \nu^1_\pm\cdot\nu=0.\]
We now insert the above expansions of the dependent variables in the transformed equations~\eqref{eq:innerVPS.vn}, \eqref{eq:match.full}, and then, treating $0<\ve\ll1$ as a small parameter, use Taylor expansions to separate terms of different order. Taking also into account the expansions of the differential operators in~\eqref{eq:diffop.expansion},~\eqref{eq:divgrad.expansion} as well as the identity~\eqref{eq:d.Lapl}, we then sort by orders of $\ve$. This leads to a hierarchy of linear equations for the higher-order corrections. 
Our main focus is the formal derivation of the interface evolution laws, which will
emerge at `third order'.

Before we start, let us briefly illustrate the expansion procedure for
$g(Y_\ve^\pm):=g(t,s,Y_\ve^\pm(s,t))$ 
assuming that $\rho\mapsto g(t,s,\rho)$ is smooth enough. With the above ansatz $Y_\ve^\pm=\sum_{i\ge0}\ve^iY_\pm^i$, for $0<\ve\ll1$, formal Taylor expansion of $\rho\mapsto g(t,s,\rho)$ yields
\begin{align*}g(Y_\ve^\pm)=g(Y_\pm^0)&+\ve\pa_\rho g(Y_\pm^0) Y_\pm^1
	+\ve^2\big(\pa_\rho g(Y_\pm^0)Y_\pm^2+\frac12\pa_\rho^2g(Y^0_\pm)(Y^1_\pm)^2\big)
	+O(\ve^3).
\end{align*}

\paragraph{Leading order.}Transition layer: $O(\ve^{-2}),  O(\ve^{-2}), O(1)$.  Continuity conditions: $O(\ve^{-1}), O(1),  O(\ve^{-1})$.
The starting point in the hierarchy is to assume that $W^0=0$, which can be interpreted as a quasi-stationarity condition and leads to an asymptotic analysis that is consistent with the present continuity conditions at the free boundary. 
Given this hypothesis, the leading order equations are
\begin{subequations}\label{eq:0.full}
	\begin{align}
		0&=-\pa_\rho\Big(\frac{m(U^0)}{n(U^0)}\pa_\rho (A(U^0)Q^0)\Big),\label{eq:0.full.ch}
		\\0&=A(U^0)\frac{m(U^0)}{n(U^0)}\pa_\rho \big(\frac{1}{n(U^0)}\big)\pa_\rho(A(U^0)Q^0),
		\label{eq:0.full.q}
		\\0&=-\pa_\rho^2U^0-U^0.\label{eq:0.full.w}
	\end{align}
\end{subequations}
These equations are imposed for $\rho\in(Y_-^0,Y_+^0)=:J$ and are supplemented by the leading order equations of~\eqref{eq:match.full}, to be understood in the trace sense,
\begin{subequations}\label{eq:0.m}
	\begin{align}
		-\frac{m(U^0)}{n(U^0)}\pa_\rho(A(U^0)Q^0) \,\nu\cdot\nu_\pm^0=0\quad \text{on }\{\rho=Y_\pm^0\},\label{eq:0.m.ch}\\
		\label{eq:0.m.ell1.toy}
		U^0=\pm1 \quad \text{on }\{\rho=Y_\pm^0\},\\
		\pa_\rho U^0 \,\nu\cdot\nu_\pm^0=0\quad \text{on }\{\rho=Y_\pm^0\}.\label{eq:0.m.ell2.toy}
	\end{align}
\end{subequations}
We first consider the problem for $U^0$. To this end,  recall that our hypothesis that the zero level sets of $\{U(\cdot;\ve)\}_\ve$ converge to $\Gamma$, i.e.\ to $\{\rho=0\}$,  enforces $U^0_{|\rho=0}=0$.
Combining this condition with  equations~\eqref{eq:0.full.w},~\eqref{eq:0.m.ell1.toy},~\eqref{eq:0.m.ell2.toy}, and recalling that $\nu\cdot\nu_\pm^0=\pm1$, yields the following discrete family of solutions 
\[U^0(t,s,\rho)=\sin\rho,\qquad Y_\pm^0(t,s,\rho)=\pm\pi\big(\tfrac{1}{2}+k_\pm(t,s)\big)\]
with free parameters $k_\pm(t,s)\in 2\mathbb N_0$.
Since we are assuming that $U^0=U^0(t,s,\rho), Y_\pm^0=Y_\pm^0(t,s)$ are continuous functions of their arguments, the discreteness of this family of possible solutions 
implies that $U^0$ and $Y_\pm^0$ must be constant in $(t,s)$, and hence that $k_\pm\in 2\mathbb N_0$ are fixed.
So far, our reasoning is the same as in~\cite{CEN_1996}. 	In order to extract a `unique' solution from this discrete family, one needs to impose a selection criterion. 
One possibility is to impose monotonicity of the profile $U^0$ as an extra property, as was done in~\cite{CEN_1996}.
A more fundamental, and thus perhaps preferrable criterion is that of energy minimality, see~\cite[Theorem~5]{CNH_2024}. 
Under either of these two criteria, one ends up with the same unique solution triple $(U^0,Y^0_-,Y^0_+)$ given by 
\begin{align}\label{eq:U0Y0.sol.toy}
	U^0(\rho)=\sin\rho,\qquad \rho \in J=(Y_-^0,Y_+^0),\quad Y^0_\pm=\pm\frac\pi2.
\end{align}
This further entails $\nabla_sY^0_\pm\equiv0$, and therefore $\nu^1_\pm\equiv0$.

Equation~\eqref{eq:0.full.ch} implies that 
$\frac{m(U^0)}{n(U^0)}\pa_\rho(A(U^0)Q^0)=c_0$ in $J$ for a function $c_0=c_0(t,s)$ that is independent of $\rho$. Invoking~\eqref{eq:0.m.ch},  we deduce that $c_0\equiv0$, and hence
$\frac{m(U^0)}{n(U^0)}\pa_\rho(A(U^0)Q^0)=0$ in $J$.
  Since $\frac{m(U^0)}{n(U^0)}\neq0$ for all $\rho\in(Y_-^0,Y_+^0)$ (cf.~\ref{hp:m.full},~\ref{hp:nm}), we infer that $\pa_\rho(A(U^0)Q^0)=0.$ Consequently,  
\begin{align*}
	A(U^0)Q^0= a_0,\quad \text{  where }a_0=a_0(t,s).
\end{align*}

\paragraph{First order.} 
Transition layer: $O(\ve^{-1}),  O(\ve^{-1}), O(\ve)$. Continuity conditions: $O(1), O(\ve),  O(1)$. 
The bulk equations at first order are imposed for $\rho\in J$
\begin{subequations}
	\begin{align}\label{eq:lead.inner.CH}
		0&=\pa_\rho(m(U^0)E_1),
		\\0&=-A(U^0)m(U^0)\pa_\rho \big(\frac{1}{n(U^0)}\big)E_1,\label{eq:lead.inner.q}
		\\W^1&=-\pa_\rho^2U^1-U^1+\pa_\rho U^0\kap_\gamma,\label{eq:first.n.w}
	\end{align}
\end{subequations}
where 
\begin{align*}
	E_1:=\pa_\rho W^1-\frac{1}{n(U^0)}\pa_\rho(A(U^0)Q^1+A'(U^0)U^1Q^0).
\end{align*}
They are supplemented by the appropriate continuity conditions stemming from~\eqref{eq:match.full}
\begin{subequations}
	\begin{align}
		&(m(U^0)E_1)_{|\rho=Y^0_\pm}=0,
		\label{eq:match.Flux0}\\
		\label{eq:match.U0}
		&U^1_{|\rho=Y^0_\pm}=0,\quad \big(\pa_\rho U^1+\pa_\rho^2U^0Y_\pm^1\big)_{|\rho=Y^0_\pm}=0.
	\end{align}
\end{subequations}
Here, for equation~\eqref{eq:match.Flux0}, we used the orthogonality $\nu\cdot\nabla_\gamma\equiv0$.

Equations~\eqref{eq:lead.inner.CH},~\eqref{eq:match.Flux0} imply that  
$m(U^0)E_1\equiv0$,
and thus, since $m(U^0)>0$ in $J$,
\begin{align}\label{eq:vps.zeroth.prelim}
	E_1=0.
\end{align}
This also means that~\eqref{eq:lead.inner.q} is trivially satisfied.

We next consider the linear elliptic Dirichlet problem~\eqref{eq:first.n.w},~\eqref{eq:match.U0} in $\rho$ for $U^1$ with `right-hand side' data $r^1:=W^1-\pa_\rho U^0\kap_\gamma$. By elliptic theory (cf.~\cite[Chapter~8]{GT_2001}), solvability of this problem is ensured if and only if $r^1$ is $L^2(J)$-orthogonal to the kernel of the elliptic operator $-\pa_\rho^2-\text{Id}$, which is spanned by $\pa_\rho U^0$.  This leads to the solvability condition $(\pa_\rho U^0,W^1)_{L^2(J)}-\|\pa_\rho U^0\|_{L^2(J)}^2\kap_\gamma=0$.
Abbreviating $\sigma:=\int_J(\pa_\rho U^0)^2\,\dd\rho$, it becomes
\begin{align}\label{eq:condW1.v}
	\int_J W^1\pa_\rho U^0\,\dd\rho = \sigma \,\kappa_\gamma.
\end{align}

Let us also note that the second equation in~\eqref{eq:match.U0} combined with~\eqref{eq:U0Y0.sol.toy} determines $Y_\pm^1$ in terms of $U^1$ via
\begin{align*}
	Y_\pm^1=\pm \pa_\rho U^1_{|\rho=Y^0_\pm}.
\end{align*}
Since the actual values of the higher-order corrections $Y^i_\pm$, $i\ge1$, will not be needed directly for our purpose, we will not explicitly consider~\eqref{eq:ell2.m.full} at the subsequent higher orders. 

\paragraph{Second order.} 
Transition layer: $O(1), O(1), O(\ve^2)$. Continuity conditions: $O(\ve), O(\ve^2), O(\ve)$.
Using~\eqref{eq:vps.zeroth.prelim}, we obtain the equations
\begin{subequations}
	\begin{align}
		0&=\pa_\rho(m(U^0)E_2)-\frac{m(U^0)}{n(U^0)}\DeltaG a_0,\label{eq:first.n.u}
		\\	0&= 
		-\frac{1}{\tau(U^0)}Q^0 -
		A(U^0)m(U^0)\pa_\rho(\frac{1}{n(U^0)})E_2,
		\label{eq:2.n.q}
		\\W^2&=-\pa_\rho^2U^2-U^2+\pa_\rho U^1\kappa_\gamma
		+\pa_\rho U^0\rho|\mathcal{W}_\gamma|^2,
		\label{eq:2.in.ell.full}
	\end{align}
\end{subequations}
where 
\begin{align*}
	E_2:=\pa_\rho W^2&-\frac{1}{n(U^0)}\pa_\rho\big(A(U^0)Q^2+A'(U^0)U^1Q^1+(A'(U^0)U^2+\frac12A''(U^0)(U^1)^2)Q^0\big)
	\\&+\frac{n'(U^0)}{n(U^0)^2}U^1\pa_\rho(A(U^0)Q^1).
\end{align*}
Due to~\eqref{eq:vps.zeroth.prelim}, the fact that $\nu_\ve^\pm=\pm\nu+O(\ve^2)$, and thanks to the orthogonality relation $\nu\cdot\nablaGeps\equiv0$ (in its expanded form:
 $\nu\cdot\nablaG\equiv0$, $\nu\cdot\mrem^i\equiv0$ with $\mrem^i$ as in~\eqref{eq:diffop.expansion}), the continuity condition associated to~\eqref{eq:ch.m.full}
states
\begin{align}\label{eq:2.m.ch.full.s}
	(m(U^0)E_2)_{\rho=Y^0_\pm} =0.
\end{align} 
Equation~\eqref{eq:ell1.m.full} at the relevant order states $U^2_{|\rho=Y^0_\pm}=0$,   thus complementing~\eqref{eq:2.in.ell.full}. 

Owing~\eqref{eq:2.m.ch.full.s}, integration of~\eqref{eq:first.n.u} over $\rho\in J$ 
 implies that 
\begin{align*}
	-\int_J\frac{m(U^0)}{n(U^0)}\,\dd\rho \;\DeltaG a_0(t,s)=0.
\end{align*}
Since $\frac{m(U^0)}{n(U^0)}$ has a sign (cf.\ \ref{hp:nm}) and hence $\int_J\frac{m(U^0)}{n(U^0)}\,\dd\rho\not=0$, we deduce that $-\DeltaG a_0\equiv0$.
This, in turn, combined with~\eqref{eq:first.n.u} and~\eqref{eq:2.m.ch.full.s} yields
\begin{align}\label{eq:101.vn}
	E_2=0.
\end{align}
Inserting~\eqref{eq:101.vn} into~\eqref{eq:2.n.q} and using the finiteness of $\tau$ (cf.\ hypothesis~\ref{hp:tauA.reg}), 
we thus arrive at
\begin{align*}
	Q^0=0.
\end{align*}
Equation~\eqref{eq:vps.zeroth.prelim} therefore becomes
\begin{align}\label{eq:vps.zeroth.vn}
	\pa_\rho W^1-\frac{1}{n(U^0)}\pa_\rho(A(U^0)Q^1)=0.
\end{align}

\paragraph{Third order.} 
Transition layer: $O(\ve), O(\ve), O(\ve^3)$. Continuity conditions: $O(\ve^2), O(\ve^3), O(\ve^2)$.
Using~\eqref{eq:vps.zeroth.vn} and~\eqref{eq:101.vn}, the equations~\eqref{eq:ch.in.full} and~\eqref{eq:q.in.full} at order $O(\ve)$ can be cast in the form
\begin{subequations}
\begin{align}\label{eq:2nd.vn.u*}
	-\pa_\rho U^0V&=\pa_\rho(m(U^0)E_3)+\DeltaG(m(U^0)f),
	\\\label{eq:2nd.vn.q*}	-\pa_\rho R(U^0)V&=-\frac{1}{\tau(U^0)}Q^1
	-A(U^0)\pa_\rho\big(\frac{1}{n(U^0)}\big)\, (m(U^0)E_3),
\end{align}
with $R'=\frac An$, and where we introduced
\begin{align}\label{eq:def.f}
	f:=W^1-\frac{1}{n(U^0)}A(U^0)Q^1
\end{align}
and 
\begin{align*} E_3:=
	&\pa_\rho W^3-\frac{1}{n(U^0)}\pa_\rho\big(A(U^0)Q^3+A'(U^0)U^1Q^2+(A'(U^0)U^2+\frac{1}{2}A''(U^0)(U^1)^2)Q^1\big)
	\\&+\frac{n'(U^0)}{n(U^0)^2}U^1\pa_\rho\big(A(U^0)Q^2+A'(U^0)U^1Q^1\big)
	+\bigg(\frac{n'(U^0)}{n(U^0)^2}U^2-\frac{1}{2}\left(\frac{1}{n}\right)''_{|u=U^0}(U^1)^2\bigg)	\pa_\rho(A(U^0)Q^1).
\end{align*}
For later use, we observe that
$W^1$ and $Q^1$ are uniquely determined by $f$ through the linear system
\begin{align}\label{eq:W1.f}
	W^1&=f+\frac{1}{n(U^0)} A(U^0)Q^1,
	\\\pa_\rho f&=-\pa_\rho\big(\frac{1}{n(U^0)}\big) A(U^0)Q^1
	=\frac{n'(U^0)}{n(U^0)^2}A(U^0)Q^1\pa_\rho U^0,
	\label{eq:Q1.f}
\end{align}
where in~\eqref{eq:Q1.f} we used \eqref{eq:vps.zeroth.vn} to find that $\pa_\rho f=\pa_\rho W^1-\pa_\rho(\frac{1}{n(U^0)} A(U^0)Q^1)=-\pa_\rho(\frac{1}{n(U^0)})A(U^0)Q^1$.

For completeness, we note that the equation coming from~\eqref{eq:w.in.full} states
	\begin{align*}
		W^3=-\pa_\rho^2U^3-U^3+\pa_\rho U^2\kappa_\gamma
	+\pa_\rho U^1\rho|\mathcal{W}_\gamma|^2+\pa_\rho U^0\rho^2 k_3^3.
\end{align*}
It is supplemented by $U^3_{|\rho=Y^0_\pm}=0$, which stems from the continuity condition~\eqref{eq:ell1.m.full}.

The continuity condition associated to~\eqref{eq:ch.m.full} at $O(\ve^2)$ states
\begin{align}\label{eq:3.m.ch.full.s}
	(m(U^0)E_3)_{\rho=Y^0_\pm}=0.
\end{align}
\end{subequations}

To proceed with the equations at `third order', we need to distinguish between constant coupling $n\equiv1$ and functions $n$ satisfying the complementary hypothesis~\ref{hp:n.ndeg}. 
In the remaining part of the asymptotic expansions,  we will focus on identifying the equations that 
determine the interface evolution law. 

\subsubsection{Third order for \texorpdfstring{$n\equiv 1$}{}}\label{sssec:isd}

In this paragraph, we consider the setting of \Cref{ass:isd}. In particular, we let $n\equiv1$. In this case, the identity~\eqref{eq:vps.zeroth.vn} implies that $f=f(t,s)$ is independent of $\rho$.
Thus, using~\eqref{eq:3.m.ch.full.s} and integrating~\eqref{eq:2nd.vn.u*} over $\rho\in J$, yields
\begin{align}\label{eq:V0.v}
	V=	-\frac{2}{\delta}\DeltaG f,
\end{align}
where $\delta=4(\int_Jm(U^0)\,\dd\rho)^{-1}=4(\int_{-1}^{+1}\frac{m(u)}{\sqrt{1-u^2}}\,\dd u)^{-1}$ because of $\pa_\rho U^0=\sqrt{1-(U^0)^2}$, $U^0(\rho)=\sin\rho$.

We now turn to~\eqref{eq:2nd.vn.q*}, which for $n\equiv1$ reduces to
\[	-\pa_\rho U^0A(U^0)V=-\frac{1}{\tau(U^0)}Q^1.\]
Multiplying this equation by $\tau(U^0) A(U^0)$ and substituting $W^1-f$ for $A(U^0)Q^1$ (cf.~\eqref{eq:def.f}) yields
\begin{align}\label{eq:Z.vel.1}
	\pa_\rho U^0\tau(U^0) A(U^0)^2\, V &=W^1-f.
\end{align}
We multiply~\eqref{eq:Z.vel.1} by $\pa_\rho U^0$ and integrate over $\rho\in J$. 
Combined with~\eqref{eq:V0.v} and~\eqref{eq:condW1.v}, this gives
\begin{align}\label{eq:-Lapl+1.f}
	-\frac{2\omega}{\delta}\DeltaG f+\bbU f=\sigma\kappa_\gamma,
\end{align}
where $\omega=\int_JA(U^0)^2\tau(U^0)(\pa_\rho U^0)^2\,\dd\rho=\int_{-1}^{+1}A(u)^2\tau(u)\sqrt{1{-}u^2}\,\dd u$ and 
$\bbU:=U^0(Y^0_+){-}U^0(Y^0_-)=2$. 

For a smooth closed hypersurface and any $\hat\om>0$, the linear operator 
$\sff\mapsto-\hat\om\Delta_\Gamma \sff+ \sff$
induces an isomorphism from $H^2(\Gamma)$ onto $L^2(\Gamma)$.
Hence, in global notation, equations~\eqref{eq:V0.v},~\eqref{eq:-Lapl+1.f} amount to the interface evolution law
\begin{align*}
	\sfV_\Gamma=-\sigma\Delta_\Gamma(\delta\,\textrm{Id}-\omega\Delta_\Gamma)^{-1}\kappa_\Gamma,
\end{align*}
where we recall that $\sfV_\Gamma,\kappa_\Gamma:\Gamma\to\mathbb{R}$ denote the normal velocity resp.\ the mean curvature of $\Gamma$.

\subsubsection{Third order for non-constant coupling  \texorpdfstring{$n$}{}}\label{sssec:fsd}

Here, we consider the setting of \Cref{ass:fsd}. To solve equation~\eqref{eq:2nd.vn.u*} for $m(U^0)E_3$, we integrate over $(-\frac\pi2,\rho)$  and use~\eqref{eq:3.m.ch.full.s} to deduce
\begin{align*}
	m(U^0)E_3=-(U^0+1)V-\DeltaG\int_{-\pi/2}^\rho m(U^0)f\dd\rho'.
\end{align*}
Inserted in~\eqref{eq:2nd.vn.q*}, this gives
\begin{multline}\label{eq:mF.ins}
	\bigg(	-A(U^0)\pa_\rho\Big(\frac{1}{n(U^0)}\Big)(U^0+1)	-\pa_\rho R(U^0)\bigg)V
	\\=-\frac{1}{\tau(U^0)}Q^1
	+A(U^0)\pa_\rho\Big(\frac{1}{n(U^0)}\Big)\, \bigg(\DeltaG\int_{-\pi/2}^\rho m(U^0) f\dd\rho'\bigg).
\end{multline}
Owing to hypothesis~\ref{hp:n.ndeg}, we may divide~\eqref{eq:mF.ins} by $A(U^0)\pa_\rho(\frac{1}{n(U^0)})$. We then recall~\eqref{eq:Q1.f} 
 to substitute $\frac{1}{A(U^0)\pa_\rho(\frac{1}{n(U^0)})}\pa_\rho f$ for $-Q^1$. After multiplying the resulting equation by 
$-\frac{1}{A(U^0)\pa_\rho(\frac{1}{n(U^0)})}$ and computing \[\frac{\frac{A(U^0)}{n(U^0)}\pa_\rho U^0}{A(U^0)\pa_\rho(\frac{1}{n(U^0)})}=-\frac{n(U^0)}{n'(U^0)},\] we deduce 
\begin{align*}
\Big((U^0+1)-\frac{n(U^0)}{n'(U^0)}\Big)
	 V=-\frac{1}{A(U^0)^2\tau(U^0)}\big(\frac{n(U^0)^2}{\pa_\rho n(U^0)}\big)^2\pa_\rho f
	-\DeltaG \int_{-\pi/2}^\rho m(U^0) f\,\dd\rho'.
\end{align*}
Upon differentiation in $\rho$, we arrive at the equation
\begin{subequations}\label{eq:problem.fV}
{\begin{align}\label{eq:f1.vn}
		\boxed{\mathfrak{L}f:=-\pa_\rho\big(
			a(U^0)\pa_\rho f\big)
			-m(U^0)\DeltaG f
			=\Big(\pa_\rho U^0{-}\pa_\rho\big(\frac{n(U^0)}{n'(U^0)}\big)\Big) V\quad\text{ in }\;\{-\tfrac\pi2<\rho<\tfrac\pi2\},}
\end{align}}%
where we abbreviated (cf.~\ref{hp:am})
\[a(u):=\bigg(\frac{1}{A^2\tau}\Big(\frac{n^2}{n'}\Big)^2\bigg)_{|u}\frac{1}{1-u^2},\]
and used the fact that $(\pa_\rho U^0)^2=1-(U^0)^2$.

In order to identify the boundary conditions for $f$ at $\{\rho=\pm \frac\pi2\}$ that supplement
equation~\eqref{eq:f1.vn}, we subtract~\eqref{eq:2nd.vn.u*} from~\eqref{eq:f1.vn}, simplify, and rearrange terms to find
\begin{align*}
	\pa_\rho\big(-a(U^0)\pa_\rho f+\frac{n(U^0)}{n'(U^0)}V+m(U^0)E_3\big)=0.
\end{align*}
Hence, there exists $c_1=c_1(t,s)$, independent of $\rho$, such that $-a\pa_\rho f+\frac{n(U^0)}{n'(U^0)}V+m(U^0)E_3=c_1$. 
Inserting $m(U^0)E_3=c_1-(-a\pa_\rho f+\frac{n(U^0)}{n'(U^0)}V)$ into~\eqref{eq:2nd.vn.q*}, and substituting $\frac{1}{A(U^0)\pa_\rho\big(\frac{1}{n(U^0)}\big)}\pa_\rho f$
for $-Q^1$ in~\eqref{eq:2nd.vn.q*}, we deduce, upon rearranging terms, that $c_1=0$. 
Owing to~\eqref{eq:3.m.ch.full.s}, we thus arrive at the boundary conditions 
 \begin{align}\label{eq:bc.f}
 \boxed{-a(U^0)\pa_\rho f=-\frac{n(U^0)}{n'(U^0)}V\quad\text{ on }\;\{\rho=\pm\tfrac\pi2\}.}
 \end{align}
We next formulate the solvability condition~\eqref{eq:condW1.v} in terms of $f$, using~\eqref{eq:W1.f}, \eqref{eq:Q1.f}. This gives
\begin{align*}
	\int_J W^1\pa_\rho U^0\,\dd\rho
	=\int_J\big( \pa_\rho U^0f+\frac{n(U^0)}{n'(U^0)}\pa_\rho f\big)\,\dd\rho.
\end{align*}
Hence, condition~\eqref{eq:condW1.v} takes the form
{\begin{align}\label{eq:constraint.f.vn}
		\boxed{	\mathfrak{C}f:= \int_J\big( \pa_\rho U^0f+\frac{n(U^0)}{n'(U^0)}\pa_\rho f\big)\,\dd\rho
			=\sigma\kap_\gamma.}
\end{align}}%
\end{subequations}
Note that $D\mathfrak{C}(f)= \pa_\rho U^0-\pa_\rho\big(\frac{n(U^0)}{n'(U^0)}\big)$ in the sense of distributions. Hence, the velocity field $V=V(t,s)$ on the right-hand side of~\eqref{eq:f1.vn},~\eqref{eq:bc.f}, which is independent of $\rho$, arises as the Lagrange multiplier associated to the constraint~\eqref{eq:constraint.f.vn}.
In order to derive the geometric evolution law, we are thus left to determine the couple $(f,V)$ satisfying the equations~\eqref{eq:problem.fV}.

\section{Well-posedness of the constrained elliptic problem}
\label{sec:constr-elliptic}
In this section, we rigorously establish the existence, uniqueness, and regularity of a solution $(f,V)$ to~\eqref{eq:problem.fV} by recasting the equations as a variational problem for sufficiently regular closed connected embedded hypersurfaces $\Gamma$. We further derive an abstract formula for the interface evolution law by identifying the operator that maps given curvature data $\kappa$ to the normal velocity $V$.

\subsection{Notation and hypotheses}\label{ssec:notations.constr-ell}
The problem in this section being purely spatial, we here drop any temporal dependence and write $\gamma:\mathcal{O}\to\Gamma$, $f=f(s,\rho)$, $s\in\mathcal{O},\rho\in J:=(-\frac\pi2,\frac\pi2)$, etc. 
As our analysis of problem~\eqref{eq:problem.fV} is essentially independent of the preceding formal asymptotics, let us separately formulate a relaxed set of hypotheses on the (time-independent) geometry $\Gamma$ and the coefficients $m,n,A^2\tau$ that suffices for the analysis of the present section. 

\paragraph{Hypotheses.}
\begin{enumerate}[label=(e\arabic*)]
	\item\label{hp:e.geom} $\Gamma\Subset \mathbb{R}^d$ is a smooth, closed (incl.\ compact), connected, embedded hypersurface
	\item\label{hp:e.coeff} $m(u)=(1-u^2)^i\tilde m(u)$ for some $i\in \mathbb{N}$, where $\tilde m\in C^\infty([-1,1])$ with $\min_{[-1,1]} \tilde m>0$; \\$ n,A^2\tau\in C^\infty([-1,1])$, $n,n'\not=0$ a.e.,  and \ref{hp:tauA.reg}
	\item\label{hp:tauA^2n.pos} 
	$\underline\iota:=\inf_{(-1,1)}\frac{n^2(u)}{(A^2\tau)(u)\sqrt{1-u^2}}>0$.
\end{enumerate}

\paragraph{Global coordinates.}
For the variational arguments below, it is natural to formulate the problem globally in terms of unknowns $\sff:\Gamma\times [-1,1]\to\mathbb{R}$ and $\sfV:\Gamma\to\mathbb{R}$, which will then yield the local solution $(f,V)$ to~\eqref{eq:problem.fV} for $(s,\rho)\in \mathcal{O}\times J$ (at a fixed time $t$) via 
\[	f(s,\rho)=\sff(\sfs,u),\quad 	V(s)=\sfV(\sfs)\qquad \text{ with }\; (\sfs,u)=(\gamma(s),U^0(\rho)),\quad s\in \mathcal{O}, \rho\in J,\]
where we recall that $U^0(\rho)=\sin\rho$.
Here, $\gamma=\gamma(t,\cdot):\mathcal{O}\subset\mathbb{R}^{d-1}\to \Gamma(t)$ stands for any of the local parametrisations of the evolving hypersurface, evalued at time $t$. 
Note that  for differentiable functions $\sfg=\sfg(\sfs,u)$, due to $\pa_\rho U^0=\sqrt{1-(U^0)^2}$ and the definition of $\Delta_\gamma$ (cf.\ \Cref{sec:prelim}),
\begin{align*}
\frac{1}{\sqrt{1-(U^0)^2}}\frac{\pa}{\pa\rho}\sfg(\sfs,U^0(\cdot))&=(\pa_u\sfg)(\sfs,U^0(\cdot)),
	\\\DeltaG \sfg(\gamma(\cdot),u)&=(\Delta_\Gamma {\sf g})(\gamma(\cdot),u).
\end{align*}
Hence, in the $(\sfs,u)$-coordinates, problem~\eqref{eq:problem.fV} takes the form
\begin{subequations}\label{eq:u-var.key-problem}
	\begin{alignat}{3}
		-\pa_u(\sfa\pa_u \sff)-\sfm\Delta_\Gamma \sff &= \big(1-\pa_u(\frac{n}{n'})\big)\sfV \qquad&&\text{ in }\Gamma\times[-1,1],\label{eq:u-var.elliptic}
		\\-\sfa\pa_u \sff &= - \frac{n}{n'}\sfV&&\text{ on }\Gamma\times\{\pm1\},
		\label{eq:u-var.bc}
		\\ \int_{-1}^{+1}(\sff+\frac{n}{n'}\pa_u\sff)\,\dd u&=\sigma\kappa&& \text{ on }\Gamma\label{eq:u-var.constraint}
	\end{alignat}
\end{subequations}
with $\kappa=\kappa_\Gamma$ and $\sfV=\sfV_\Gamma$, where
here and in the rest of this manuscript, we adopt the notation
\begin{subequations}\label{eq:def.sfa-sfm}
	\begin{align}\label{eq:def.a(u)}
		&\sfa(u):=\Big(\frac{1}{A^2\tau}\big(\frac{n^2}{n'}\big)^2\Big)_{|u}\,\frac{1}{\sqrt{1-u^2}},
		\\\label{eq:def.m(u)/}
		&\sfm(u):=\frac{m(u)}{\sqrt{1-u^2}}.
	\end{align}
\end{subequations}
Observe that~\ref{hp:tauA^2n.pos} implies the bound
\begin{align}\label{eq:e3LC.an}
\sfa\ge\ul\iota\,\Big(\frac{n}{n'}\Big)^2.
\end{align}
This will ensure compatibility of the constraint~\eqref{eq:u-var.constraint} with the functional setting induced by the elliptic operator in~\eqref{eq:u-var.elliptic}.

Since we are interested in determining the propagation operator inducing the interface dynamics, we will develop the well-posedness theory for general functions $\kappa:\Gamma\to\mathbb{R}$, a priori not equal to the mean curvature $\kappa_\Gamma$ of $\Gamma$. 
We will always assume that $\kappa\in\Hd$.

\paragraph{Surface divergence theorem.}

Let us briefly recall the following integration-by-parts formula for sufficiently regular functions $\sff,\mathsf{g}:\Gamma\to\mathbb{R}$ 
\begin{align*}
	\int_{\Gamma}\nabla_\Gamma \sff\cdot\nabla_\Gamma \mathsf{g}\,\dd\mathcal{H}^{d-1} 
	= \int_{\Gamma}(-\Delta_\Gamma \sff)\, \mathsf{g}\,\dd\mathcal{H}^{d-1},
\end{align*}
which is a consequence of the surface divergence theorem on $\Gamma$ for tangential vector fields. This formula will be used below without explicit mention.

\subsection{Function spaces}\label{sssec:spaces}

Let 
\[\mathcal{C}\sff(\sfs)=\int_{-1}^{+1}\big(\sff(\sfs,u)+\frac{n(u)}{n'(u)}\pa_u\sff(\sfs,u)\big)\,\dd u,\]
whever the integral converges.
Then, define the space
\[\hat H:=\left\{\sff\in C^\infty(\Gamma\times [-1,1]):\,
\sqrt{\sfa}\pa_u \sff\in L^2(\Gamma\times [-1,1]),\; \mathcal{C}\sff\in\Hd
\right\}.
\]
Note that, due to~\eqref{eq:e3LC.an},  
\begin{align*}
	\int_{\Gamma}\int_{[-1,1]} 
	\sfa |\pa_u \sff|^2\,\dd u\,\dd \mathcal{H}^{d-1}\ge \ul\iota\int_{\Gamma}\int_{[-1,1]} \left|\frac{n}{n'}\pa_u \sff\right|^2\dd u\,\dd\mathcal{H}^{d-1}.
\end{align*}
Consequently, $\frac{n}{n'}\pa_u \sff\in L^2(\Gamma\times [-1,1])\subset L^1(\Gamma\times [-1,1])$ for $\sff\in C^\infty(\Gamma\times[-1,1])$ with $\sqrt{\sfa}\pa_u \sff\in L^2(\Gamma\times [-1,1])$, showing that the integral $\mathcal{C}\sff=\int_{-1}^{1}\big(\sff+\frac{n}{n'}\pa_u \sff\big)\,\dd u$ is well-defined a.e.\ in $\Gamma$. Thus, the space $\hat H$ is well-defined.

For $\sff,\sfg\in \hat H$ let
\begin{align*}
	\big(\sff,\sfg\big)_{\mathcal{E}}:=\int_{\Gamma}\int_{[-1,1]} \Big(
	\sfa\pa_u \sff\,\pa_u \sfg
	+\sfm\nabla_\Gamma \sff\cdot \nabla_\Gamma \sfg\Big)\,\dd u\,\dd \mathcal{H}^{d-1},
\end{align*}
and 
\begin{align*}
	(\sff,\sfg)_{H}:=(\sff,\sfg)_{\mathcal{E}}
	+(\mathcal{C}\sff,\mathcal{C}\sfg)_{\Hd}.
\end{align*}
The non-negative bilinear form $(\cdot,\cdot)_{H}$ defines an inner product on the space $\hat H$. To see the definiteness, suppose that $(\bar \sff,\bar \sff)_{H}=0$ for some $\bar \sff\in \hat H$.
Since $\sfm,\sfa$ are positive a.e.\ in $[-1,1]$, this implies that $\nabla_\Gamma\bar \sff=0$, $\pa_u\bar f=0$, and hence $\bar \sff\equiv c$ for a fixed constant $c\in\mathbb{R}$. Thus, $2c=\int_{-1}^{1}\bar \sff\,\dd u=\mathcal{C}\bar \sff=0$.
Hence $c=0$, showing the definiteness. 

We now define the Hilbert space $H$ as the completion of $\hat H$ with respect to $\|\cdot\|_H:=(\cdot,\cdot)_H^{1/2}$.
Furthermore, given $\kappa\in \Hd$, we let
\begin{align*}
	M_\kappa = \left\{\sff\in H: \;\mathcal{C}\sff=\sigma\kap\right\}.
\end{align*}
The set $M_\kappa$ is non-empty (since the function $\sff(\sfs,u)\equiv\frac\sigma2\kappa(\sfs)$ lies in $M_\kappa$) and forms an affine subspace of $H$.
Furthermore, due to $\|\mathcal{C}\sff\|_{\Hd}\le \|\sff\|_H$, the linear operator $\mathcal{C}:H\to \Hd$ is continuous, which implies that $M_\kap\subset H$ is closed.

\subsection{Variational characterisation and interface dynamics}\label{sssec:varchar}

For $\sff\in H$ define the quadratic functional
\begin{align*}
\mathcal{E}(\sff)=\frac12\int_{\Gamma}\int_{[-1,1]}\Big(
\sfa(\pa_u \sff)^2
+\sfm|\nabla_\Gamma \sff|^2\Big)\,\dd u\,\dd \mathcal{H}^{d-1},
\end{align*}
i.e.\ $\mathcal{E}(\sff)=\frac12(\sff,\sff)_{\mathcal{E}}$.

Consider the minimisation problem of $\mathcal{E}$ on $M_\kap$: find $\sff\in M_\kap$ such that 
\begin{align}\label{eq:c-min.f.vn}
\mathcal{E}(\sff)=\inf_{\tilde\sff\in M_\kap}\mathcal{E}(\tilde\sff).
\end{align}
The Lagrangian $L:H\times\Hd^*\to\mathbb{R}$ associated to~\eqref{eq:c-min.f.vn} is given by \[L(\sff,\sfV)=\mathcal{E}(\sff)-\la\sfV,\mathcal{C}\sff-\sigma\kap\ra_{\Hd^*,\Hd}.\] At any critical point $(\sff,\sfV)$ it holds that $\pa_\sff L(\sff,\sfV)=0$, $\pa_\sfV L(\sff,\sfV)=0$. Hence,
\begin{align*}
D\mathcal{E}(\sff)-\la\sfV,\mathcal{C}\cdot\ra_{\Hd^*,\Hd}=0\quad\text{ in }H^*,
\end{align*}
which is the appropriate weak formulation of~\eqref{eq:u-var.elliptic},~\eqref{eq:u-var.bc} and 
\begin{align*}
	\mathcal{C}\sff-\sigma\kap=0\quad\text{ in }\Hd,
\end{align*}
which specifies~\eqref{eq:u-var.constraint}. 
Thus, the system~\eqref{eq:u-var.key-problem} are
 the Euler--Lagrange equations $DL_{|(\sff,\sfV)}=0$ of~\eqref{eq:c-min.f.vn}. 
 We formalise these observations in the following proposition.

\begin{proposition}\label{prop:solveP.vn}
Assume hypotheses~\ref{hp:e.geom},~\ref{hp:e.coeff}, and~\ref{hp:tauA^2n.pos}.
	Given a function $\kap\in \Hd$,
	there exists a unique  couple $(\sff,\sfV)\in M_\kap\times \Hd^*$ solution to
		\begin{subequations}
			\label{eq:var.abstract}
		\begin{alignat}{3}
			\label{eq:Lf.F}
		D\mathcal{E}(\sff)&=\la\sfV,\mathcal{C}\cdot\ra_{\Hd^*,\Hd}\quad&&\text{ in }H^*,
\\\label{eq:LV.F}
		\mathcal{C}\sff&=\sigma\kap\quad&&\text{ in }\Hd.
		\end{alignat}
	\end{subequations}
	In particular, $\int_{-1}^{1}\big(\sff+\frac{n}{n'}\pa_u \sff\big)\,\dd u=\sigma\kap$ a.e.\ in $\Gamma$, and 
	for all $\vp\in H$
	\begin{align}
		\int_{\Gamma}\int_{[-1,1]} \Big(\sfa\pa_u \sff\,\pa_u \vp
		+\sfm\nabla_\Gamma \sff\cdot \nabla_\Gamma \vp\Big)\,\dd u\,\dd \mathcal{H}^{d-1}= 
		\big\la\sfV,\int_{-1}^{1}\big(\vp+\frac{n}{n'}\pa_u \vp\big)\,\dd u\big\ra_{\Hd^*,\Hd}.\qquad
	\label{eq:weakForm.vn}\end{align}
	Define the linear solution operator
		\begin{align*}
		\mathcal{\wh G}=(\mathcal{F},\mathcal{G}): \Hd&\to H\times \Hd^*
	\\	\kap&\mapsto (\sff,\sfV).
		\end{align*}
Then, $\mathcal{G}\kappa=-\frac12\Delta_\Gamma\int_{-1}^1\sfm\sff\,\dd u$, where $\sff:=\mathcal{F}\kappa$. More precisely, for all $\psi\in \Hd$
\begin{align}\label{eq:Gkappa.formula}
\la\mathcal{G}\kappa,\psi\ra_{\Hd^*,\Hd}=\frac12\int_{\Gamma}\nabla_\Gamma\left(\int_{-1}^1\sfm\mathcal{F}\kappa\,\dd u\right)\cdot\nabla_\Gamma\psi\,\dd\mathcal{H}^{d-1}.
\end{align}
Furthermore, $\mathcal{\wh G}$ is continuous and 
\begin{subequations}\label{eq:hatG.cont}
\begin{align}\label{eq:Fkap.bd}
(\sff,\sff)_{\mathcal{E}}^\frac12&\le\frac{\sigma}{\sqrt{\delta}}\|\nabla_\Gamma\kappa\|_{L^2(\Gamma)},\quad\sff:=\mathcal{F}\kappa,
 \\
\sup_{\{\psi\in H^1:\|\nabla_\Gamma\psi\|_{L^2}\le 1\}} \la\mathcal{G}\kappa,\psi\ra_{H^1(\Gamma)^*,H^1(\Gamma)}&\le
\frac\sigma\delta\,\|\nabla_\Gamma\kappa\|_{L^2(\Gamma)},
 \label{eq:Gkap.bd}
\end{align}
\end{subequations}
where $\sigma,\delta$ are given by~\eqref{eq:sig-del.u}.
\end{proposition}
The second component $\mathcal{G}=\mathcal{G}_\Gamma$ of the operator $\mathcal{\wh G}$ determines the evolution law of the moving hypersurface $\Gamma=\Gamma(t)$ through $\sfV_\Gamma=\mathcal{G}_\Gamma\kappa_\Gamma$, where $\kappa_\Gamma\in H^1(\Gamma)$ denotes the mean curvature of $\Gamma$, and $\sfV_\Gamma$ the normal velocity (cf.\ problem~\eqref{eq:u-var.key-problem} resp.~\eqref{eq:problem.fV}).
\begin{definition}
We call the operator $\mathcal{G}_\Gamma:\kappa\to -\frac12\Delta_\Gamma\int_{-1}^1\sfm\mathcal{F}\kappa\,\dd u$ the \emph{propagation operator}.
\end{definition}
\begin{proof}[Proof of Proposition~\ref{prop:solveP.vn}]
	The functional $\mathcal{E}:H\to \mathbb{R}$ is
convex and continuous, and thus weakly lower semi-continuous. Furthermore, the restriction 
$\mathcal{E}:M_\kap\to \mathbb{R}$ is
 mildly coercive on $M_\kap$ ensuring that minimising sequences of $\mathcal{E}$ in $M_\kap$ are bounded with respect to $\|\cdot\|_H$.
The affine space $M_\kap\subset H$ is closed, and thus weakly closed. 
Consequently, a standard application of the direct method of the calculus of variations (cf.~\cite[Proposition~41.2]{Zeidler_Variational}) yields a unique solution $f\in M_\kap$ to the constrained minimisation problem
\begin{align}\label{eq:c-min.f.F.vn}
	\mathcal{E}(\sff)=\inf_{\tilde\sff\in M_\kap}\mathcal{E}(\tilde\sff).
\end{align}
The uniqueness of the solution $\sff$ to~\eqref{eq:c-min.f.vn} follows from the strict convexity of $\mathcal{E}_{|M_\kap}$.

Equation~\eqref{eq:LV.F} is immediate, since $\sff\in M_\kap$. To deduce~\eqref{eq:Lf.F}, we note that the continuous linear operator $\mathcal{C}:H\to \Hd$ is a submersion (since $\mathcal{C}(\frac12 h)=h$ for all $h=h(\sfs)\in \Hd$). Therefore, the theory of Lagrange multipliers (see e.g.~\cite[Theorem 43\,D\,(1)]{Zeidler_Variational}) yields the existence of a unique $\sfV\in \Hd^*$ such that for all $\vp\in H$
\begin{align}\label{eq:LagrangeMultiplierCond.vn}
\la D\mathcal{E}(\sff),\vp\ra_{H^*,H}=\la\sfV, D\mathcal{C}(\sff)\vp\ra_{\Hd^*,\Hd}=\la\sfV, \mathcal{C}\vp\ra_{\Hd^*,\Hd},
\end{align}
where the second equality follows from the linearity of $\mathcal{C}$.
The uniqueness of solutions to~\eqref{eq:LagrangeMultiplierCond.vn} follows by invoking the converse direction~\cite[Theorem 43\,D\,(2)]{Zeidler_Variational} of the Lagrange multiplier rule and the uniqueness of the solution $\sff$ to~\eqref{eq:c-min.f.F.vn}.

By construction, $D\mathcal{E}(\sff)=(\sff,\cdot)_{\mathcal{E}}$. Inserting this identity in~\eqref{eq:LagrangeMultiplierCond.vn}, we conclude the weak formulation~\eqref{eq:weakForm.vn}.

Choosing in~\eqref{eq:weakForm.vn} the test function $\vp\equiv\psi$ with $\psi\in \Hd$, which is  admissible since $\vp\in H$, we deduce~\eqref{eq:Gkappa.formula}.

It remains to show the bounds~\eqref{eq:Fkap.bd},~\eqref{eq:Gkap.bd}, which imply the continuity of the linear map $\mathcal{\wh G}$ from $\Hd$ to $H\times \Hd^*$.
From~\eqref{eq:Gkappa.formula} we deduce the bound
\[
\la\mathcal{G}\kappa,\psi\ra_{H^1(\Gamma)^*,H^1(\Gamma)}\le\frac12 \left(\int_{-1}^1\sfm\,\dd u\right)^\frac12(\sff,\sff)_{\mathcal{E}}^\frac12\|\nabla_\Gamma\psi\|_{L^2(\Gamma)}=\frac{1}{\sqrt{\delta}}(\sff,\sff)_{\mathcal{E}}^\frac12\|\nabla_\Gamma\psi\|_{L^2(\Gamma)},\quad \sff=\mathcal{F}\kappa.\]
Choosing $\mathcal{F}\kappa\;(=\sff)$ itself as a test function in~\eqref{eq:Lf.F} then gives
\begin{align*}(\mathcal{F}\kappa,\mathcal{F}\kappa)_{\mathcal{E}}
&\le \sigma \sup_{\{\psi\in H^1:\|\nabla_\Gamma\psi\|_{L^2}\le 1\}} \la\mathcal{G}\kappa,\psi\ra_{H^1(\Gamma)^*,H^1(\Gamma)}\|\nabla_\Gamma\kappa\|_{L^2(\Gamma)}
\\&\le\frac\sigma2 \left(\int_{-1}^1\sfm\,\dd u\right)^\frac12(\mathcal{F}\kappa,\mathcal{F}\kappa)_{\mathcal{E}}^\frac12\|\nabla_\Gamma\kappa\|_{L^2(\Gamma)}
=\frac{\sigma}{\sqrt{\delta}} 
(\mathcal{F}\kappa,\mathcal{F}\kappa)_{\mathcal{E}}^\frac12\|\nabla_\Gamma\kappa\|_{L^2(\Gamma)}. 
\end{align*}
\end{proof}

\subsection{Regularity}\label{ssec:regularity}
Here, we show a basic regularity property of the solution $(\sff,\sfV)$ to the constrained elliptic equation~\eqref{eq:Lf.F},~\eqref{eq:LV.F} in tangential variables.
To this end, it will be convenient to work with a suitable orthonormal basis 
of eigenfunctions of the minus Laplace--Beltrami operator $-\Delta_{\Gamma}$. We therefore begin with some preliminaries introducing the appropriate functional setting and spectral decomposition, which we will also utilise in Section~\ref{ssec:symbol}.

\paragraph{Homogeneous Sobolev spaces.} Given a hypersurface $\Gamma$ satisfying~\ref{hp:e.geom}, we denote by $\dot L^2(\Gamma):=\{\sfh\in L^2(\Gamma):\av_{\Gamma} \sfh\,\dd\mathcal{H}^{\dimx-1}=0\}$
 the Hilbert space of square-integrable real-valued functions on $\Gamma$ with zero average.
The minus Laplace--Beltrami operator $-\Delta_\Gamma$, considered as an unbounded operator $-\Delta_\Gamma:D(-\Delta_\Gamma)\Subset\dot L^2(\Gamma)\to \dot L^2(\Gamma)$ with compactly embedded domain, is selfadjoint and strictly positive. 
Thus, by the spectral theorem, there exists an orthonormal basis of eigenfunctions $\{\ee_k\}_{k\in\mathbb{N}}\subset \dot L^2(\Gamma)$ of  $-\Delta_\Gamma$ with associated eigenvalues  $0<\lambda_1\le\lambda_2\le \dots$ satisfying $\lambda_k\uparrow+\infty$ as $k\to\infty$.
For $s\in\mathbb{R}$ and $\mathsf{h}=\sum_{k\in\mathbb{N}}h_k\ee_k$, $h_k\in \mathbb{R}$, we define 
\begin{align*}
	\|\mathsf{h}\|_{\dot H^s}^2:=\sum_{k\in\mathbb{N}}\lambda_k^s|h_k|^2,
\end{align*}
and let
\begin{align*}
	\dot H^s(\Gamma):=\{{\sf h}=\sum_{k\in\mathbb{N}}h_k\ee_k:
	\|\mathsf{h}\|_{\dot H^s}<\infty\}
\end{align*}
denote the homogeneous $L^2$-based Sobolev space of order $s$. Observe that $\dot H^2(\Gamma)$ is the domain of $-\Delta_\Gamma$, and that, owing to~\eqref{eq:G.upper-bd}, the domain of $\mathcal{G}_\Gamma$ contains $\dot H^2(\Gamma)$.
Further note that $-\Delta_\Gamma: \dot H^s(\Gamma)\to  \dot H^{s-2}(\Gamma)$ is an isometric isomorphism.
Finally, observe the natural isomorphism $\dot H^{-s}(\Gamma)\simeq\dot H^s(\Gamma)^*$
given by
\begin{align*}
	\dot H^{-s}(\Gamma)\ni {\sf h}=\sum_{k\in \mathbb{N}}h_k\ee_k\mapsto \tilde{\sf h}, \quad \la \tilde {\sf h},\phi\ra_{\dot H^s(\Gamma)^*,\dot H^s(\Gamma)}=\sum_{k\in\mathbb{N}}h_k\phi_k.
\end{align*}
We further let 
\begin{align*}
	\Lambda=\{\lambda_k:k\in \mathbb{N}\}, \qquad\text{ and }\qquad\Lambda_R=\{\lambda_k\in\Lambda:\lambda\le R\}, \; R>0.
\end{align*}
In general, an eigenvalue $\lambda\in\Lambda$ may, of course, have multiplicity strictly larger than one in the sense that
$\lambda=\lambda_k=\lambda_l$ for certain $k\not=l$.

\paragraph{Projection on eigenspace.}

\begin{lemma}\label{l:G.diagonal} The following holds true:
	\begin{enumerate}
		\item For all $k,l\in \mathbb{N}$ with $k\neq l$ it holds that
		\begin{align}\label{eq:phikl=0}
			(\mathcal{F}\ee_k,\ee_l)_{L^2(\Gamma)}=0\quad\text{a.e.\ in }(-1,1).
		\end{align}
		\item There exists $\zeta:\Lambda\to\mathbb{R}_{>0}$ such that for all $\sfh\in \dot H^2(\Gamma)$,
		$\sfh=\sum_{k\in\mathbb{N}}h_k\ee_k$, 
		\begin{align}\label{eq:GGam.spectral}
			\mathcal{G}_{\Gamma}\sfh = \sum_{k\in\mathbb{N}}\zeta(\lambda_k)h_k\ee_k, \qquad 
		\end{align}
		The map $\zeta$ is uniquely determined by 
		\begin{align}\label{eq:zetak}
			\zeta(\lambda_k)=(\mathcal{G}_\Gamma\ee_k,\ee_k)_{L^2(\Gamma)}=\frac12\lambda_k\int_{-1}^1\sfm f_k\,\dd u,\quad f_k:=(\mathcal{F}\ee_k,\ee_k)_{L^2(\Gamma)},
			\quad k\in\mathbb{N}.
		\end{align}
	\end{enumerate}
\end{lemma}
\begin{proof}
	Given $k\neq l$, we take $\kappa=\ee_k$ and $\vp(\sfs,u)=\phi_{kl}(u)\ee_l(\sfs)$, where
	$\phi_{kl}=(\mathcal{F}\ee_k,\ee_l)_{L^2(\Gamma)}$ in \Cref{prop:solveP.vn} (cf.\ \eqref{eq:LV.F}, \eqref{eq:weakForm.vn}).
	Then 
	\[\mathcal{C}\phi_{kl}=\mathcal{C}(\mathcal{F}\ee_k,\ee_l)_{L^2(\Gamma)}=\sigma(\ee_k,\ee_l)_{L^2(\Gamma)}=0.\]
	Hence, with the above choice of $\vp$, the right-hand side of equation \eqref{eq:weakForm.vn} vanishes, and we infer, upon rearranging terms,
	\[\int_{-1}^1 \big(\sfa|\pa_u\phi_{kl}|^2 
	+\lambda_l\sfm|\phi_{kl}|^2\big)\dd u= 0,\]
	which implies~\eqref{eq:phikl=0}. Choosing $\vp\equiv\ee_l$ in~\eqref{eq:weakForm.vn} then yields $(\mathcal{G}_\Gamma\ee_k,\ee_l)_{L^2(\Gamma)}=\zeta_k\delta_{kl}$ with $\zeta_k$ given by the right-hand side of~\eqref{eq:zetak}. In view of the completeness of the orthonormal system $\{e_k\}_{k\in \mathbb{N}}\subset \dot L^2(\Gamma)$, we thus infer~\eqref{eq:GGam.spectral} with $\zeta(\lambda_k)$ replaced by $\zeta_k$. 
	
	It remains to show that $\zeta_k=\zeta_l$ whenever $\lambda_k=\lambda_l$.
This is a consequence of the fact that the problem uniquely determining $f_k=(\mathcal{F}\ee_k,\ee_k)_{L^2(\Gamma)}$ only depends on $k$ through $\lambda_k$:
specifically, the equations for $f_k$ are obtained by choosing in~\eqref{eq:var.abstract} the data $\kappa=\ee_k$ and in the weak formulation~\eqref{eq:weakForm.vn} the test function $\vp(\sfs,u)=\phi(u)\ee_k(\sfs)$ for $\phi\in C^\infty([-1,1])$ with $\phi'\in C^\infty_c((-1,1))$, and by taking the $L^2(\Gamma)$-inner product of
\eqref{eq:LV.F} with $\ee_k$:
\begin{subequations}\label{eq:k.gen.abstract}
	\begin{align}
		\label{eq:k.gen.abstract.e}
		&	\int_{-1}^1 \big(\sfa \pa_u f_k\,\pa_u \phi
		+\lambda_k\sfm f_k\phi\big)\dd u= 
		\zeta_k\int_{-1}^1\big(\phi+\frac{n}{n'}\pa_u \phi\big)\dd u\;\;\;
		\forall \,\phi\in C^\infty,\, \supp\phi'\Subset(-1,1), 
		\\& \int_{-1}^1\big(f_k+\frac{n}{n'}\pa_u f_k\big)\,\dd u=\sigma,
		\label{eq:k.gen.abstract.c}
	\end{align}
\end{subequations}
where we recall that $\sigma$ is given by~\eqref{eq:sig-del.u}. Here, $\zeta_k$ is the Lagrange parameter to~\eqref{eq:k.gen.abstract.c}.
Thus, under the constraint~\eqref{eq:k.gen.abstract.c}, the solution $f_k$ to equation~\eqref{eq:k.gen.abstract.e} is already uniquely determined when restricting to
 test functions $\phi$ satisfying $\mathcal C\phi\equiv0$, for which the right-hand side of~\eqref{eq:k.gen.abstract.e} vanishes. This shows that the solution $f_k$ only depends on $k$ through $\lambda_k$, and hence the same is true for $\zeta_k$. Thus, we can set $\zeta(\lambda_k):=\zeta_k$ for all $k\in \mathbb N$.
\end{proof}
	
\paragraph{Regularity.}
Taking advantage of the fundamental orthogonality relations derived in Lemma~\ref{l:G.diagonal}, we are now in a position to establish regularity in tangential variables.
\begin{lemma}[Higher regularity in tangential variables]\label{l:hreg.s}$ $
	Let $\kappa\in H^{k+1}(\Gamma)$ for some $k\in \mathbb{N}$. 
	Then $(-\Delta_\Gamma)^{k/2}\mathcal{F}\kappa\in H$ and $(-\Delta_\Gamma)^{k/2}\mathcal{G}\kappa\in H^1(\Gamma)^*$.
\end{lemma}
\begin{proof}
By hypothesis,  $\kappa-\av_\Gamma\kappa=\sum_{j\in \mathbb{N}}\kappa_j\ee_j$ for coefficients $\kappa_j$ satisfying
$\sum_{j\in \mathbb{N}}\lambda_j^{k+1}|\kappa_j|^2<\infty$.
Let $N\in \mathbb{N}$. Due to the linearity of the operator $\mathcal{\widehat{G}}$, we know that 
$\big(\sum_{j=1}^N\lambda^{k/2}_j\mathcal{F}\kappa_j\ee_j,\sum_{j=1}^N\lambda^{k/2}_j\mathcal{G}\kappa_j\ee_j\big)\in H\times H^1(\Gamma)^*$ is the solution to~\eqref{eq:var.abstract} with datum
$\sum_{j=1}^N\lambda^{k/2}_j\kappa_j\ee_j$. 
Thus, owing to Lemma~\ref{l:G.diagonal}, the estimates~\eqref{eq:hatG.cont} provide us with $N$-truncated versions of the bounds
\begin{align*}
\mathcal{E}((-\Delta_\Gamma)^{k/2}\mathcal{F}\kappa)^{1/2}&\lesssim \|(-\Delta_\Gamma)^{k/2}\kappa\|_{H^1(\Gamma)}
\\\|(-\Delta_\Gamma)^{k/2}\mathcal{G}\kappa\|_{H^1(\Gamma)^*}&\lesssim \|(-\Delta_\Gamma)^{k/2}\kappa\|_{H^1(\Gamma)}.
\end{align*}
Since $\|(-\Delta_\Gamma)^{k/2}\kappa\|_{H^1(\Gamma)}^2\sim\sum_{j\in \mathbb{N}}\lambda_j^{k+1}|\kappa_j|^2<\infty$, the asserted regularity follows in the limit $N{\to}\infty$.
\end{proof}
The regularity in the normal variable depends on the choice of the coefficients $m,n,A^2\tau$.

	\begin{remark}[Analyticity]\label{r:analyticity}
	In Section~\ref{ssec:symbol} we explicitly determine the operators $\mathcal{F},\mathcal{G}$ by computing their action on the basis $\{\ee_j\}_{j\in \mathbb{N}}$.  
	There, we will see that, for a specific choice of coefficients as in \Cref{ass:fsd}, $\mathcal{F}\ee_j$ is analytic in $u$ for all $j$  as long as the coefficient functions are analytic in $u$.
\end{remark}

\section{The (new) geometric evolution law}\label{sec:geom.law}

We now investigate the structural properties of the propagation operator 
$\mathcal{G}_\Gamma:=\mathcal{G}$ given by
\begin{align*}
\mathcal{G}_\Gamma:\kappa\mapsto -\frac12\Delta_\Gamma\int_{-1}^1 \sfm\,\mathcal{F}\kap\,\dd u,
\end{align*}
which, as we have seen in Proposition~\ref{prop:solveP.vn}, determines the interface dynamics via $\sfV_\Gamma=\mathcal{G}_\Gamma\kappa_\Gamma$.

Throughout this section, we assume the general hypotheses~\ref{hp:e.geom},~\ref{hp:e.coeff}, and~\ref{hp:tauA^2n.pos} from \Cref{sec:constr-elliptic}, ensuring that $\mathcal{G}_\Gamma: H^1(\Gamma)\to H^1(\Gamma)^*$ is well-defined,  and adopt the notations introduced in \Cref{ssec:notations.constr-ell}. Recall, in particular, the definition~\eqref{eq:def.m(u)/} of $\sfm=\sfm(u)$.
In the context of an evolving hypersurface, the hypotheses~\ref{hp:e.geom} on the geometry are to be understood pointwise in time.

\subsection{Gradient-flow structure}\label{ssec:nonlin.str}
 
Below, we will use, without further notice, the observation that the regularity property in Lemma~\ref{l:hreg.s} implies that 
$\mathcal{G}_\Gamma h\in L^2(\Gamma)$ for all $h\in H^2(\Gamma)$.

\begin{proposition}[Symmetry, invariance, and positivity of the propagation operator]\label{prop:gs}$ $
\begin{enumerate}
	\item\label{it:sym} \emph{Symmetry.} The operator $\mathcal{G}_\Gamma$ is symmetric with respect to $L^2(\Gamma)$ in the sense that 
	\begin{align}\label{eq:G.sym}
		(\mathcal{G}_\Gamma h,\kap)_{L^2(\Gamma)}=(h,\mathcal{G}_\Gamma\kap)_{L^2(\Gamma)}\qquad \text{for all } h,\kappa\in H^2(\Gamma).
	\end{align}
	\item \emph{Invariance.}
	 It holds that 
	 \begin{align}\label{eq:G.inv}
	 \mathcal{G}_\Gamma 1_\Gamma\equiv0,
	 \end{align}
	  where $1_\Gamma$ denotes the constant function on $\Gamma$ that is identically equal to $1$.
	\item\label{it:G.pos} \emph{Positivity.} It holds that 
	\begin{align}\label{eq:G.pos}
	(\mathcal{G}_\Gamma \kap,\kap)_{L^2(\Gamma)}\ge0\qquad \text{for all } \kappa\in H^2(\Gamma).
	\end{align}
 Furthermore, the equality $(\mathcal{G}_\Gamma \kap,\kap)_{L^2(\Gamma)}=0$ with $\kap\in H^2(\Gamma)$ holds true if and only if $\kap\equiv c$ on $\Gamma$ for some constant $c\in \mathbb{R}$.
 \item\label{it:upper.bd} \emph{Upper bound.} It holds that 
 \begin{align}\label{eq:G.upper-bd}
 \mathcal{G}_\Gamma\le -\frac\sigma\delta\Delta_\Gamma
 \end{align}
 in the sense that
 $(\mathcal{G}_\Gamma\kappa,\kappa)_{L^2(\Gamma)}\le (-\frac\sigma\delta\Delta_\Gamma\kappa,\kappa)_{L^2(\Gamma)}$  for all $\kappa\in H^2(\Gamma)$.
\end{enumerate}
\end{proposition}

\begin{remark}[Gradient structure]
	Since $-\kappa_\Gamma$ can be obtained by normal variation of the surface area functional,
	 the properties~\ref{it:sym},~\ref{it:G.pos} asserted in \Cref{prop:gs}, combined with the conservation law in Corollary~\ref{cor:vmc-flow}~\ref{it:vol.pres} below, mean that, formally, the interface evolution law $\sfV_\Gamma=\mathcal{G}_\Gamma\kappa_\Gamma$ has the structure of a volume-preserving gradient flow of the surface area functional, where the {\em formal} metric is induced by $\big(\eval{{\mathcal{G}_\Gamma}}_{{1_\Gamma}^\perp}\big)^{-1}$. 
\end{remark}

\begin{proof}[Proof of \Cref{prop:gs}] Abbreviate $\mathcal{G}=\mathcal{G}_\Gamma$.
	The proof relies on the characterisation of the solution operator $\mathcal{\widehat{G}}=(\mathcal{F},\mathcal{G})$ in Proposition~\ref{prop:solveP.vn}.
	The starting point is the equality
	\[\mathcal{C}\big(\mathcal{F}h-\frac\sigma2h\big)=0\qquad\text{for all }h\in H^1(\Gamma),\]
	which follows from~\eqref{eq:LV.F} and the definition of $\mathcal{C}$.
	Using~\eqref{eq:Lf.F} and the fact that $D\mathcal{E}(\sff)\vp=(\sff,\vp)_{\mathcal{E}}$, it allows us to deduce that
	\begin{align}\label{eq:test.with.H0}
		\Big(\mathcal{F}\kap,(\mathcal{F}h-\frac\sigma2h)\Big)_{\mathcal{E}}=0\qquad\text{ for all }\kap,h\in H^2(\Gamma).
	\end{align}
From~\eqref{eq:test.with.H0} and equation~\eqref{eq:Lf.F}, we then infer the key identity
\begin{align}\label{eq:reprG.F}
	\big(\mathcal{F}\kap,\mathcal{F}h\big)_{\mathcal{E}}=\big(\mathcal{F}\kap,\frac\sigma2h\big)_{\mathcal{E}}=\la\mathcal{G}\kap,\mathcal{C}(\tfrac\sigma2h)\ra_{H^1(\Gamma)^*,H^1(\Gamma)}
	=\sigma(\mathcal{G}\kap,h)_{L^2(\Gamma)}.
\end{align}
	Thus, assertion~\eqref{eq:G.sym} resp.~\eqref{eq:G.pos} follows from the symmetry resp.\ the non-negativity of the bilinear form $(\cdot,\cdot)_{\mathcal{E}}$ combined with the positivity of $\sigma>0$.
	
	To show the invariance property, we compute for $h\in H^2(\Gamma)$ fixed but arbitrary, using the symmetry of $\mathcal{G}_\Gamma$,~\eqref{eq:Lf.F}, and a calculation as in~\eqref{eq:reprG.F}:
	\begin{align*}
	(h,\mathcal{G} 1_\Gamma)_{L^2(\Gamma)}
	=(\mathcal{G} h, 1_\Gamma)_{L^2(\Gamma)}
	=\tfrac12\la \mathcal{G} h, \mathcal{C}1_\Gamma\ra_{H^1(\Gamma)^*,H^1(\Gamma)}
=\tfrac12(\mathcal{F}h,1_\Gamma)_{\mathcal{E}}=0.
	\end{align*}
Since $h\in H^2(\Gamma)$ was arbitrary, we infer that $\mathcal{G} 1_\Gamma\equiv0$ on $\Gamma$. Alternatively, this assertion can be deduced from~\eqref{eq:Gkap.bd}.

Suppose now that $(\mathcal{G}  \kap,\kap)_{L^2(\Gamma)}=0$ for some $\kap\in H^2(\Gamma)$. From the representation~\eqref{eq:reprG.F} and the definition of the bilinear form $(\cdot,\cdot)_{\mathcal{E}}$, we conclude that 
$\pa_u(\mathcal{F}\kappa)=0$, $\nabla_\Gamma(\mathcal{F}\kap)=0$ a.e.\ on $\Gamma\times [-1,1]$. 
Consequently, there exists $\tilde c\in \mathbb{R}$ such that $\mathcal{F}\kap=\tilde c$ a.e.\ on $\Gamma\times [-1,1]$, and thus $\sigma\kap=\mathcal{C}(\mathcal{F}\kap)=2\tilde c$. Hence $\kap\equiv c$ for $c:=\frac{2}{\sigma}\tilde c\in\mathbb{R}$.
The converse direction that $(\mathcal{G} c,c)_{L^2(\Gamma)}=0$ for constant functions $c$ follows from~\eqref{eq:G.inv}.

The upper bound is an immediate consequence of inequality~\eqref{eq:Gkap.bd}.
\end{proof}

Having established the relevant structural properties of the linear operator $\mathcal{G}_\Gamma$, we may now deduce volume preservation and area decrease of the associated geometric flow along classical solutions.

\begin{corollary}[Volume-preserving curvature flow]\label{cor:vmc-flow}
	Let $\Gamma=\cup_{t\in I}\{t\}{\times}\Gamma(t)$ be a smoothly evolving hypersurface governed by the geometric law \[\sfV_\Gamma=\mathcal{G}_\Gamma\kappa_\Gamma.\]
	Then:
	\begin{enumerate}[label=(\roman*)]
		\item\label{it:vol.pres} \emph{Volume preservation.}  $\frac{\dd}{\dd t}\mathcal{H}^{d}(\Om^-)=0$, where $\Om^-(t)$ denotes the domain enclosed by $\Gamma(t)$. 
		\item\label{it:area.decre} \emph{Area decrease.} $\frac{\dd}{\dd t}\mathcal{H}^{d-1}(\Gamma)\le0$.
		\item \emph{Equilibria.} $\sfV_\Gamma=0$ if and only if $\kappa_\Gamma$ is constant, i.e.\ if $\Gamma(t)\equiv S_{r}^{d-1}(x)$ is a Euclidean sphere.
	\end{enumerate}
\end{corollary}
\begin{proof}
The assertions of \Cref{cor:vmc-flow} are consequences of the properties of $\mathcal{G}_\Gamma$ obtained in \Cref{prop:gs}, see e.g.~\cite{PruessSimonett_2016}. A short derivation is provided below for completeness:

	\textit{Re (i):}
	We compute, using the transport theorem for moving domains (cf.~\cite[Chapter 2.5.5]{PruessSimonett_2016}), the symmetry property~\eqref{eq:G.sym} of $\mathcal{G}_\Gamma$, and the invariance~\eqref{eq:G.inv},
	\begin{align*}
		\frac{\dd}{\dd t}\int_{\Om^-} 1\,\dd x = \int_\Gamma \sfV_\Gamma\,\dd \mathcal{H}^{d-1}=(\mathcal{G}_\Gamma \kap_\Gamma,1_\Gamma)_{L^2(\Gamma)}
		=(\kap_\Gamma,\mathcal{G}_\Gamma 1_\Gamma)_{L^2(\Gamma)}=0.
	\end{align*}
	
		\textit{Re (ii):}
		It follows from the transport theorem for moving hypersurfaces (cf.~\cite[Chapter 2.5.4]{PruessSimonett_2016}) and the positivity of $\mathcal{G}_\Gamma$ (cf.~\cref{it:G.pos} in \Cref{prop:gs}) that the surface area functional is non-increasing along solutions
		\begin{align*}
			\frac{\dd}{\dd t}\int_{\Gamma}1\,\dd\hms &= -\int_\Gamma \sfV_\Gamma\,\kappa_\Gamma\,\dd\hms
			= -(\mathcal{G}_\Gamma\kappa_\Gamma,\kappa_\Gamma)_{L^2(\Gamma)}\le 0
		\end{align*}
		with strict inequality unless $\kap_\Gamma=c$ for some $c\in \mathbb{R}$.
		
		\textit{Re (iii):} It follows from the second part of \cref{it:G.pos} in  \Cref{prop:gs} that $\mathcal{G}_\Gamma\kap_\Gamma=\sfV_\Gamma=0$ is equivalent to $\kap_\Gamma\equiv c\in \mathbb{R}$. Combined with the properties~\ref{hp:e.geom} of the hypersurface $\Gamma(t)$ and 
		Aleksandrov's characterisation of closed connected $C^2$ hypersurfaces with constant mean curvature, embedded in $\mathbb{R}^\dimx$, (cf.~\cite{Aleksandrov_1956}), this amounts to $\Gamma(t)$ being a sphere.
	
\end{proof}

\subsection{Spectral representation of the propagation operator}\label{ssec:symbol}

Our next goal is to explicitly compute the action of the operator $\mathcal{G}_\Gamma:\kappa\mapsto -\frac12\Delta_\Gamma\int_{-1}^1 \sfm\sff\,\dd u$
in terms of $-\Delta_\Gamma$. 
In view of the invariance property $\mathcal{G}_\Gamma 1_\Gamma\equiv0$, it suffices to determine $\mathcal{G}_\Gamma$ on functions $\kappa:\Gamma\to\mathbb{R}$ with $\int_\Gamma\kappa=0$.
The present spectral approach takes advantage of the observation that the operator  $\mathcal{G}_\Gamma$ is diagonal with respect to the orthonormal basis
$\{\ee_k\}_{k\in \mathbb{N}}$ of eigenfunctions of $-\Delta_\Gamma$, as shown in Lemma~\ref{l:G.diagonal}.
We further observe that, in equation~\eqref{eq:k.gen.abstract} (henceforth to be understood with $\zeta_k=\zeta(\lambda_k)$), which determines $f_k:=(\mathcal{F}\ee_k,\ee_k)_{L^2(\Gamma)}$, the underlying geometry is solely encoded in the eigenvalue $\lambda_k$ of the minus Laplace--Beltrami operator on $\Gamma$ with respect to the $k$-th basis function $\ee_k$.
In the following, we will determine the solution $f_k$ to~\eqref{eq:k.gen.abstract} for specific choices of $m,n, A^2\tau$,  see hypotheses~\ref{hp:s.gen},~\ref{hp:s.a=1/m} in \Cref{ssec:explicit} below.
By virtue of identities~\eqref{eq:GGam.spectral},~\eqref{eq:zetak}, this will allow us to specify $\zeta(\lambda_k)$, and thus the propagation operator. A key interest lies in identifying the asymptotic growth law of $\zeta(\lambda)$ as $\lambda\to\infty$.	
We emphasise that the explicit solution $(\sff,\sfV)$ to be constructed below agrees with the unique weak solution of Proposition~\ref{prop:solveP.vn}.

\subsubsection{Problem formulation} 
\label{ssec:explicit}
Let us first list the hypotheses under which the subsequent analysis is valid.
\paragraph{Hypotheses.} 
\begin{enumerate}[label=(s\arabic*)]
	\item\label{hp:s.gen} Let~\ref{hp:e.geom},~\ref{hp:e.coeff},~\ref{hp:tauA^2n.pos} 
	as well as~\ref{hp:n.ndeg}  be in force.
	\item\label{hp:s.a=1/m} 
Assume~\ref{hp:am} with $\tilde a\equiv\text{const}>0$ (required as of \Cref{sssec:a=1/m}), and let $m$ be even.
\end{enumerate}
The first condition in~\ref{hp:s.a=1/m} amounts to requiring that $\sfa=\frac{\text{const}}{\sfm}$. The hypothesis that $m$ (or equivalently $\sfm$) be even has been made to simplify the presentation and can easily be removed.

Notice that the above assumptions are compatible with those in Assertion~\ref{ass:fsd}. 

\paragraph{Strong formulation.}
Upon an integration by parts in equation~\eqref{eq:k.gen.abstract.e} 
and in the constraint~\eqref{eq:k.gen.abstract.c}, problem~\eqref{eq:k.gen.abstract} may be formulated as follows. Determine for $\lambda=\lambda_k>0$ the solution couple $f=f_{\lambda},\vlo=\zeta(\lambda)$ of the system:
\begin{subequations}\label{eq:weak.ibp.all}
\begin{align}\label{eq:weak.a.lambda}
\int_{-1}^1\big((-\pa_u(\sfa\pa_u f)+\sfm\lambda f)\phi\big)\,\dd u
=\int_{-1}^1\frac{n''n}{(n')^2}\,\phi\,\dd u \,\vlo+ \bigg[(-\sfa\pa_u f+\frac{n}{n'}\vlo)\phi\bigg]_{-1}^{1} \qquad
\end{align}
for all $\phi\in C^\infty([-1,1])$ with $\phi'\in C^\infty_c((-1,1))$, and 
\begin{align}\label{eq:constr.expl}\int_{-1}^1\frac{n''n}{(n')^2} f\,\dd u+\bigg[\frac{n}{n'}f\bigg]_{-1}^{1}=\sigma.
\end{align}
\end{subequations}
Problem~\eqref{eq:weak.ibp.all} can be decomposed into three subproblems:
\begin{subequations}\label{eq:inhom.all.expl}
\begin{enumerate}
	\item\label{it:in} First considering $\phi\in C^\infty_c((-1,1))$, reduces~\eqref{eq:weak.a.lambda} to the second-order differential equation 
	\begin{align}\label{eq:inhom.am}
		-\pa_u(\sfa\pa_u f)+\sfm\lambda f=\frac{n''n}{(n')^2}\,\vlo
	\end{align}
	in the pointwise sense.
	\item\label{it:bc} Taking now into account that in~\eqref{eq:weak.a.lambda} general test functions $\phi\in C^\infty([-1,1])$ with $\phi'\in C^\infty_c((-1,1))$ are admitted, yields the associated boundary conditions on $(-1,1):$
	\begin{align}\label{eq:bc.am}\sfa\pa_u f=\frac{n}{n'}\,\vlo \quad\text{ for }u\in\{\pm1\}.
	\end{align}
	\item\label{it:constr} The constraint is taken as stated, i.e.
	\begin{align}\label{eq:constr.am}\int_{-1}^1\frac{n''n}{(n')^2} f\,\dd u+\bigg[\frac{n}{n'}f\bigg]_{-1}^{1}=\sigma.
		\end{align}
\end{enumerate}
\end{subequations}
If $f$ and $\sfa\pa_u f$ are sufficiently regular, the three equations~\eqref{eq:inhom.am}--\eqref{eq:constr.am} are equivalent to~\eqref{eq:k.gen.abstract}.

Our strategy is now to first compute the general solution $f$ to~\cref{it:in} for given $\vlo$. This solution has two degrees of freedom, denoted by $b_1,b_2\in\mathbb{R}$, which we then specify in such a way that $f$ fulfils the boundary conditions in \cref{it:bc}. In the last step, we fix $\vlo$ in such a way that \cref{it:constr} is fulfilled.

\subsubsection{Explicit solution}
\label{sssec:a=1/m}

Our explicit approach below takes advantage of the identity
$\sfa=\frac{\tilde a}{\sfm}$ with $\tilde a\equiv\text{const}>0$ imposed in 
hypothesis~\ref{hp:s.a=1/m}. To simplify the presentation, we suppose that $\tilde a=1$. The extension to the case of general $\tilde a\equiv\text{const}>0$ is straightforward by suitable rescalings, see also \Cref{ssec:vanishing.slope}. 

For $\sfa=1/\sfm$ the change of variables $r=\alpha(u)$,  $\alpha(u):=\int_0^u \sfm(u')\,\dd u'$ brings the homogeneous equation 
\begin{align}\label{eq:hom.a.m}
-\pa_u(\sfa\pa_u f)+\sfm\lambda f=0
\end{align}
into the constant-coefficient form 
\begin{align}\label{eq:const.coeff}
	-\pa_r^2\hat f+\lambda \hat f=0.
\end{align}
Equation~\eqref{eq:const.coeff} has two explicit linearly independent solutions $\hat f_\pm(r)=\frac{1}{\lambda^{1/4}}\ee^{\pm\sqrt{\lambda}r}$. 
Returning to the original variables,  the solutions $f_\pm=\hat f_\pm\circ\alpha$ to the homogeneous equation~\eqref{eq:hom.a.m} take the form 
\begin{align*}
	f_+(u)=\frac{1}{\lambda^{1/4}}\ee^{\sqrt{\lambda}\alpha(u)},\qquad f_-(u)=\frac{1}{\lambda^{1/4}}\ee^{-\sqrt{\lambda}\alpha(u)}.
\end{align*}
For later use, we note that, since $m,\sfm$ are even, the function $\alpha$ is odd.

The Wronskian $W$ associated to $(f_+,f_-)$ is given by
\[W=\pa_u f_+f_--\pa_u f_-f_+=2\sfm.\]
Let $\tilde F:=\sfm F:=\sfm \frac{nn''}{(n')^2}\,\vlo$.
Then $\tilde F/W=\frac12\ell\vlo$, where
\begin{align}\label{eq:def.ell}
\ell(u):=\frac{nn''}{(n')^2}.
\end{align}

We assert that, using the method of \textit{variation of parameters}, the general solution to the inhomogeneous equation~\eqref{eq:inhom.am} can be written in the form 
\begin{align}\label{eq:inhom.f.gen.n}
	f(u)=\Big(-f_+(u)\int_{1}^u f_-\tfrac{\ell}{2}\dd u'+f_-(u)\int_{-1}^u f_+\tfrac{\ell}{2}\dd u' +\frac{b_1}{\lambda^{1/4}}f_+(u)\ee^{-\sqrt{\lambda}\alpha(1)}+\frac{b_2}{\lambda^{1/4}}f_-(u)\ee^{-\sqrt{\lambda}\alpha(1)}\Big)\vlo,\qquad 
\end{align}
where $b_1,b_2\in \mathbb{R}$ are free parameters.
For convenience, we provide the calculations showing the  solution property:
first we compute, using~\eqref{eq:inhom.f.gen.n},
\begin{multline}\label{eq:a.df}
	\sfa\pa_u f=\Big(-\ee^{\sqrt{\lambda}\alpha(u)}\int_{1}^u  \ee^{-\sqrt{\lambda}\alpha}\tfrac{\ell}{2}\dd u'-\ee^{-\sqrt{\lambda}\alpha(u)}\int_{-1}^u \ee^{\sqrt{\lambda}\alpha}\tfrac{\ell}{2}\dd u' \\+b_{1}\ee^{-\sqrt{\lambda}(\alpha(1)-\alpha(u))}-b_{2}\ee^{-\sqrt{\lambda}(\alpha(1)+\alpha(u))}\Big)\vlo.
\end{multline}
Differentiating once more with respect to $u$, we deduce
\begin{multline*}
	-\pa_u(\sfa\pa_u f)
	=-\lambda^{1/2}\sfm\Big(-\ee^{\sqrt{\lambda}\alpha(u)}\int_{1}^u  \ee^{-\sqrt{\lambda}\alpha}\tfrac{\ell}{2}\dd u'+\ee^{-\sqrt{\lambda}\alpha(u)}\int_{-1}^u \ee^{\sqrt{\lambda}\alpha}\tfrac{\ell}{2}\dd u' 
	\\
	+b_{1}\ee^{-\sqrt{\lambda}(\alpha(1)-\alpha(u))}+b_{2}\ee^{-\sqrt{\lambda}(\alpha(1)+\alpha(u))}\Big)\vlo
	+\ell\vlo.
\end{multline*}
Observing that
\begin{multline*}
	\Big(-\ee^{\sqrt{\lambda}\alpha(u)}\int_{1}^u  \ee^{-\sqrt{\lambda}\alpha}\tfrac{\ell}{2}\dd u'+\ee^{-\sqrt{\lambda}\alpha(u)}\int_{-1}^u \ee^{\sqrt{\lambda}\alpha}\tfrac{\ell}{2}\dd u' 
	\\+b_{1}\ee^{-\sqrt{\lambda}(\alpha(1)-\alpha(u))}+b_{2}\ee^{-\sqrt{\lambda}(\alpha(1)+\alpha(u))}\Big)\vlo=\lambda^{1/2} f,
\end{multline*}
we deduce that $f$ chosen according to~\eqref{eq:inhom.f.gen.n} satisfies, in the pointwise sense, the equation~\eqref{eq:inhom.am}, i.e.
\[	-\pa_u(\sfa\pa_u f)+\lambda \sfm f=\ell\vlo\]
with $\ell$ given by~\eqref{eq:def.ell}.

The parameters $b_1,b_2$ and $\vlo$ will now be fixed in such a way that $f=f_k$ fulfils all remaining properties, which will ensure that $f_k\ee_k$ coincides with the unique weak solution $\sff=\mathcal{F}\ee_k$ constructed in Proposition~\ref{prop:solveP.vn} for data $\kappa=\ee_k$.
We recall that $\zeta>0$, and thus $\zeta\not=0$, which also follows from condition~\eqref{eq:constr.am} by virtue of
 $\sigma>0$.
Let us first impose the boundary conditions~\eqref{eq:bc.am}.
We abbreviate 
\begin{align}\label{eq:def.c} 
	c(u)=\frac{n(u)}{n'(u)}.
\end{align}
Then, using~\eqref{eq:a.df}, condition~\eqref{eq:bc.am} turns into the system
\begin{align*}
	-\ee^{-\sqrt{\lambda}\alpha(1)}\int_{-1}^{1} \ee^{\sqrt{\lambda}\alpha}\tfrac{\ell}{2}\dd u' +b_{1}-b_{2}\ee^{-2\sqrt{\lambda}\alpha(1)}&= c(1)
	\\	\ee^{\sqrt{\lambda}\alpha(-1)}\int_{-1}^{1} \ee^{-\sqrt{\lambda}\alpha}\tfrac{\ell}{2}\dd u'
	+b_{1}\ee^{-2\sqrt{\lambda}\alpha(1)}-b_{2}&= c(-1).
\end{align*}
Define the $2{\times}2$-matrix
\begin{align*}
	M:=\begin{pmatrix}
		1&-\ee^{-2\sqrt{\lambda}\alpha(1)}
		\\ \ee^{-2\sqrt{\lambda}\alpha(1)}&-1
	\end{pmatrix}.
\end{align*}
Note that $M$ is invertible for $\lambda>0$.
Thus, condition~\eqref{eq:bc.am} uniquely determines $b=(b_1,b_2)\in \mathbb{R}$ by $M b=\Mb$, where 
\begin{align*}
	\Mb=\begin{pmatrix}
		c(1)+\ee^{-\sqrt{\lambda}\alpha(1)}\int_{-1}^{1} \ee^{\sqrt{\lambda}\alpha}\tfrac\ell2\dd u'
		\\c(-1)-\ee^{\sqrt{\lambda}\alpha(-1)}\int_{-1}^{1} \ee^{-\sqrt{\lambda}\alpha}\tfrac\ell2\dd u'
	\end{pmatrix}.
\end{align*}
We next estimate the asymptotic behaviour of $p$ as $\lambda\to\infty$.
Owing to~\ref{hp:n.ndeg}, the factor $\frac{nn''}{(n')^2}$
 appearing in the definition of $\ell$ (cf.~\eqref{eq:def.ell}) is bounded: $C_n:=\sup_{u\in[-1,1]}\big|\frac{n(u)n''(u)}{(n'(u))^2}\big|<\infty$. Furthermore,
$\alpha(u)=\int_0^u \sfm(\tilde u)\,\dd\tilde u$ is odd and increasing with $\max_{[-1,1]}\alpha=\alpha(1)$. Therefore,
\begin{align*}
	\bigg|\ee^{-\sqrt{\lambda}\alpha(1)}\int_{-1}^{1} \ee^{\pm\sqrt{\lambda}\alpha(u')}\tfrac{\ell(u')}{2}\dd u'\bigg|
	&\le C_n\int_0^{1} \ee^{-\sqrt{\lambda}\int_{u'}^{1}\sfm(\tilde u)\,\dd\tilde u}\,\dd u'.
\end{align*}
Definition~\eqref{eq:def.m(u)/} and hypothesis~\ref{hp:e.coeff} imply that $\sfm(u)=\frac{m(u)}{\sqrt{1-u^2}}\gtrsim(1-u)^{i-\frac12}$ on $(0,1)$, and hence $\int_{u'}^{1}\sfm(\tilde u)\,\dd \tilde u\gtrsim(1-u')^{i+\frac12}$.
We thus obtain, for a small fixed constant $\delta>0$,
\begin{align*}
	\int_0^{1} \ee^{-\sqrt{\lambda}\int_{u'}^{1}\sfm(\tilde u)\,\dd\tilde u}\,\dd u'
	&\lesssim 
	\int_0^{1} \ee^{-\delta\sqrt{\lambda}(1-u')^{i+\frac12}}\dd u'
	\\&\lesssim
	\lambda^{-\frac{1}{2(i+\frac12)}}	\int_0^{\lambda^{\frac{1}{2(i+\frac12)}}} \ee^{-\delta u^{i+\frac12}}\,\dd u\lesssim
	\lambda^{-\frac{1}{2i+1}},
\end{align*}
where, in the second step, we employed the change of variables $u=\lambda^{\frac{1}{2(i+\frac12)}}(1-u')$.

In combination, the last two estimates show that 
\begin{align}\label{eq:bound.err.i}
	\bigg|\ee^{-\sqrt{\lambda}\alpha(1)}\int_{-1}^{1} \ee^{\pm\sqrt{\lambda}\alpha}\tfrac\ell2\dd u'\bigg|
	&\le C\lambda^{-\frac{1}{2i+1}}
\end{align}
for a constant $C\in(0,\infty)$ that is independent of $\lambda$.
Therefore, as $\lambda\to\infty$,
\begin{align*}
	\Mb_1=c(1)+O(\lambda^{\;-\errfi}),\qquad \Mb_2=c(-1)+O(\lambda^{\;-\errfi}).
\end{align*}
Since $M=\diag(1,-1)+O(\ee^{-2\sqrt{\lambda}\alpha(1)})$
 as $\lambda\to\infty$, 
we conclude that, as $\lambda\to\infty$,
\begin{align}\label{eq:b.asympt.lambda}
	b_1=c(1)+O(\lambda^{\;-\errfi}),\qquad b_2=-c(-1)+O(\lambda^{\;-\errfi}).
\end{align} 
It remains to determine $\vlo$ in such a way that~\eqref{eq:constr.am} holds true. 
To this end, let us compute the value of $\mathcal{C}\tilde f =\int_{-1}^1\frac{n''n}{(n')^2} \tilde f\,\dd u+\big[\frac{n}{n'}\tilde f\big]_{-1}^{1}$, where 
(cf.~\eqref{eq:inhom.f.gen.n})
\begin{align*}
	\tilde f:=f/\vlo=-f_+(u)\int_{1}^u f_-\frac\ell2\dd u'+f_-(u)\int_{-1}^u f_+\frac\ell2\dd u' +\frac{b_1}{\lambda^{1/4}}f_+(u)\ee^{-\sqrt{\lambda}\alpha(1)}+\frac{b_2}{\lambda^{1/4}}f_-(u)\ee^{-\sqrt{\lambda}\alpha(1)}.
\end{align*}
Reasoning similarly as in the derivation of the bound~\eqref{eq:bound.err.i} and using~\eqref{eq:b.asympt.lambda}, we find
\begin{align}\label{eq:est.main.constr}\big[\frac{n}{n'}\tilde f\big]_{-1}^{1}=\big(c(1)b_1-c(-1)b_2\big)\,\lambda^{-\frac12}+O(\lambda^{-\frac12-\errfi})=c_*\lambda^{-1/2}+O(\lambda^{-\frac12-\errfi})
\end{align}
with 
\begin{align*}
	c_*:=c(1)^2+c(-1)^2>0.
\end{align*}
Furthermore, we assert that 
\begin{align}\label{eq:est.remainder.constr}
	\Big|\int_{-1}^1\frac{n''n}{(n')^2} \tilde f\,\dd u\Big|
	\lesssim\lambda^{-\frac12-\errfi}.
\end{align} 
\begin{proof}[Proof of the bound~\eqref{eq:est.remainder.constr}]
	We estimate 
	\[
	\Big|\int_{-1}^1\frac{n''n}{(n')^2} \tilde f\,\dd u\Big|
	\lesssim
	\int_{-1}^1 |\tilde f|\,\dd u
	\lesssim \lambda^{-{1/2}}R_1+\lambda^{-{1/2}}R_2,
	\]
	with the non-negative terms $R_i, R_{i,j}\ge0$, $i,j=1,2$, 
	\begin{align*}R_1:=R_{1,1}+R_{1,2}
		:=\int_{-{1}}^{1} \ee^{\sqrt{\lambda}\alpha(u)}\int_u^{1} \ee^{-\sqrt{\lambda}\alpha}\dd u'\dd u
		+\int_{-{1}}^{1} \ee^{-\sqrt{\lambda}\alpha(u)}\int_{-1}^u \ee^{\sqrt{\lambda}\alpha}\dd u'\dd u,
	\end{align*}
	and 
\begin{align*}R_2&:=R_{2,1}+R_{2,2}
	:=\int_{-{1}}^{1} \ee^{-\sqrt{\lambda}(\alpha(1)-\alpha(u))}\dd u
	+\int_{-{1}}^{1} \ee^{-\sqrt{\lambda}(\alpha(1)+\alpha(u))}\dd u.
\end{align*}
Each of the two summands of $R_2$ can be bounded similarly as~\eqref{eq:bound.err.i} giving
\begin{align*}
	R_2\lesssim \lambda^{-\frac{1}{2i+1}}.
\end{align*}
We next turn to $R_{1,1}:$
\begin{align*}
	R_{1,1}&=\int_{-{1}}^{1} 
	\int_u^{1} \ee^{-\sqrt{\lambda}(\alpha(u')-\alpha(u))}\dd u'\dd u
	\\&=\int_{-{1}}^{1} 
	\int_u^{1} \ee^{-\sqrt{\lambda}\int^{u'}_u \sfm(\tilde u)\,\dd\tilde u}
	\dd u'\dd u=:I_1+I_2+I_3,
\end{align*}
where in the last line we split the double-integal in three parts corresponding to: \\
$I_1:$ $u>0$;\\ $I_2:$ $u<0, u'>0$;\\
$I_3:$ $u<0,u'<0$.\\
Since $m(u)\gtrsim (1-u)^{i-\frac12}$ for $u\in (0,1)$, there exists $\delta>0$ such that 
\begin{align*}
	I_1&:=\int_{0}^{1} 
	\int_u^{1} \ee^{-\sqrt{\lambda}\int^{u'}_u \sfm(\tilde u)\,\dd\tilde u}
	\dd u'\,\dd u
	\\&\lesssim \int_{0}^{1} 
	\int_u^{1} \ee^{-\delta\sqrt{\lambda}\big(
		(1-u)^{i+\frac12}-(1-u')^{i+\frac12}
		\big)
	}
	\dd u'\,\dd u
\end{align*}
Upon changing variables $\hat u:=\lambda^{\frac{1}{2i+1}}(1-u)$, $\bar u=\lambda^{\frac{1}{2i+1}}(1-u')$, we obtain 
\begin{align*}
	I_1&\lesssim \lambda^{-\frac{2}{2i+1}}
	\int_{0}^{\lambda^\frac{1}{2i+1}} 
	\int_0^{\hat u} \ee^{-\delta(
		\hat u^{i+\frac12}-\bar u^{i+\frac12})
	}
	\,\dd\bar u\dd\hat u
	\\&\le \lambda^{-\frac{2}{2i+1}}
	\int_{0}^{\lambda^\frac{1}{2i+1}} 
	\bigg(	\int_0^1 \ee^{-\delta(
		\hat u^{i+\frac12}-\bar u^{i+\frac12})
	}
	 \,\dd\bar u
	+	\ee^{-\delta
		\hat u^{i+\frac12}}\int_1^{\max\{\hat u,1\}} \ee^{\delta\bar u^{i+\frac12}}
	\bar u^{i-\frac12}\,\dd\bar u
	\bigg)
	\,\dd\hat u
	\\&\le \lambda^{-\frac{2}{2i+1}}
	\int_{0}^{\lambda^\frac{1}{2i+1}} 
	\big(C_1+C_2	\big)\,\dd\hat u
	\\&\lesssim \lambda^{-\frac{1}{2i+1}},
\end{align*}
where in the penultimate step we used $\ee^{\delta\bar u^{i+\frac12}}
\bar u^{i-\frac12}=\frac{1}{(i+\frac12)\delta}\frac{\dd}{\dd \bar u}\ee^{\delta {\bar u}^{i+\frac12}}$.

For the integrals $I_3$ and $I_2$, we obtain analogous bounds, so that $R_{1,1}	\lesssim \lambda^{-\frac{1}{2i+1}}.$

The term $R_{1,2}$ can be handled in the same way as $R_{1,1}$, leading to 
$R_{1,2}	\lesssim \lambda^{-\frac{1}{2i+1}}$.

In combination, this proves the asserted estimate~\eqref{eq:est.remainder.constr}.
\end{proof}

 From~\eqref{eq:est.main.constr},~\eqref{eq:est.remainder.constr} we conclude that $\mathcal{C}\tilde f=\lambda^{-1/2}\big(c_*+O(\lambda^{-\frac{1}{2i+1}})\big)$.
Since imposing the constraint $\mathcal{C}f=\sigma$ translates into $\zeta=\sigma/\mathcal{C}\tilde f$, the expression for $\mathcal{C}\tilde f$ computed above determines the asymptotic growth of $\zeta$, as $\lambda\to\infty$, in the form
\begin{align*}
\vlo(\lambda)=\sigma\eta\sqrt{\lambda}\,(1+O(\lambda^{-\errfi})),\qquad 
\eta:=\frac{1}{c_*}=\big((\tfrac{n(1)}{n'(1)})^2+(\tfrac{n(-1)}{n'(-1)})^2\big)^{-1}.
\end{align*}
Thus, for coefficient functions satisfying the hypotheses~\ref{hp:s.gen},~\ref{hp:s.a=1/m}, the action of the operator $\mathcal{G}_\Gamma$ is given by~\eqref{eq:GGam.spectral} with $\zeta=\zeta(\lambda_k)$ as above, 
and the corresponding curvature flow takes the form
\begin{align}\label{eq:minusLaplBeltr^12}
	\sfV_\Gamma=\sigma\eta\sqrt{-\Delta_\Gamma}\kappa_\Gamma  +\sigma\mathcal{R}(\sqrt{-\Delta_\Gamma})\kappa_\Gamma,
\end{align}
where $\mathcal{R}(\sqrt{-\Delta_\Gamma})$ denotes a linear pseudo-differential operator of order strictly less than one (and hence of lower order with repect to $\sqrt{-\Delta_\Gamma}$).
For linearly degenerate mobility, i.e.\ $i=1$, we obtain the growth law
$\zeta(\lambda)=\sigma\eta\lambda^{1/2}+\sigma O(\lambda^{1/6})$ and a remainder $\mathcal{R}$ of order at most $\frac13$. 

The geometric evolution law~\eqref{eq:minusLaplBeltr^12} has the structure of a third-order quasi-linear parabolic equation, and differs both from intermediate surface diffusion, which is parabolic of order two, and from classical surface diffusion, which is parabolic of order four.
One may refer to laws of the above type more generally as \textit{fractional surface diffusion}. Notice that while~\eqref{eq:minusLaplBeltr^12} illuminates the PDE structure of the law $\sfV_\Gamma=\mathcal{G}_\Gamma\kappa_\Gamma$, its  variational structure has been captured by the arguments in Section~\ref{ssec:nonlin.str}.

\paragraph{Examples.}
We conclude by a selection of prototypical choices of $m$ and $n$ that obey the hypotheses~\ref{hp:s.gen},~\ref{hp:s.a=1/m} of the present section:
\begin{enumerate}[label=(C\arabic{*})]
	\item\label{case:degmob.cl} 
	$m(u)=1-u^2$, leading to 
	\[\alpha(u)=u-\frac13u^3.\]
	\item\label{case:degmob.ddeg}
	$m(u)=(1-u^2)^2$, leading to
	\[\alpha(u)=u-\frac23u^3+\frac15u^5.\]
\end{enumerate}
The arguably simplest choice of an admissible coupling function $n(u)$ with $\inf|n'|>0$, is given by an affine choice with non-trivial slope.
Without loss of generality, we suppose that $n$ is larger in the polymeric phase $\{u\approx+1\}:$
\begin{align}\label{eq:n.affine}
	n(u)=\beta_0+\beta_1(u+1), \quad \beta_i>0,\, i=0,1, \quad u\in[-1,1].
\end{align}
Notice that for this choice of $n$, it holds that $\frac{n(u)}{n'(u)}=(u+1)+\frac{\beta_0}{\beta_1}$ and $n''\equiv0$.
Hence,  the constraint~\eqref{eq:constr.am} simplifies to 
\begin{align*}
	\big[cf\big]_{-1}^{+1}=\sigma, \qquad c=c(u)=(u+1)+\frac{\beta_0}{\beta_1},
\end{align*}
the inhomogeneity on the right-hand side of~\eqref{eq:inhom.am} vanishes, and  $\ell\equiv0$ in the solution formula~\eqref{eq:inhom.f.gen.n}.
Thus, the preceding derivation (in \Cref{sssec:a=1/m}) shows that, if $n$ is affine, we even have exponential smallness of the remainder term, 
asymptotically as $\sqrt{\lambda}\to\infty$,
\begin{align}\label{eq:g1g2.asympt}
	b_1=c(1)+O(\ee^{-2\sqrt{\lambda}\alpha(1)}),\qquad b_2= -c(-1)+ O(\ee^{-2\sqrt{\lambda}\alpha(1)}).
\end{align}
The solution $f$ of~\eqref{eq:inhom.all.expl} is then given by
\begin{align*}
	f(u)=\frac{\vlo}{\sqrt{\lambda} }\ee^{-\alpha(1)\sqrt{\lambda}}\big(b_1\ee^{\alpha(u)\sqrt{\lambda}}+b_2\ee^{-\alpha(u)\sqrt{\lambda}}\big)
\end{align*} 
with $b$ as in~\eqref{eq:g1g2.asympt} and where $\vlo$ is determined by
\begin{align*}
	\frac{\vlo}{\sqrt{\lambda}}\big(c(1)[b_1+b_2\ee^{-2\alpha(1)\sqrt{\lambda}}]-c(-1)[b_1\ee^{-2\alpha(1)\sqrt{\lambda}}+b_2]\big)=\sigma.
\end{align*}
The identities~\eqref{eq:g1g2.asympt} imply that, as $\lambda\to\infty$,
\[\big(c(1)[b_1+b_2\ee^{-2\alpha(1)\sqrt{\lambda}}]-c(-1)[b_1\ee^{-2\alpha(1)\sqrt{\lambda}}+b_2]\big)\;= \;
c(1)^2+c(-1)^2
+O(\ee^{-2\sqrt{\lambda}\alpha(1)}). \]
Hence, letting $\eta:= \big(c(1)^2{+}c(-1)^2\big)^{-1}$, we find that 
\begin{align*}
\vlo(\lambda)=\sigma\eta\sqrt{\lambda} +\sigma\sqrt{\lambda}O(\ee^{-2\sqrt{\lambda}\alpha(1)}).
\end{align*}
 
 We conclude by summarising the main results established in the present section (\Cref{ssec:symbol}).
 \begin{proposition}\label{prop:spectral}
Assume hypotheses~\ref{hp:s.gen},~\ref{hp:s.a=1/m}, and let $a(u)m(u)=1$.
Then, the operator $\mathcal{G}_\Gamma=\mathcal{G}$ determined by the constrained elliptic equation in \Cref{prop:solveP.vn} takes the form
\begin{align*}
	\mathcal{G}_\Gamma=\sigma\eta\sqrt{-\Delta_\Gamma}+\sigma\mathcal{R}(\sqrt{-\Delta_\Gamma})
\end{align*}
with $\eta=\big((\tfrac{n(1)}{n'(1)})^2+(\tfrac{n(-1)}{n'(-1)})^2\big)^{-1}$,
where $\mathcal{R}(\sqrt{-\Delta_\Gamma})$ denotes a linear pseudo-differential operator of lower order with respect to $\sqrt{-\Delta_\Gamma}$. For linearly degenerate mobility, i.e.\ $i=1$ in~\ref{hp:e.coeff} (cf.~\ref{hp:s.gen}) the remainder $\mathcal{R}$ is of order at most $\frac13$.

If, in addition, the function $n(u)$ is affine (cf.~\eqref{eq:n.affine}), then  $\mathcal{R}(\sqrt{-\Delta_\Gamma})$ extends to a bounded linear operator on $L^2(\Gamma)$ 
with the property that $(\mathcal{R}(\sqrt{-\Delta_\Gamma})\ee_k,\ee_k)_{L^2(\Gamma)}$ decays to zero exponentially as $\sqrt{\lambda_k}\to\infty$.
 \end{proposition} 
Combining \Cref{prop:gs},~\Cref{l:G.diagonal}, and \Cref{prop:spectral} with the formal asymptotics in Section~\ref{sec:asymptotics} completes the justification of \Cref{ass:fsd}.

\subsection{Formal limit towards the intermediate surface diffusion flow}
\label{ssec:vanishing.slope}
In this section, we derive the assertion of \Cref{ass:fsd->isd}. 
To this end, let $\eps>0$ be a small parameter and consider the coupling function
\[n_\eps(u)=1+\eps u.\]
Further let $A_\eps^2\tau_\eps=n_\eps^4\tilde m(u)$, so that $a=\frac{\eps^{-2}}{m}$.

Our goal is to show that, as $\eps\downarrow0$, on compact subsets $\Lambda_R$ in frequency space with $R<\infty$ fixed but arbitrary, we quantitatively recover the intermediate surface diffusion law~\eqref{eq:isd.main} from the third-order versions in Proposition~\ref{prop:spectral}. To this end, it suffices to determine the leading-order asymptotic behaviour of $\zeta(\lambda)=\zeta_\eps(\lambda)$ for $\lambda\le R$ as $\eps\downarrow0$.

\paragraph{Coefficients of the second-order intermediate law.} Let us first identify the parameters $\om,\delta$ in~\eqref{eq:isd.main} for the present choice of coefficient functions in the limit $\eps\downarrow0$.
For $\eps=0$, the above choice of coefficients reduces to
$n\equiv1$ and
 $A^2(u)\tau(u)=\tilde m(u)=\frac{m(u)}{1-u^2}$, meaning that (cf.~\eqref{eq:sig-del-om.u})
\begin{align*}
\om=\int_{-1}^{+1}A(u)^2\tau(u)\sqrt{1-u^2}\,\dd u=\int_{-1}^{+1}\frac{m(u)}{\sqrt{1-u^2}}\,\dd u=\frac4\delta.
\end{align*}
Thus, in this case, 
 the propagation operator derived in \Cref{ass:isd} takes the form
\begin{align}\label{eq:Gdeldel}
\mathcal{G}_\Gamma=-\sigma\Delta_\Gamma(\delta\,\text{\normalfont Id}-\frac4\delta\Delta_\Gamma)^{-1},
\end{align}
corresponding to 
\begin{align}\label{eq:zetadeldel}
\zeta(\lambda)=\sigma\lambda\big(\delta+\frac4\delta\lambda\big)^{-1},\quad\lambda\in\Lambda.
\end{align}

\paragraph{Solution to third-order fractional laws for $\eps>0$.}
The functions $m$ and $\sfm$ are kept independent of $\eps$ and are even. Hence,
\[\alpha(\pm 1)=\pm\frac2\delta,\qquad \delta=4\big(\int_{-1}^1\sfm\,\dd u\big)^{-1}. \]
The solution $f=f_\lambda,\zeta=\zeta(\lambda)$ to~\eqref{eq:k.gen.abstract} with $\lambda=\lambda_k$ (see also~\eqref{eq:inhom.all.expl}) is obtained by replacing $\lambda$ by $(\eps\sqrt{\lambda})^2$ in the calculations of \Cref{sssec:a=1/m} and observing that $n''=0$. It takes the form
\begin{align}\label{eq:f.fsd->isd}
	f(u)=\Big(\frac{b_1}{\eps^{1/2}\lambda^{1/4}}f_+(u)\ee^{-\eps\sqrt{\lambda}\alpha(1)}+\frac{b_2}{\eps^{1/2}\lambda^{1/4}}f_-(u)\ee^{-\eps\sqrt{\lambda}\alpha(1)}\Big)\vlo
\end{align}
with $b_1,b_2,\zeta$ to be determined, where 
\begin{align*}
	f_+(u)=\frac{1}{\eps^{1/2}\lambda^{1/4}}\ee^{\eps\sqrt{\lambda}\alpha(u)},\qquad f_-(u)=\frac{1}{\eps^{1/2}\lambda^{1/4}}\ee^{-\eps\sqrt{\lambda}\alpha(u)}.
\end{align*}
We now proceed similarly as in \Cref{sssec:a=1/m} with the exception that here we need to compute all error terms explicitly up to the relevant order, since we are interested in quantitative results for $\lambda\le R$.

We first determine $b=(b_1,b_2)^T$. Notice that
\begin{align*}
	\sfa\pa_u f=\eps^{-2}\big( b_{1}\ee^{-\eps\sqrt{\lambda}(\alpha(1)-\alpha(u))}-b_{2}\ee^{-\eps\sqrt{\lambda}(\alpha(1)+\alpha(u))}\big)\vlo.
\end{align*}
Letting
\begin{align*}
	M_\eps:=\eps^{-2}\begin{pmatrix}
		1&-\ee^{-2\eps\sqrt{\lambda}\alpha(1)}
		\\ \ee^{-2\eps\sqrt{\lambda}\alpha(1)}&-1
	\end{pmatrix},
\end{align*}
we find that $b$ is determined by $M_\eps b=p$, where 
$p_1=c(1)=:c_+$, $p_2=c(-1)=:c_-$ with $c=\tfrac{n}{n'}$ (cf.\ \eqref{eq:def.c}).
Abbreviating $r:=\ee^{-2\eps\sqrt{\lambda}\alpha(1)}$, 
the inverse matrix $M_\eps^{-1}$ of $M_\eps$ can be written as
\begin{align*}
	M_\eps^{-1}:=
	\frac{\eps^2}{1-r^2}\begin{pmatrix}
		1&-r
		\\ r&-1
	\end{pmatrix}.
\end{align*}
Therefore 
\begin{align}\label{eq:b.fsd->isd}
b=M_\eps^{-1}p=\frac{\eps^2}{1-r^2}
\begin{pmatrix}
	c_+-rc_-\\rc_+-c_-
\end{pmatrix}.
\end{align}
We now impose the constraint $\big[\frac{n}{n'}f\big]_{-1}^{1}=\sigma$, which determines $\zeta$. In view of~\eqref{eq:f.fsd->isd}, it reads as
\begin{align*}
	\frac{\vlo}{\sqrt{\lambda}\eps}\big(c_+[b_1+b_2r]-c_-[b_1r+b_2]\big)=\sigma.
\end{align*}
Inserting the formula~\eqref{eq:b.fsd->isd} for $b$ and rearranging terms, this amounts to
\begin{align}\label{eq:zeta.to.isd.exact}
	&\frac{\zeta}{\sqrt{\lambda}}\frac{\eps}{1-r^2}\big((c_+^2+c_-^2)(1+r^2)-4c_+c_- r\big)=\sigma,
	\\&\text{with}\quad c_\pm=\eps^{-1}\pm 1,\quad r=\ee^{-2\alpha(1)\eps\sqrt{\lambda}}.
	\nonumber 
\end{align}
For consistency, observe that $c_+[b_1+b_2r]-c_-[b_1r+b_2]>0$ whenever $\lambda>0$.
To determine the dominant behaviour of $\zeta$ for $\lambda\le R$ and $0<\eps\ll_R1$, we abbreviate 
$\mu:=2\alpha(1)$ and note that
\begin{align*}
r&=\ee^{-\eps\sqrt{\lambda}\mu}=1-\eps\sqrt{\lambda}\mu+\frac12\eps^2\lambda\mu^2+O_R(\eps^3), \\r^2&=\ee^{-2\eps\sqrt{\lambda}\mu}=1-2\eps\sqrt{\lambda}\mu+2\eps^2\lambda\mu^2+O_R(\eps^3).
\end{align*}
Inserting these expansions as well as $c_\pm=\eps^{-1}\pm 1$ into~\eqref{eq:zeta.to.isd.exact}, and simplifying terms, we infer
\[\frac{\zeta}{2\lambda\mu}\big(8+2\mu^2\lambda+O_R(\eps)\big)=\sigma.\]
Observing that $\mu=\frac4\delta$, we thus arrive at the following quantitative formula, for $0<\eps\ll_R 1$ small,
\[\zeta(\lambda)=\sigma\lambda\Big(\delta+\frac4\delta\lambda+O_R(\eps)\Big)^{-1},\qquad \lambda\in\Lambda_R,\]
which reduces to~\eqref{eq:zetadeldel} as $\eps\downarrow0$.
Thus, for bounded frequencies, we recover in the limit $\eps\downarrow0$ the propagation operator~\eqref{eq:Gdeldel} associated to the intermediate surface diffusion flow with the same coefficients.

\section{Conclusion}

In this work, we have performed a formal asymptotic analysis of a degenerate Cahn-Hilliard model for viscoelastic phase separation, specifically deriving the sharp-interface asymptotics governing the evolution of the phase boundaries. Our analysis demonstrates that the cross-diffusive coupling between the order parameter and the bulk stress variable significantly alters the classical surface diffusion dynamics, leading to non-local geometric evolution laws.
We identified two distinct regimes determined by the nature of the coupling function and showed that these regimes can formally be connected by a singular limit. In the case of constant coupling, the interface evolves according to the intermediate surface diffusion flow.
More notably, for non-constant coupling functions that monotonically connect the phases, the effective dynamics are governed by a third-order geometric flow where the propagation operator exhibits the scaling behaviour of the fractional operator 
$\sqrt{-\Delta_\Gamma}$ at leading order.
A central result of our derivation is the characterisation of the normal velocity $\sf V$ as a Lagrange multiplier within a constrained elliptic problem. We have shown that this problem admits a rigorous variational formulation, which eventually allowed us to deduce a formal gradient-flow structure for the geometric evolution law
that reflects the gradient-flow structure of the original diffuse-interface system.
Future work will be aimed at providing rigorous underpinnings of our findings, including variational and analytical properties of the resulting geometric flows. For a toy model with constant mobility, one step of our asymptotics  concerning relaxation limits with positive interface thickness was recently made rigorous in~\cite{GH_2025}.

\appendix

\section{Differential geometry
}\label{app:diffgeo}

This appendix is a slight extension of~\cite[Appendix]{AGG_2012}, see also~\cite{PruessSimonett_2016}. It serves to determine  higher-order corrections in the geometric quantities and transformed differential operators.

\subsection{Geometric identities}\label{ssec:diffgeo.id}
The signed distance function $d=d(x)$ to the smooth, closed hypersurface $\Gamma\Subset \mathbb{R}^\dimx$ satisfies in a tubular neighbourhood of $\Gamma$ the identity (cf.~\cite[Chapter 2.3.2]{PruessSimonett_2016})
\begin{align*}
	\Delta d=\sum_{i=1}^{\dimx-1}\frac{-\kappa_{i}\circ\proj}{1-\kappa_i\circ\proj \,d},
\end{align*}
where $\{\kappa_i\}$ denote the principle curvatures of $\Gamma$ and $\proj$ the orthogonal projection onto $\Gamma$.

Taylor expanding the right-hand side, for small $|d|$, gives for $\kappa_i:=\kappa_i\circ\proj$
\begin{align*}
\sum_{i=1}^{\dimx-1}\frac{-\kappa_i}{1-\kappa_i d}
	&=-\sum_{i=1}^{\dimx-1}\kappa_i-\sum_{i=1}^{\dimx-1}\kappa_i^2d-\sum_{i=1}^{\dimx-1}\kappa_i^3d^2+O(|d|^3).
\end{align*}
Define
\begin{align*}
	\kappa_\Gamma:=\sum_{i=1}^{\dimx-1}\kappa_i,\quad
	k_2=\big(\sum_{i=1}^{\dimx-1}\kappa_i^2\big)^{\frac12},\quad 
	k_3=\big(\sum_{i=1}^{\dimx-1}\kappa_i^3\big)^{\frac13}.
\end{align*}
The quantity $\kappa_\Gamma$ is the mean curvature of $\Gamma$, and $k_2$ equals the Frobenius norm $|\mathcal{W}_\Gamma|=\big(\sum_{i=1}^{\dimx-1}\kappa_i^2\big)^{\frac12}$ of the Weingarten tensor $\mathcal{W}_\Gamma$.

In conclusion,
\begin{align*}
	\Delta d=-\kappa_\Gamma\circ\proj-|\mathcal{W}_\Gamma\circ\proj|^2d- k_3^3\circ\proj \,d^2+O(|d|^3).
\end{align*}

\subsection{Transformations}\label{ssec:diffgeo.transform}

For completeness, we briefly sketch the derivation of the well-established formulas used in the transformation of spatial differential operators to the new, rescaled variables introduced in \Cref{sec:prelim}. The presentation follows~\cite[Appendix]{AGG_2012} and uses notations from~\Cref{sec:prelim}.

Let $\gamma^{\ve}(s,\rho):=\gamma(s)+\ve\rho\nu(s)$, $\ve\in[0,1]$, and $\mathbb{G}^\ve=(g_{ij}^\ve)$, where
$g_{ij}^\ve=\pa_{i}\gamma^{\ve}\cdot\pa_{j}\gamma^{\ve}$. 
 Notice that $g_{ij}^\ve=g^\ve_{ji}$ for all $i,j\in\{1,\dots,\dimx\}$.
 Abbreviate $\dimx'=\dimx-1$.
  For all $i\in\{1,\dots, \dimx'\}$ it holds that  $g^\ve_{i\dimx}\equiv0$ due to $\pa_{s_i}\nu\cdot\nu=0$.
Thus, the matrix $\mathbb{G}^\ve$ and its inverse take the block diagonal form
\begin{align*}
	\mathbb{G}^\ve=
	\begin{pmatrix}
	\mathbb{G}^\ve_{\dimx'\times \dimx'}&0_{\dimx'}
		\\0_{\dimx'}^T&\ve^{2}
	\end{pmatrix},\qquad (\mathbb{G}^\ve)^{-1}=
	\begin{pmatrix}
		(\mathbb{G}^\ve_{\dimx'\times \dimx'})^{-1}&0_{\dimx'}
		\\0_{\dimx'}^T&\ve^{-2}
	\end{pmatrix},
\end{align*}
where $\dimx'=\dimx-1$.

\paragraph{Differential operators in new coordinates.}
Let $\rho=s_\dimx$ and $U=u\circ\gamma^\ve, \boldsymbol{J}=\boldsymbol{j}\circ\gamma^\ve$.
Then the differential operators in the reference coordinates determined by the parametrisation $\gamma^\ve$ are given by 
\begin{align*}
\nabla_x u\circ \gamma^\ve=\sum_{i,j=1}^\dimx (g^\ve)^{ij}\pa_{s_i}U\pa_{s_j}\gamma^\ve
&=\sum_{i,j=1}^{\dimx-1} (g^\ve)^{ij}\pa_{s_i}U\pa_{s_j}\gamma^\ve+\ve^{-1}\pa_\rho U\,\nu
\\&=\nablaGeps U+\ve^{-1}\pa_\rho U\,\nu,
\end{align*}
\begin{align*}
	\divv_x \boldsymbol{j}\circ \gamma^\ve=\sum_{i,j=1}^\dimx (g^\ve)^{ij}\pa_{s_j}\gamma^\ve\cdot\pa_{s_i}\boldsymbol{J}
	=\divGeps \boldsymbol{J}+\ve^{-1}\pa_\rho \boldsymbol{J}\cdot \nu.
\end{align*}
Combined with basic geometric identities, the above formulas imply~\eqref{eq:diffop.traf.space}. 

\paragraph{Expansions.}
Let $g_{ij}:=g_{ij}^0$. Then
\begin{align*}
g_{ij}^\ve&=g_{ij}+\ve\rho(\pa_{s_i}\nu\cdot\pa_{s_j}\gamma+\pa_{s_j}\nu\cdot\pa_{s_i}\gamma)+\ve^2\rho^2\pa_{s_i}\nu\cdot\pa_{s_j}\nu
\\&=g_{ij}+d r_{ij}^{(1)}+d^2r_{ij}^{(2)}\quad\text{with}\; d:=d(\gamma^\ve(s,\rho))=\ve\rho,
\end{align*}
where the coefficients $r_{ij}^{(l)}$ only depend on $\gamma=\gamma(s)$.
Hence, for suitable $(\tilde r^{(l)})^{ij},l=1,2,$ that only depend on $\gamma$,
\begin{align*}
	(g^{ij})^\ve&=g^{ij}+d(\tilde r^{(1)})^{ij}+d^2(\tilde r^{(2)})^{ij} +O(|d|^3),\quad d:=\ve\rho.
\end{align*}
It follows that 
\begin{align*}
\nablaGeps U&=\nablaG U
+d\sum_{i,j=1}^{\dimx-1} 
\big(g^{ij}\pa_{s_j}\nu+(\tilde r^{(1)})^{ij}\pa_{s_j}\gamma\big)\pa_{s_i}U
+O(|d|^2)
\\&=\nablaG U
+d\sum_{i=1}^{\dimx-1} \mrem^i
\pa_{s_i}U+O(|d|^2),
\end{align*}
where $\mrem^i:=\sum_{j=1}^{\dimx-1}\big(g^{ij}\pa_{s_j}\nu+ (\tilde r^{(1)})^{ij}\pa_{s_j}\gamma\big)$ is tangential, i.e.\ \[\nu\cdot\mrem^i\equiv0,\quad i=1,\dots,\dimx-1.\]
Likewise, we obtain 
\begin{align*}
	\divGeps \boldsymbol{J}&=\divG \boldsymbol{J}
	+d\sum_{i,j=1}^{\dimx-1} 
	\big(g^{ij}\pa_{s_j}\nu+(\tilde r^{(1)})^{ij}\pa_{s_j}\gamma\big)\cdot\pa_{s_i}\boldsymbol{J}
	+O(|d|^2)
	\\&=\divG \boldsymbol{J}
	+d\sum_{i=1}^{\dimx-1} \mrem^i\cdot\pa_{s_i}\boldsymbol{J}+O(|d|^2),
\end{align*}
where throughout $d:=\ve\rho$.

\printbibliography[heading=bibintoc]

@article {Onsager_1931_recipr_I,
	AUTHOR = {Onsager, Lars},
	TITLE = {Reciprocal relations in irreversible processes {I}},
	JOURNAL = {Physical Review},
	VOLUME = {37},
	YEAR = {1931},
	NUMBER = {},
	PAGES = {405--426},
}

@article {Mielke_2013_bulk-interface,
	AUTHOR = {Mielke, Alexander},
	TITLE = {Thermomechanical modeling of energy-reaction-diffusion
	systems, including bulk-interface interactions},
	JOURNAL = {Discrete Contin. Dyn. Syst. Ser. S},
	FJOURNAL = {Discrete and Continuous Dynamical Systems. Series S},
	VOLUME = {6},
	YEAR = {2013},
	NUMBER = {2},
	PAGES = {479--499},
	ISSN = {1937-1632,1937-1179},
	MRCLASS = {80A99 (35K57 82B35)},
	MRNUMBER = {2997555},
	DOI = {10.3934/dcdss.2013.6.479},
	URL = {https://doi.org/10.3934/dcdss.2013.6.479},
}

@article {Aleksandrov_1956,
	AUTHOR = {Aleksandrov, A. D.},
	TITLE = {Uniqueness theorems for surfaces in the large.},
	JOURNAL = {Vestnik Leningrad. Univ.},
	FJOURNAL = {Vestnik Leningrad. Univ.},
	VOLUME = {11},
	YEAR = {1956},
	NUMBER = {19},
	PAGES = {5--17},
	MRCLASS = {53.0X},
	MRNUMBER = {86338},
	MRREVIEWER = {H.\ Busemann},
}

@article {EG_1996,
	AUTHOR = {Elliott, Charles M. and Garcke, Harald},
	TITLE = {On the {C}ahn-{H}illiard equation with degenerate mobility},
	JOURNAL = {SIAM J. Math. Anal.},
	FJOURNAL = {SIAM Journal on Mathematical Analysis},
	VOLUME = {27},
	YEAR = {1996},
	NUMBER = {2},
	PAGES = {404--423},
	ISSN = {0036-1410},
	MRCLASS = {35K55 (35Q99 82C26)},
	MRNUMBER = {1377481},
	MRREVIEWER = {Jing\ Xue\ Yin},
	DOI = {10.1137/S0036141094267662},
	URL = {https://doi.org/10.1137/S0036141094267662},
}

@article {GSK_2008,
	AUTHOR = {Gugenberger, Clemens and Spatschek, Robert and Kassner, Klaus},
	TITLE = {Comparison of phase-field models for surface diffusion},
	JOURNAL = {Phys. Rev. E (3)},
	FJOURNAL = {Physical Review E. Statistical, Nonlinear, and Soft Matter
	Physics},
	VOLUME = {78},
	YEAR = {2008},
	NUMBER = {1},
	PAGES = {016703, 17},
	ISSN = {1539-3755,1550-2376},
	MRCLASS = {82C24},
	MRNUMBER = {2496195},
	DOI = {10.1103/PhysRevE.78.016703},
	URL = {https://doi.org/10.1103/PhysRevE.78.016703},
}

@article {EG_1997_surface_motion,
	AUTHOR = {Elliott, Charles M. and Garcke, Harald},
	TITLE = {Existence results for diffusive surface motion laws},
	JOURNAL = {Adv. Math. Sci. Appl.},
	FJOURNAL = {Advances in Mathematical Sciences and Applications},
	VOLUME = {7},
	YEAR = {1997},
	NUMBER = {1},
	PAGES = {467--490},
	ISSN = {1343-4373},
	MRCLASS = {58E12 (35K99 82B24 82B26)},
	MRNUMBER = {1454678},
	MRREVIEWER = {Jing\ Xue\ Yin},
}

@article {ABC_1994,
	AUTHOR = {Alikakos, Nicholas D. and Bates, Peter W. and Chen, Xinfu},
	TITLE = {Convergence of the {C}ahn-{H}illiard equation to the
	{H}ele-{S}haw model},
	JOURNAL = {Arch. Rational Mech. Anal.},
	FJOURNAL = {Archive for Rational Mechanics and Analysis},
	VOLUME = {128},
	YEAR = {1994},
	NUMBER = {2},
	PAGES = {165--205},
	ISSN = {0003-9527},
	MRCLASS = {35Q99 (35B25 35C20 35K55 76D10 80A22)},
	MRNUMBER = {1308851},
	MRREVIEWER = {Peter\ J.\ Sternberg},
	DOI = {10.1007/BF00375025},
	URL = {https://doi.org/10.1007/BF00375025},
}

@article {BMST_2022_M3AS_sd,
	AUTHOR = {Bretin, Elie and Masnou, Simon and Sengers, Arnaud and Terii,
	Garry},
	TITLE = {Approximation of surface diffusion flow: a second-order
	variational {C}ahn-{H}illiard model with degenerate
	mobilities},
	JOURNAL = {Math. Models Methods Appl. Sci.},
	FJOURNAL = {Mathematical Models and Methods in Applied Sciences},
	VOLUME = {32},
	YEAR = {2022},
	NUMBER = {4},
	PAGES = {793--829},
	ISSN = {0218-2025,1793-6314},
	MRCLASS = {74N20 (35A15 35A35 53E10 53E40 65M32)},
	MRNUMBER = {4421217},
	DOI = {10.1142/S0218202522500178},
	URL = {https://doi.org/10.1142/S0218202522500178},
}

@article{CahnTaylor_1994_surfaceDiffusion,
	title={Surface motion by surface diffusion},
	author={Cahn, John W and Taylor, Jean E},
	journal={Acta Met. Mat.},
	volume={42},
	number={4},
	pages={1045--1063},
	year={1994},
	publisher={Elsevier}
}

@article{TaylorCahn_1994_anisotropic,
	AUTHOR = {Taylor, Jean E. and Cahn, John W.},
	TITLE = {Linking anisotropic sharp and diffuse surface motion laws via
	gradient flows},
	JOURNAL = {J. Statist. Phys.},
	FJOURNAL = {Journal of Statistical Physics},
	VOLUME = {77},
	YEAR = {1994},
	NUMBER = {1-2},
	PAGES = {183--197},
	ISSN = {0022-4715,1572-9613},
	MRCLASS = {58E12 (49Q05 53C42)},
	MRNUMBER = {1300532},
	MRREVIEWER = {Nathan\ Smale},
	DOI = {10.1007/BF02186838},
	URL = {https://doi.org/10.1007/BF02186838},
}

@article {BL_2022_vps_degmob,
	AUTHOR = {Brunk, Aaron and Luk\'a\v{c}ov\'a-Medvid'ov{\'a}, M\'aria},
	TITLE = {Global existence of weak solutions to viscoelastic phase
	separation: part {II}. {D}egenerate case},
	JOURNAL = {Nonlinearity},
	FJOURNAL = {Nonlinearity},
	VOLUME = {35},
	YEAR = {2022},
	NUMBER = {7},
	PAGES = {3459--3486},
	ISSN = {0951-7715,1361-6544},
	MRCLASS = {35D30 (35A01 35G50 76A05 76A10)},
	MRNUMBER = {4443942},
	DOI = {10.1088/1361-6544/ac591e},
	URL = {https://doi.org/10.1088/1361-6544/ac591e},
}

@article {BE_1991,
	AUTHOR = {Blowey, J. F. and Elliott, C. M.},
	TITLE = {The {C}ahn-{H}illiard gradient theory for phase separation
	with nonsmooth free energy. {I}. {M}athematical analysis},
	JOURNAL = {European J. Appl. Math.},
	FJOURNAL = {European Journal of Applied Mathematics},
	VOLUME = {2},
	YEAR = {1991},
	NUMBER = {3},
	PAGES = {233--280},
	ISSN = {0956-7925,1469-4425},
	MRCLASS = {35D10 (35B40 80A22)},
	MRNUMBER = {1123143},
	MRREVIEWER = {Bishun\ Pandey},
	DOI = {10.1017/S095679250000053X},
	URL = {https://doi.org/10.1017/S095679250000053X},
}

@misc{GH_2025,
	title={Well-posedness and relaxation in a simplified model for viscoelastic phase separation via Hilbertian gradient flows}, 
	author={Moritz Immanuel Gau and Katharina Hopf},
	year={2025},
	eprint={2508.10722},
	archivePrefix={arXiv},
	primaryClass={math.AP},
	url={https://arxiv.org/abs/2508.10722}, 
}

@incollection {CNH_2024,
	AUTHOR = {Carlen, Eric A. and Novick-Cohen, Amy and Hari, Lydia Peres},
	TITLE = {Connecting the deep quench obstacle problem with surface
	diffusion via their steady states},
	BOOKTITLE = {From particle systems to partial differential equations},
	SERIES = {Springer Proc. Math. Stat.},
	VOLUME = {465},
	PAGES = {239--267},
	PUBLISHER = {Springer, Cham},
	YEAR = {2024},
	MRCLASS = {35K20 (35R70)},
	MRNUMBER = {4807496},
	MRREVIEWER = {Jozil\ O.\ Takhirov},
	DOI = {10.1007/978-3-031-65195-3\_11},
	URL = {https://doi.org/10.1007/978-3-031-65195-3_11},
}

@book {GT_2001,
	AUTHOR = {Gilbarg, David and Trudinger, Neil S.},
	TITLE = {Elliptic partial differential equations of second order},
	SERIES = {Classics in Mathematics},
	PUBLISHER = {Springer-Verlag, Berlin},
	YEAR = {2001},
	PAGES = {xiv+517},
	ISBN = {3-540-41160-7},
	MRCLASS = {35-02 (35Jxx)},
	MRNUMBER = {1814364},
}

@article {LM_1989,
	AUTHOR = {Luckhaus, Stephan and Modica, Luciano},
	TITLE = {The {G}ibbs-{T}hompson relation within the gradient theory of
	phase transitions},
	JOURNAL = {Arch. Rational Mech. Anal.},
	FJOURNAL = {Archive for Rational Mechanics and Analysis},
	VOLUME = {107},
	YEAR = {1989},
	NUMBER = {1},
	PAGES = {71--83},
	ISSN = {0003-9527},
	MRCLASS = {49F22 (35J85 49D29 73B30 82A25)},
	MRNUMBER = {1000224},
	MRREVIEWER = {P.\ W.\ Bates},
	DOI = {10.1007/BF00251427},
	URL = {https://doi.org/10.1007/BF00251427},
}

@incollection {Ambrosio_2000,
	AUTHOR = {Ambrosio, L.},
	TITLE = {Geometric evolution problems, distance function and viscosity
	solutions},
	BOOKTITLE = {Calculus of variations and partial differential equations
	({P}isa, 1996)},
	PAGES = {5--93},
	PUBLISHER = {Springer, Berlin},
	YEAR = {2000},
	ISBN = {3-540-64803-8},
	MRCLASS = {35K55 (35D05 49L25)},
	MRNUMBER = {1757696},
}

@book {Zeidler_Variational,
	AUTHOR = {Zeidler, Eberhard},
	TITLE = {Nonlinear functional analysis and its applications. {III}},
	NOTE = {Variational methods and optimization},
	PUBLISHER = {Springer-Verlag, New York},
	YEAR = {1985},
	PAGES = {xxii+662},
	ISBN = {0-387-90915-X},
	MRCLASS = {49-02 (46G99 47Hxx 58-02 90Cxx)},
	MRNUMBER = {768749},
	MRREVIEWER = {Jean\ Mawhin},
	DOI = {10.1007/978-1-4612-5020-3},
	URL = {https://doi.org/10.1007/978-1-4612-5020-3},
}

@Inbook{ES_1999_abstractParabolic,
	author="Escher, Joachim
	and Simonett, Gieri",
	title={Moving surfaces and abstract parabolic evolution equations},
	bookTitle="Topics in Nonlinear Analysis: The Herbert Amann Anniversary Volume",
	year="1999",
	publisher="Birkh{\"a}user, Basel",
	pages="183--212",
	isbn="978-3-0348-8765-6",
	doi="10.1007/978-3-0348-8765-6_10",
	url="https://doi.org/10.1007/978-3-0348-8765-6_10"
}

@article {EMS_1998_surface-diffusion,
	AUTHOR = {Escher, Joachim and Mayer, Uwe F. and Simonett, Gieri},
	TITLE = {The surface diffusion flow for immersed hypersurfaces},
	JOURNAL = {SIAM J. Math. Anal.},
	FJOURNAL = {SIAM Journal on Mathematical Analysis},
	VOLUME = {29},
	YEAR = {1998},
	NUMBER = {6},
	PAGES = {1419--1433},
	ISSN = {0036-1410,1095-7154},
	MRCLASS = {58E15 (35K99 35R35 65C20)},
	MRNUMBER = {1638074},
	MRREVIEWER = {Vladimir\ Grebenev},
	DOI = {10.1137/S0036141097320675},
	URL = {https://doi.org/10.1137/S0036141097320675},
}

@article {EGI_2002_limiting_isd,
	AUTHOR = {Escher, Joachim and Giga, Yoshikazu and Ito, Kazuo},
	TITLE = {On a limiting motion and self-intersections for the
	intermediate surface diffusion flow},
	JOURNAL = {J. Evol. Equ.},
	FJOURNAL = {Journal of Evolution Equations},
	VOLUME = {2},
	YEAR = {2002},
	NUMBER = {3},
	PAGES = {349--364},
	ISSN = {1424-3199,1424-3202},
	MRCLASS = {53C44 (35K55)},
	MRNUMBER = {1930611},
	MRREVIEWER = {Vladimir\ Grebenev},
	DOI = {10.1007/s00028-002-8092-z},
	URL = {https://doi.org/10.1007/s00028-002-8092-z},
}

@inproceedings {EGI_2001_limiting_isd_curves,
	AUTHOR = {Escher, J. and Giga, Y. and Ito, K.},
	TITLE = {On a limiting motion and self-intersections of curves moved by
	the intermediate surface diffusion flow},
	BOOKTITLE = {Proceedings of the {T}hird {W}orld {C}ongress of {N}onlinear
	{A}nalysts, {P}art 6 ({C}atania, 2000)},
	JOURNAL = {Nonlinear Anal.},
	FJOURNAL = {Nonlinear Analysis. Theory, Methods \& Applications. An
	International Multidisciplinary Journal},
	VOLUME = {47},
	YEAR = {2001},
	NUMBER = {6},
	PAGES = {3717--3728},
	ISSN = {0362-546X,1873-5215},
	MRCLASS = {35K57 (74N25)},
	MRNUMBER = {1972315},
	DOI = {10.1016/S0362-546X(01)00491-6},
	URL = {https://doi.org/10.1016/S0362-546X(01)00491-6},
}

@incollection {EscherIto_2004_intermediate,
	AUTHOR = {Escher, Joachim and Ito, Kazuo},
	TITLE = {On the intermediate surface diffusion flow},
	BOOKTITLE = {Free boundary problems ({T}rento, 2002)},
	SERIES = {Internat. Ser. Numer. Math.},
	VOLUME = {147},
	PAGES = {131--138},
	PUBLISHER = {Birkh\"{a}user, Basel},
	YEAR = {2004},
	ISBN = {3-7643-2193-8},
	MRCLASS = {35R35 (35K55 53C44)},
	MRNUMBER = {2044569},
}

@book {PruessSimonett_2016,
	AUTHOR = {Pr\"{u}ss, Jan and Simonett, Gieri},
	TITLE = {Moving interfaces and quasilinear parabolic evolution
	equations},
	SERIES = {Monographs in Mathematics},
	VOLUME = {105},
	PUBLISHER = {Birkh\"{a}user/Springer},
	YEAR = {2016},
	PAGES = {xix+609},
	MRCLASS = {35-02 (35B30 35K93 35R35 47F05 58Jxx 76A15 80A22)},
	MRNUMBER = {3524106},
	MRREVIEWER = {Glen\ E.\ Wheeler},
	DOI = {10.1007/978-3-319-27698-4},
	URL = {https://doi.org/10.1007/978-3-319-27698-4},
}

@article {Pego_1989,
	AUTHOR = {Pego, R. L.},
	TITLE = {Front migration in the nonlinear {C}ahn-{H}illiard equation},
	JOURNAL = {Proc. Roy. Soc. London Ser. A},
	FJOURNAL = {Proceedings of the Royal Society. London. Series A.
	Mathematical, Physical and Engineering Sciences},
	VOLUME = {422},
	YEAR = {1989},
	NUMBER = {1863},
	PAGES = {261--278},
	ISSN = {0962-8444,2053-9169},
	MRCLASS = {80A20 (35Q20)},
	MRNUMBER = {997638},
}

@article {AGG_2012,
	AUTHOR = {Abels, Helmut and Garcke, Harald and Gr\"{u}n, G\"{u}nther},
	TITLE = {Thermodynamically consistent, frame indifferent diffuse
	interface models for incompressible two-phase flows with
	different densities},
	JOURNAL = {Math. Models Methods Appl. Sci.},
	FJOURNAL = {Mathematical Models and Methods in Applied Sciences},
	VOLUME = {22},
	YEAR = {2012},
	NUMBER = {3},
	PAGES = {1150013, 40},
	ISSN = {0218-2025,1793-6314},
	MRCLASS = {35R35 (35C20 35Q35 76D05 76D45 76T99 80A22)},
	MRNUMBER = {2890451},
	MRREVIEWER = {Eugen\ Varvaruca},
	DOI = {10.1142/S0218202511500138},
	URL = {https://doi.org/10.1142/S0218202511500138},
}

@article {CEN_1996,
	AUTHOR = {Cahn, J. W. and Elliott, C. M. and Novick-Cohen, A.},
	TITLE = {The {C}ahn-{H}illiard equation with a concentration dependent
	mobility: motion by minus the {L}aplacian of the mean
	curvature},
	JOURNAL = {European J. Appl. Math.},
	FJOURNAL = {European Journal of Applied Mathematics},
	VOLUME = {7},
	YEAR = {1996},
	NUMBER = {3},
	PAGES = {287--301},
	ISSN = {0956-7925,1469-4425},
	MRCLASS = {80A22 (35Q99)},
	MRNUMBER = {1401172},
	MRREVIEWER = {J.\ R.\ Ockendon},
	DOI = {10.1017/S0956792500002369},
	URL = {https://doi.org/10.1017/S0956792500002369},
}

@article{tanaka2000viscoelastic,
	title={Viscoelastic phase separation},
	author={Tanaka, Hajime},
	journal={Journal of Physics: Condensed Matter},
	volume={12},
	number={15},
	pages={R207},
	year={2000},
	publisher={IOP Publishing}
}

@article{TO_1996_network,
	title = {Network Domain Structure in Viscoelastic Phase Separation},
	author = {Taniguchi, Takashi and Onuki, Akira},
	journal = {Phys. Rev. Lett.},
	volume = {77},
	issue = {24},
	pages = {4910--4913},
	numpages = {0},
	year = {1996},
	month = {12},
	publisher = {American Physical Society},
	doi = {10.1103/PhysRevLett.77.4910},
	url = {https://link.aps.org/doi/10.1103/PhysRevLett.77.4910}
}

@article{doi1992dynamic,
	title={Dynamic coupling between stress and composition in polymer solutions and blends},
	author={Doi, Masao and Onuki, Akira},
	journal={Journal de Physique II},
	volume={2},
	number={8},
	pages={1631--1656},
	year={1992},
	publisher={EDP Sciences}
}

@article{STDL_2019_energy-stable,
	AUTHOR = {Strasser, Paul J. and Tierra, Giordano and D\"unweg, Burkhard
	and Luk\'a\v{c}ov\'a-Medvid'ov{\'a}, M\'aria},
	TITLE = {Energy-stable linear schemes for polymer-solvent phase field
	models},
	JOURNAL = {Comput. Math. Appl.},
	FJOURNAL = {Computers \& Mathematics with Applications. An International
	Journal},
	VOLUME = {77},
	YEAR = {2019},
	NUMBER = {1},
	PAGES = {125--143},
	ISSN = {0898-1221,1873-7668},
	MRCLASS = {65M60 (65M12 76A10 82C26)},
	MRNUMBER = {3907406},
	MRREVIEWER = {Georgios\ D.\ Akrivis},
	DOI = {10.1016/j.camwa.2018.09.018},
	URL = {https://doi.org/10.1016/j.camwa.2018.09.018},
}

@article{ZZE_2006,
	title = {Modified Models of Polymer Phase Separation},
	author = {Zhou, Douglas and Zhang, Pingwen and E, Weinan},
	year = {2006},
	journal = {Phys. Rev. E},
	volume = {73},
	number = {6},
	pages = {061801},
	publisher = {{American Physical Society}},
	doi = {10.1103/PhysRevE.73.061801}
}

@article{Tanaka_2022_vps,
	title={Viscoelastic phase separation in biological cells},
	author={Tanaka, Hajime},
	journal={Communications Physics},
	volume={5},
	number={1},
	pages={167},
	year={2022},
	publisher={Nature Publishing Group UK London}
}

@article{LMS_2016_sharp,
	AUTHOR = {Lee, Alpha Albert and M\"unch, Andreas and S\"uli, Endre},
	TITLE = {Sharp-interface limits of the {C}ahn-{H}illiard equation with
	degenerate mobility},
	JOURNAL = {SIAM J. Appl. Math.},
	FJOURNAL = {SIAM Journal on Applied Mathematics},
	VOLUME = {76},
	YEAR = {2016},
	NUMBER = {2},
	PAGES = {433--456},
	ISSN = {0036-1399,1095-712X},
	MRCLASS = {35K35 (35B25 35K59 74N20 80A22 82C26)},
	MRNUMBER = {3466205},
	MRREVIEWER = {Christopher\ P.\ Grant},
	DOI = {10.1137/140960189},
	URL = {https://doi.org/10.1137/140960189},
}

\end{document}